\documentclass{amsart}
\usepackage{amsmath,amsthm,enumerate,mathrsfs,graphicx,makecell,amssymb,booktabs,color}

\newtheorem{theorem}{Theorem}
\newtheorem{proposition}{Proposition}[section]
\newtheorem{corollary}[proposition]{Corollary}
\newtheorem{lemma}[proposition]{Lemma}

\newtheorem{conjecture}{Conjecture}

\newtheorem{question}{Question}

\DeclareMathOperator{\spann}{span}

\DeclareMathOperator{\rank}{rank}

\DeclareMathOperator{\prox}{Prox}

\DeclareMathOperator{\GL}{GL}
\DeclareMathOperator{\Gr}{Gr}
\DeclareMathOperator{\ned}{End}
\newcommand{\iii}{\mathtt{i}}
\newcommand{\jjj}{\mathtt{j}}
\newcommand{\kkk}{\mathtt{k}}
\newcommand{\ellellell}{\mathtt{l}}

\title[The Bernoulli property for matrix equilibrium states]{Totally ergodic generalised matrix equilibrium states have the Bernoulli property}

\begin{document}
\author{Ian D. Morris}
\begin{abstract}
We show that every totally ergodic generalised matrix equilibrium state is $\psi$-mixing with respect to the natural partition into cylinders and hence is measurably isomorphic to a Bernoulli shift in its natural extension. This implies that the natural extensions of ergodic generalised matrix equilibrium states are measurably isomorphic to Bernoulli processes extended by finite rotations. This resolves a question of Gatzouras and Peres in the special case of self-affine repelling sets with generic translations.

 \end{abstract}
\maketitle


\section{Background and motivation}\label{se:one}

Given a dynamical system $f \colon M \to M$ defined on a manifold $M$ it is a matter of fundamental interest to be able to describe the behaviour of typical trajectories. In practice ``typical'' is usually understood in measure-theoretic terms, leading us to ask what happens to trajectories whose starting point belongs to a set of full, or at least positive, Lebesgue measure. There are however situations in which this is insufficient: for example, a dynamical system may admit a repelling invariant set such that Lebesgue almost every point in an open neighbourhood of the invariant set eventually leaves that open set never to return; but it may still be of interest to understand which behaviours are typical among those points whose trajectories remain on the repelling set at all future times. Since the repelling set itself will usually have zero Lebesgue measure, the most natural way to understand this question is arguably to look for invariant measures supported on the repelling set with the largest possible dimension, which in this article will always be taken to mean Hausdorff dimension. This raises obvious fundamental questions: do such measures exist, are they unique, and what are their ergodic properties? Questions and conjectures in this direction have been raised on a number of occasions by various authors (see for example \cite{ChPe10,KePe96,LaGa92,ScWe01}). Among these, the following conjecture of D. Gatzouras and Y. Peres is typical:
\begin{conjecture}[\cite{GaPe97}]\label{qu:eensparkrangers}
Let $f \colon M \to M$ be an expanding map and $K\subseteq M$ a compact invariant set which satisfies specification. Then $K$ supports a unique ergodic $f$-invariant measure with the same Hausdorff dimension as $K$. This measure is mixing for $f$ and, perhaps, its natural extension is measurably isomorphic to a Bernoulli shift. 
\end{conjecture}
Here an invariant set is said to satisfy \emph{specification} if it admits a Markov partition which satisfies a certain quantitative topological mixing property; for details we refer the reader to \cite{GaPe97}. This property is in particular satisfied if the dynamical system $f \colon K \to K$ is topologically conjugate to the full shift on finitely many symbols, which will be the case for all of the examples considered in this article. 

Remarkably, the conjecture of Gatzouras and Peres has been answered negatively in every particular. It has been shown that the measure of maximal dimension can fail to exist, and even that the supremum of the dimensions of invariant measures can fall short of the dimension of $K$ itself (see \cite{DaSi17}); it has been shown that the measure of maximal dimension can exist but fail to be unique (see \cite{BaFe11,MoSe19b}); and it has been shown that the measure of maximal dimension can exist but fail to be totally ergodic (and in particular fail to be mixing), a result which is implied by previous work of the author \cite[\S2]{Mo18a}. 

The fundamental difficulty of Conjecture \ref{qu:eensparkrangers} is as follows. It has long been known that if $f \colon M \to M$ is an expanding map then its absolutely continuous invariant measures can be characterised as the invariant measures $\mu$ which maximise the quantity $h(\mu)-\int \log |\det D_xf|\,d\mu(x)$, where $h(\mu)$ denotes the entropy of $\mu$ with respect to $f$; we call such measures \emph{equilibrium states} of the function $x \mapsto -\log |\det D_xf|$, which we refer to as a \emph{potential}. This definition can be alternatively presented by saying that the absolutely continuous invariant measure $\mu$ maximises the entropy minus the total of the $d:= \dim M$ different Lyapunov exponents of $f$ with respect to $\mu$. When the dimension of the set $K$ in Conjecture \ref{qu:eensparkrangers} is instead equal to $s \in (0,d)$, the measure of maximal dimension is believed to be typically characterised by the property of maximising the entropy minus the sum of the $\lfloor s\rfloor$ least expanding Lyapunov exponents, minus $(s-\lfloor s\rfloor)$ times the next least expanding Lyapunov exponent. If all of the Lyapunov exponents are equal then this sum of weighted Lyapunov exponents is simply $(s/d)$ times the logarithm of the Jacobian, and the potential can then be realised as a continuous real-valued function. This makes the classical thermodynamic formalism of Bowen, Ruelle and Sinai, which applies to H\"older continuous real-valued potentials, applicable to the problem. For this reason Conjecture \ref{qu:eensparkrangers} has long been satisfactorily understood in the special case of repelling sets of \emph{conformal} expanding maps in which all Lyapunov exponents of a given invariant measure are guaranteed to be equal. Outside this special case the problem becomes far more difficult since we are obliged to consider equilibrium states of a potential which is defined in terms of several distinct Lyapunov exponents and cannot be reduced to the classical thermodynamic formalism of continuous potentials such as $x \mapsto -\log |\det D_xf|$.  To understand the candidate measures of maximal dimension in this case it seems to be necessary to develop a ``non-commutative'' thermodynamic formalism capable of dealing with Lyapunov exponents in place of the ergodic average of a function, in which averages of ergodic sums are replaced with averages of subadditive functions given by the logarithms of the norms of certain linear cocycles. This project has seen substantial progress in the last few years (see e.g. \cite{BoMo18,FeKa11,FeSh14,KaMo18,MoSe19a,Pa19,Pi20}) and this article is concerned with the description in detail of the equilibrium states which arise in this thermodynamic formalism in the case of locally constant cocycles over the full shift. 

To address the full generality of Conjecture \ref{qu:eensparkrangers} would appear to require a theory of equilibrium states which allowed the consideration of arbitrary differentiable (or perhaps just H\"older continuous) linear cocycles defined over repelling sets. Such a theory is significantly beyond the range of current techniques, and so far the development of this thermodynamic formalism has focused principally on the simplest nontrivial context, namely the equilibrium states of locally constant linear cocycles over full symbolic shifts. This is precisely the thermodynamic formalism needed to understand the (candidate) invariant measures of maximum dimension for \emph{self-affine sets}, a class of fractal objects of independent interest which (under certain assumptions) correspond to the case of Conjecture \ref{qu:eensparkrangers} in which $M=\mathbb{R}^d$ and in which $D_xf \in \GL_d(\mathbb{R})$ takes only finitely many values when $x$ belongs to the invariant set $K$. In the present work we completely describe the qualitative mixing properties of equilibrium states of linear cocycles of this type: we will show that every ergodic \emph{generalised matrix equilibrium state} has the property that its natural extension is measurably isomorphic to the product of a Bernoulli process and a rotation of a finite set. In particular the natural extension of every totally ergodic generalised matrix equilibrium state is measurably isomorphic to a Bernoulli process. This completely resolves that part of Conjecture \ref{qu:eensparkrangers} which is concerned with mixing and the Bernoulli property in the special case where $K$ is a self-affine set which  is already known to support an invariant measure whose dimension is equal to a theoretical maximum value defined by Falconer in \cite{Fa88}. This property is known to hold for self-affine sets which are ``typical'' in certain precise senses (see \cite[Theorem 4]{JoPoSi07} and \cite[Theorem 1.9]{Fe19}).

This motivates us to ask the following speculative question:
\begin{question}\label{qu:ack}
Let $f \colon M \to M$ be a $C^2$ expanding map and $K\subseteq M$ a compact invariant set which satisfies specification and supports a unique ergodic $f$-invariant measure with the same Hausdorff dimension as $K$. Is the natural extension of this measure measurably isomorphic to the product of a Bernoulli measure with a rotation on a finite set?
\end{question}
For self-affine repelling sets which support a measure of dimension equal to the theoretical maximum defined by Falconer, the results in this article suffice to answer Question \ref{qu:ack} affirmatively. However, the full range of possible behaviours outside this class of repellers is far from being fully understood even in the self-affine case, and it is far from clear whether or not further pathological special cases will be discovered. Beyond the self-affine class we anticipate that it should not be profoundly difficult to extend our methods and results to the case of typical repellers which satisfy a \emph{fibre-bunching condition} on the derivative cocycle $(x,n) \mapsto D_xf^n$ as in \cite{BuPa20,Fa94,Pa19}, particularly if a strong additional assumption is used such as the ``pinching and twisting'' conditions introduced by Bonatti and Viana in \cite{BoVi04}. The removal of the fibre-bunching condition seems in our opinion likely to be a more substantial obstacle to further developments of these ideas.

\section{Generalised matrix equilibrium states}

\subsection{Fundamental definitions and notation} 

The class of measures which we investigate in this article, which we call generalised matrix equilibrium states, are defined on abstract symbolic spaces and can be related to self-affine sets via a coding procedure which is described later in this section. In order to describe these objects we require some fundamental definitions. For each $N \geq 2$ let us define $\Sigma_N:=\{1,\ldots,N\}^{\mathbb{N}}$ and equip this set with the infinite product topology with respect to which it is compact and metrisable. We define the \emph{shift transformation} $\sigma \colon \Sigma_N \to \Sigma_N$ by $\sigma[(x_k)_{k=1}^\infty]:=(x_{k+1})_{k=1}^\infty$ and we denote the set of all $\sigma$-invariant Borel probability measures on $\Sigma_N$ by $\mathcal{M}_\sigma(\Sigma_N)$. For convenience we will refer to such measures simply as \emph{shift-invariant measures} on $\Sigma_N$. We equip $\mathcal{M}_\sigma(\Sigma_N)$ with the weak-* topology, which is compact and metrisable and has the property that $\mu \mapsto \int f\,d\mu$ defines a continuous function $\mathcal{M}_\sigma(\Sigma_N) \to \mathbb{R}$ for every $f \in C(\Sigma_N)$. We likewise define $\hat\Sigma_N:=\{1,\ldots,N\}^{\mathbb{Z}}$ with the infinite product topology, $\hat\sigma \colon \hat\Sigma_N \to \hat\Sigma_N$ by $\hat\sigma[(x_k)_{k \in \mathbb{Z}}]:=(x_{k+1})_{\mathbb{Z}}$, and let $\mathcal{M}_{\hat\sigma}(\hat\Sigma_N)$ denote the set of all $\hat\sigma$-invariant measures on $\hat\Sigma_N$ equipped with its weak-* topology with respect to which it is compact and metrisable. 

If $\iii=(i_k)_{k=1}^n \in \{1,\ldots,N\}^n$ is a finite sequence over the symbols $1,\ldots,N$ then we refer to $\iii$ as a \emph{word} over $\{1,\ldots,N\}$; we call $n$ the \emph{length} of the word $\iii$ and denote it by $|\iii|$. If $\iii=(i_k)_{k=1}^n$ and $\jjj=(j_k)_{k=1}^m$ are words then we let $\iii\jjj$ denote the word of length $n+m$ whose first $n$ symbols are $i_1,\ldots,i_n$ and whose next $m$ symbols are $j_1,\ldots,j_m$, and call $\iii\jjj$ the \emph{concatenation} of $\iii$ with $\jjj$. If $\iii$ is a word then for each $n \geq 1$ we let $\iii^n$ denote the concatenation of $n$ successive copies of $\iii$ and call this word the \emph{$n^{\mathrm{th}}$ power of $\iii$}. We denote the set of all words over $\{1,\ldots,N\}$ by $\Sigma_N^*$ and observe that the map $(\iii,\jjj) \mapsto \iii\jjj$ defines a semigroup operation on $\Sigma_N^*$. If $x=(x_k)_{k=1}^\infty \in \Sigma_N$ and $n \geq1$ are given, we let $x|_n$ denote the word $(x_k)_{k=1}^n$; if $\iii \in \Sigma_N^*$ is given, we let $[\iii]$ denote the set of all $x \in \Sigma_N$ such that $x|_n=\iii$. We will also write $x|_n:=(x_k)_{k=1}^n$ when $x \in \hat\Sigma_N$ and denote the set $\{x\in\hat\Sigma_N \colon x|_n=\iii\}$ by $[\iii]$ when the difference of context is clear. We refer to sets of the form $[\iii]$ as \emph{cylinder sets}. Cylinder sets generate the topology of $\Sigma_N$, and shifted cylinder sets $\hat\sigma^n [\iii]$ suffice to generate the topology of $\hat\Sigma_N$. We will usually denote words of length $1$ simply by the symbol in $\{1,\ldots,N\}$ which appears in that word, and the cylinders defined by words of length $1$ are therefore denoted $[1],\ldots,[N]$.

We define the natural projection $\pi \colon \hat\Sigma_N\to \Sigma_N$ by $\pi[(x_k)_{k \in \mathbb{Z}}]:=(x_k)_{k=1}^\infty$ which is clearly continuous and surjective. It is clear that $\hat\mu \mapsto \pi_*\hat\mu$ defines a continuous function $\mathcal{M}_{\hat\sigma}(\hat\Sigma_N) \to \mathcal{M}_{\sigma}(\Sigma_N)$ and since shift-invariant measures on $\Sigma_N$ and on $\hat\Sigma_N$ are in both cases characterised by their values on cylinder sets this map is bijective. Given $\mu \in \mathcal{M}_\sigma(\Sigma_N)$ we will simply write $\hat\mu$ for the unique element of $\mathcal{M}_{\hat\sigma}(\hat\Sigma_N)$ such that $\mu=\pi_*\hat\mu$, and we call $\hat\mu$ the natural extension of the measure $\mu$. Since properties such as ergodicity, total ergodicity and mixing can be characterised in terms of correlations between cylinder sets it is not difficult to see that each of those properties holds for an invariant measure $\mu \in \mathcal{M}_\sigma(\Sigma_N)$ if and only if the corresponding property holds for $\hat\mu\in \mathcal{M}_{\hat\sigma}(\hat\Sigma_N)$. A measure $\hat\mu$ on $\hat\Sigma_N$ will be called a Bernoulli measure if it has the form $\hat\mu=(\sum_{i=1}^N p_i \delta_i)^{\mathbb{Z}}$ for some probability vector $(p_1,\ldots,p_N)$. We will say $\hat \mu$ \emph{has the Bernoulli property} if there exist a Bernoulli measure $\hat\nu$ on $\hat\Sigma_N$ and a measure-space isomorphism $\phi \colon \hat\Sigma_N \to \hat\Sigma_N$ such that $\phi \circ \hat\sigma=\hat\sigma \circ \phi$ and $\phi_*\hat\mu = \hat\nu$. (This isomorphism must be understood with respect to the completions of the relevant Borel $\sigma$-algebras: see \S\ref{se:bananas} for details.) Clearly every Bernoulli measure trivially has the Bernoulli property, but the reverse is in general false.
 
\subsection{Potentials and equilibrium states}

For the remainder of this article a \emph{potential} will be any function $\Phi \colon \Sigma_N^* \to (0,+\infty)$, where $N \geq 2$ is arbitrary. We call a potential \emph{submultiplicative} if it satisfies the inequality $\Phi(\iii\jjj)\leq \Phi(\iii)\Phi(\jjj)$ for all $\iii,\jjj \in \Sigma_N^*$ and \emph{quasimultiplicative} if there exist $\delta>0$ and $m \geq 1$ such that $\max_{|\kkk|\leq m} \Phi(\iii\kkk\jjj)\geq \delta \Phi(\iii)\Phi(\jjj)$ for all $\iii,\jjj \in\Sigma_N^*$. If $\Phi \colon \Sigma_N^* \to (0,+\infty)$ is a submultiplicative potential then we define its \emph{pressure} to be the limit
\[P(\Phi):=\lim_{n \to \infty} \frac{1}{n}\log \sum_{\substack{\iii \in \Sigma_N^*\\|\iii|=n}} \Phi(\iii)\]
which exists by subadditivity. If additionally $\mu$ is a shift-invariant measure on $\Sigma_N$ then we define the \emph{ergodic average} of $\Phi$ to be the quantity
\begin{align*}\Lambda(\Phi,\mu)&:=\lim_{n \to \infty} \frac{1}{n}\int_{\Sigma_N} \log \Phi(x|_n) d\mu(x)\\
&=\lim_{n \to \infty} \frac{1}{n}\int_{\hat\Sigma_N} \log \Phi(x|_n) d\hat\mu(x)   = \lim_{n  \to \infty} \frac{1}{n}\sum_{\substack{\iii \in \Sigma_N^*\\|\iii|=n}} \mu([\iii])\log\Phi(\iii);\end{align*}
this limit likewise exists by subadditivity. When we wish to emphasise that we are working on the two-sided shift space $\hat\Sigma_N$ we may also denote this quantity by $\Lambda(\Phi,\hat\mu)$. If $\Phi$ is a submultiplicative potential defined on $\Sigma_N^*$ then the pressure of $\Phi$ admits the characterisation
\[P(\Phi)=\sup_{\mu \in \mathcal{M}_\sigma(\Sigma_N)} \left[h(\mu) + \Lambda(\Phi,\mu)\right] =\sup_{\hat\mu \in \mathcal{M}_{\hat\sigma}(\hat\Sigma_N)} \left[h(\hat\mu) + \Lambda(\Phi,\hat\mu)\right] \]
a fact which follows from more general results obtained in \cite{CaFeHu08}. We will prefer to say that an \emph{equilibrium state} of $\Phi$ is a measure $\mu \in \mathcal{M}_\sigma(\Sigma_N)$ such that $P(\Phi)=h(\mu)+\Lambda(\Phi,\mu)$, and in this case we call the measure $\hat\mu$ the natural extension of an equilibrium state. However, this choice of terminology is somewhat arbitrary and is chosen solely in order to have distinct names for $\mu$ and for $\hat\mu$. Since $\mathcal{M}_\sigma(\Sigma_N)$ is a compact metrisable topological space with respect to its weak-* topology and the quantities $h(\mu)$ and $\Lambda(\phi,\mu)$ depend upper semi-continuously on $\mu \in \mathcal{M}_\sigma(\Sigma_N)$, the existence of at least one equilibrium state for an arbitrary submultiplicative potential follows by elementary topological considerations. Since both $h(\mu)$ and $\Lambda(\Phi,\mu)$ are also affine as functions of the measure $\mu$, the set of equilibrium states of a submultiplicative potential is moreover convex and its extreme points are ergodic measures with respect to $\sigma$. 

We may now define generalised matrix equilibrium states. If $\mu$ is a shift-invariant measure on $\Sigma_N$, where $N \geq 2$, and $(A_1,\ldots,A_N) \in \GL(V)^N$ is a tuple of linear maps defined on a real finite-dimensional vector space $V$, we write $A_\iii:=A_{i_1}A_{i_2}\cdots A_{i_n}$ for every $\iii=(i_k)_{j=1}^n \in \Sigma_N^*$ and define
\begin{align*}\lambda_1(A,\mu)&:=\lim_{n \to \infty}  \frac{1}{n} \int_{\Sigma_N} \log \|A_{x|_n}\| d\mu(x)\\
& =\lim_{n \to \infty}  \frac{1}{n} \int_{\hat\Sigma_N} \log \|A_{x|_n}\| d\hat\mu(x)=\lim_{n \to \infty} \frac{1}{n} \sum_{|\iii|=n} \mu([\iii]) \log \|A_\iii\|  \end{align*}
which we call the \emph{top Lyapunov exponent} of $(A_1,\ldots,A_N)$ with respect to $\mu$.  We say that a shift-invariant measure $\mu$ on $\Sigma_N$, where $N \geq 2$, is a \emph{generalised matrix equilibrium state} if for some integer $k \geq 1$ there exist for each $j=1,\ldots,k$ a finite-dimensional real vector space $V_j$, a tuple $(A_1^{(j)},\ldots,A_N^{(j)}) \in \GL(V_j)^N$ of invertible linear maps, and a real number $\beta_j>0$, such that $\mu$ is an equilibrium state of the potential $\Phi \colon \Sigma_N* \to (0,+\infty)$ defined by
\begin{equation}\label{eq:stoomy-brown}\Phi(\iii):=\prod_{j=1}^k \left\|A_\iii^{(j)}\right\|^{\beta_j}.\end{equation}
It is easy to check that  $\mu$ is an equilibrium state of $\Phi$ if and only if it maximises the quantity $h(\nu)+\sum_{j=1}^k \beta_j\lambda_1(A^{(j)},\nu)$ over all $\nu \in \mathcal{M}_\sigma(\Sigma_N)$. The relation between tuples of linear maps $A_i^{(j)}$ and constants $\beta_j$ on the one hand, and generalised matrix equilibrium states on the other hand, is not bijective: a potential of the above form may have multiple equilibrium states (see for example \cite{BaFe11,BoMo18,FeKa11,MoSe19b}) and a single measure may be the equilibrium state of more than one such potential. Indeed, when proving theorems concerning generalised matrix equilibrium states it is often advantageous to look for spaces $V_j$ and tuples $(A_1^{(j)},\ldots,A_N^{(j)})$ which yield the same equilibrium state but have additional properties to those strictly required by the definition.

In the case $k=1$ generalised matrix equilibrium states as defined above have been previously referred to simply as \emph{matrix equilibrium states} or \emph{matrix Gibbs states}, which motivates our choice of terminology: see for example \cite{Mo18a,Mo18b,Pi20}. Matrix equilibrium states are sufficient to study the (candidate) measures of maximal dimension for self-affine subsets of the plane, but for self-affine subsets of $\mathbb{R}^d$ with $d>2$ it seems to be unavoidably necessary to consider the case $k=2$: specifically, one must consider potentials of the form
\begin{equation}\label{eq:blue-child}\Phi^s(\iii):=\left\|A_\iii^{\wedge \lfloor s\rfloor}\right\|^{1+\lfloor s\rfloor -s} \left\|A_\iii^{\wedge\lceil s\rceil}\right\|^{s-\lfloor s\rfloor}\end{equation}
for an appropriate tuple $(A_1,\ldots,A_N) \in \GL_d(\mathbb{R})^N$ and real number $s \in (0,d)$, see \S\ref{ss:kaenmaki} below for details. Matrix equilibrium states (i.e. the case $k=1$) are substantially easier to handle and results in the general case have typically been preceded in the literature by results in the case $k=1$; the reason for this difference in difficulty can be attributed to the fact that the case $k=1$ can be understood using linear-algebraic techniques by embedding the linear maps $A_1,\ldots,A_N\in \GL(V)$ in the subalgebra of $\ned(V)$ which they generate, whereas for general $k$ substantial progress has only been made by embedding the tuples $(A_1^{(j)},\ldots,A_N^{(j)})$ in a linear algebraic group and applying ideas from algebraic geometry (see for example \cite{BoMo18,MoSe19a}). While the definition of a generalised matrix equilibrium state also makes sense in some cases in which the linear maps $A_i^{(j)}$ are not assumed to be invertible, this is more difficult to handle mathematically when $k>1$ and in the present work we will always assume the invertibility of the linear maps $A_i^{(j)}$. We will also find the following terminology helpful: if $V$ is a finite-dimensional real vector space then $(A_1,\ldots,A_N) \in \GL(V)^N$ will be called \emph{irreducible} if there does not exist a nonzero proper linear subspace $U$ of $V$ which is preserved by every $A_i$, and \emph{strongly irreducible} if there does not exist a finite collection $\{U_1,\ldots,U_m\}$ of nonzero proper linear subspaces of $V$ such that every $A_i$ induces a permutation of the set $\{U_1,\ldots,U_m\}$. 

\subsection{Main results and previous literature}

The early literature on matrix equilibrium states focused on studying the associated pressure function, proving the existence of equilibrium states, characterising their uniqueness (or otherwise) and describing their supports (see for example \cite{Fe03,FeLa02,FeKa11,Mo18a}) with results on potentials of the form $\Phi^s$ as in \eqref{eq:blue-child} available only in special cases \cite{FaSl09,Fr15,JaJaLiSt16,KaMo18}. 
The broader concept of a generalised matrix equilibrium state was introduced in \cite{BoMo18} where an upper bound was given for the number of ergodic generalised matrix equilibrium states that can correspond to a single potential, and where it was also shown that all generalised matrix equilibrium states as defined above are fully supported on the relevant symbolic space $\Sigma_N$; these results in particular apply to potentials of the form \eqref{eq:blue-child} and subsumed many prior results on that topic. In parallel with this work the author investigated several aspects of the ergodic properties of matrix equilibrium states in \cite{Mo18a,Mo19a}, showing in particular that matrix equilibrium states are mixing if and only if they are totally ergodic and giving necessary and sufficient conditions for mixing. This left open the question of whether mixing matrix equilibrium states admit stronger properties such as being exact endomorphisms or having the Bernoulli property, and also left open the question of whether similar results hold for generalised matrix equilibrium states. The former question was partially addressed by M. Piraino in \cite{Pi20} in the case of matrix equilibrium states, where a sufficient condition was given for a matrix equilibrium state to have the Bernoulli property. This nonetheless left unresolved the question of whether \emph{every} matrix equilibrium state which is mixing also has the Bernoulli property and did not address the corresponding questions for generalised matrix equilibrium states. In the present work we give complete answers to all of these questions in the following result:
\begin{theorem}\label{th:main}
Let $N \geq 2$ and let $\mu$ be a totally ergodic generalised matrix equilibrium state on $\Sigma_N$. Then $\mu$ is \emph{$\psi$-mixing}:
\[\lim_{n \to \infty} \sup_{\iii,\jjj \in \Sigma_N^*} \left|\frac{\mu([\iii] \cap \sigma^{-n-|\iii|} [\jjj])}{\mu([\iii])\mu([\jjj])} -1\right|=0,\]
and its natural extension $\hat\mu$ has the Bernoulli property.
\end{theorem}
It is interesting to ask whether the rate of convergence in the above limit is exponential as is the case in the classical thermodynamic formalism of additive potentials (see for example \cite{Ba00,Bo75,PaPo90}); this may have implications for the statistical properties of typical trajectories with respect to $\mu$ as in, for example, \cite[\S7]{PhSt75}. A sufficient condition for this exponential rate of convergence in the case of matrix equilibrium states was given by M. Piraino in \cite{Pi20}, but to answer this question in the generality of Theorem \ref{th:main} seems likely to require additional developments in the transfer operator theory of linear cocycles.

It is not difficult to extend Theorem \ref{th:main} to the case where $\mu$ is ergodic but not totally ergodic, although the fundamental result in this direction is cumbersome to state:
\begin{theorem}\label{th:recoding}
Let $k \geq 1$ and $N \geq 2$. For each $j=1,\ldots,k$ let $V_j$ be a finite-dimensional real vector space and let $(A_1^{(j)},\ldots,A_N^{(j)}) \in \GL(V_j)^N$ and $\beta_j>0$. For all $\iii \in \Sigma_N^*$ define
\[\Phi(\iii):=\prod_{j=1}^k \left\|A_\iii^{(j)}\right\|^{\beta_j}\]
and let $\mu$ be an ergodic equilibrium state of $\Phi$. If $\mu$ is not totally ergodic then there exists an integer $n$ satisfying $1 < n \leq\prod_{j=1}^k \dim V_j$ with the following properties. Let $\eta \colon  \{\iii \in \Sigma_N^* \colon |\iii|=n\}\to \{1,\ldots,N^n\}$ be the map which takes each word $\iii \in \Sigma_N^*$ of length $n$ to the integer representing its position in the lexicographical ordering on $\{\iii \in \Sigma_N^* \colon |\iii|=n\}$ and define a homeomorphism $\iota \colon \Sigma_{N} \to \Sigma_{N^n}$ satisfying $\sigma \circ \iota = \iota \circ \sigma^n$ by $\iota[(x_\ell)_{\ell=1}^\infty]:=(\eta(x_{(q-1)n+1}\cdots x_{qn}))_{q=1}^\infty$ for every $(x_\ell)_{\ell=1}^\infty \in \Sigma_N$.  For each $j=1,\ldots,k$ define an $N^n$-tuple $(B_1^{(j)},\ldots,B_{N^n}^{(j)}) \in \GL(V_j)^{N^n}$ by $B_i^{(j)}:=A_{\eta^{-1}(i)}^{(j)}$ for every $i=1,\ldots,N^n$ and $j=1,\ldots,k$, and define a potential $\Psi \colon \Sigma_{N^n}^* \to(0,+\infty)$ by
\[\Psi(\jjj)=\prod_{j=1}^k \left\|B_\jjj^{(j)}\right\|^{\beta_j}\]
for all $\jjj \in \Sigma_{N^n}^*$. Then we may write $\mu = \frac{1}{n}\sum_{i=0}^{n-1} \sigma^i_*\nu$ where $\nu$ is a $\sigma^n$-invariant measure on $\Sigma_N$ and where for every $i=0,\ldots,n-1$ the measure $(\iota \circ \sigma^i )_*\nu \in \mathcal{M}_\sigma(\Sigma_{N^n})$ is a distinct totally ergodic equilibrium state of $\Psi$.
\end{theorem}
The proof of Theorem \ref{th:recoding} is technically straightforward but involves a substantial volume of checking and is given in \S\ref{se:bananas} below. Since each measure $(\iota \circ \sigma^i)_*\nu$ is a totally ergodic equilibrium state of $\Psi$, by Theorem \ref{th:main} its natural extension is measurably isomorphic to a Bernoulli process. It is not difficult to deduce:
\begin{corollary}\label{co:vid19}
Let $k \geq 1$ and $N \geq 2$. For each $j=1,\ldots,k$ let $V_j$ be a finite-dimensional real vector space and let $(A_1^{(j)},\ldots,A_N^{(j)}) \in \GL(V_j)^N$ and $\beta_j>0$. For all $\iii \in \Sigma_N^*$ define
\[\Phi(\iii)=\prod_{j=1}^k \left\|A_\iii^{(j)}\right\|^{\beta_j}\]
and let $\mu$ be an ergodic equilibrium state of $\Phi$. Then there exists an integer $n$ satisfying $1 \leq n \leq\prod_{j=1}^k \dim V_j$ such that the natural extension of $\mu$ is measurably isomorphic to the product of a Bernoulli process with the rotation map $m \mapsto m+1 \mod n$ on $\mathbb{Z}_n$.
\end{corollary} 
The proof of this corollary is likewise presented in \S\ref{se:bananas}. We also note the following:
\begin{corollary}\label{co:xswain}
Let $k \geq 1$ and $N \geq 2$. For each $j=1,\ldots,k$ let $V_j$ be a finite-dimensional real vector space and let $(A_1^{(j)},\ldots,A_N^{(j)}) \in \GL(V_j)^N$ and $\beta_j>0$. For all $\iii \in \Sigma_N^*$ define
\[\Phi(\iii)=\prod_{j=1}^k \left\|A_\iii^{(j)}\right\|^{\beta_j}.\]
If every $(A_1^{(j)},\ldots,A_N^{(j)})$ is strongly irreducible then there is a unique equilibrium state for $\Phi$ and that equilibrium state is $\psi$-mixing and has the Bernoulli property. \end{corollary} 
\begin{proof}
By  \cite[Corollary 2.2]{BoMo18} if every $(A_1^{(j)},\ldots,A_N^{(j)})$ is strongly irreducible then $\Phi$ must have a unique equilibrium state $\mu$. If $\mu$ is not totally ergodic, let $n>1$, $\Psi$ and $(B_1^{(j)},\ldots,B_{N^n}^{(j)})$ be as given by Theorem \ref{th:recoding}. It is easy to see that the tuples $(B_1^{(j)},\ldots,B_{N^n}^{(j)})$ must also strongly irreducible by construction, so by the same reasoning $\Psi$ has a unique equilibrium state. But Theorem \ref{th:recoding} implies that $\Psi$ has at least $n$ distinct ergodic equilibrium states, which is a contradiction. We conclude that $\mu$ must be totally ergodic, so Theorem \ref{th:main} applies and $\mu$ is $\psi$-mixing and has the Bernoulli property.
\end{proof}

In the case of matrix equilibrium states total ergodicity has already been fully characterised in the following sense. If $V$ is a finite-dimensional real vector space and $(A_1,\ldots,A_N) \in \GL(V)^N$ is irreducible then for each $\beta>0$ there exists a unique matrix equilibrium state for the potential $\Phi(\iii):=\|A_\iii\|^\beta$, see for example \cite{FeKa11}. (Moreover, every ergodic matrix equilibrium state is the unique equilibrium state of such a potential.) In this situation it was shown in \cite{Mo19a} that if this matrix equilibrium state is \emph{not} totally ergodic then there necessarily exists a \emph{cyclic splitting} for $V$: we may write $V=\bigoplus_{j=1}^{m} U_j$ where each $U_j$ is a linear subspace of $V$ and where $A_iU_j = U_{j+1 \mod m}$ for all $i=1,\ldots,N$ and $j=1,\ldots,m$. (Examples in which total ergodicity of a matrix equilibrium state fails had already been constructed in \cite{Mo18a}.) It is natural to ask whether this result extends to generalised matrix equilibrium states: if a generalised matrix equilibrium state as in Theorem \ref{th:main} is not totally ergodic, is it the case that one of the tuples $(A_1^{(j)},\ldots,A_N^{(j)}) \in \GL(V_j)^N$ preserves a cyclic splitting for the associated vector space $V_j$? For generalised matrix equilibrium states the situation seems to be more complicated than this, and we are able to show that this result does not hold. We give the following example in which total ergodicity fails but the matrix tuples do not admit cyclic splittings:
\begin{proposition}\label{pr:not-tot}
Define two irreducible pairs of linear maps $(A_1,A_2), (B_1,B_2)\in \GL_2(\mathbb{R})^2$ by
\[A_1:=\begin{pmatrix}2&0\\0&1\end{pmatrix} \qquad A_2:= \begin{pmatrix}0&1\\1&0\end{pmatrix},\]
\[B_1:=\begin{pmatrix}0&1\\1&0\end{pmatrix} \qquad B_2:= \begin{pmatrix}1&0\\0&2\end{pmatrix}\]
and let $\beta_1,\beta_2>0$ be arbitrary. Define a potential $\Phi \colon \Sigma_2^* \to (0,+\infty)$ by $\Phi(\iii):=\|A_\iii\|^{\beta_1}\|B_\iii\|^{\beta_2}$. Then $\Phi$ has a unique equilibrium state and this equilibrium state is not totally ergodic.
\end{proposition}
The proof of Proposition \ref{pr:not-tot} is also given in \S\ref{se:bananas}.

\subsection{Connections with self-affine sets}\label{ss:kaenmaki}

We now describe in more detail the connections between Theorem \ref{th:main} and self-affine sets. If $T_1,\ldots,T_N \colon \mathbb{R}^d \to \mathbb{R}^d$ are invertible affine contractions (with respect to some fixed norm on $\mathbb{R}^d$ which need not be the Euclidean norm) then there exists a unique nonempty compact set $X \subset \mathbb{R}^d$ satisfying $X=\bigcup_{i=1}^N T_iX$. Such sets $X$ are referred to as \emph{self-affine sets}. In the situation where the images $T_1X,\ldots,T_NX$ are pairwise disjoint it is not difficult to define an expanding map $f \colon \mathbb{R}^d \to \mathbb{R}^d$ such that $X$ is a repelling set for $f$ and such that $D_xf = T_i^{-1}$ whenever $x \in T_iX$, so self-affine sets with this disjointness property (which is called the strong separation condition in the fractal geometry literature) are a particular case of the expanding repellers discussed in \S\ref{se:one}. Besides their connection with questions such as Conjecture \ref{qu:eensparkrangers} self-affine sets are the subject of a deep and substantial literature in their own right, beginning in the 1980s with such works as \cite{Be84,Fa88,Mc84} and flowering into a highly active contemporary research topic (see for example \cite{Ba07,BaHoRa19,BaFe13,BoMo18,DaSi17,Fe19,FeSh14,Fr12,JoPoSi07}). As well as in its connection to Conjecture \ref{qu:eensparkrangers} the construction of high-dimensional measures on self-affine sets is important to the problem of obtaining sharp lower bounds on the Hausdorff dimension of the set itself. Theorem \ref{th:main} in particular has implications for the structure of certain high-dimensional measures on self-affine sets, called \emph{K\"aenm\"aki measures}, which we now describe.

If $V$ is a $d$-dimensional vector space equipped with an inner product, we recall that the \emph{singular values} of $A \in \GL(V)$ are defined to be the positive square roots of the eigenvalues of the positive definite linear map $A^\top A$. We write the singular values as $\sigma_1(A),\ldots,\sigma_d(A)$ in decreasing order with repetition in the case of multiple eigenvalues. 
For each $s \geq 0$ and $A \in \GL_d(\mathbb{R})$ the \emph{singular value function}, introduced by Falconer in \cite{Fa88}, is the function $\varphi^s \colon \GL_d(\mathbb{R})\to \mathbb{R}$ defined by
\[\varphi^s(A):=\left\{\begin{array}{cl} \sigma_1(A)\cdots \sigma_{\lfloor s\rfloor}(A) \sigma_{\lceil s\rceil}(A)^{s-\lfloor s\rfloor}&\text{if $0\leq s\leq d$,}\\\left|\det A\right|^{\frac{s}{d}}&\text{if $s \geq d$,}\end{array}\right.\]
where $\varphi^0(A)$ is understood to equal $1$. The singular value function satisfies $\varphi^s(AB) \leq \varphi^s(A)\varphi^s(B)$ for all $A,B \in \GL_d(\mathbb{R})$. If $T_1,\ldots,T_N \colon \mathbb{R}^d \to \mathbb{R}^d$ are invertible affine contractions with respect to some fixed norm on $\mathbb{R}^d$, let us write each $T_i$ in the form $T_i(u):=A_iu + v_i$ for all $u \in \mathbb{R}^d$, where $A_i \in \GL(\mathbb{R}^d)$ and $v_i \in \mathbb{R}^d$ for each $i=1,\ldots,N$.

For each $s \geq 0$ we say that a $\varphi^s$-equilibrium state for $(T_1,\ldots,T_N)$ is an equilibrium state of the submultiplicative potential $\Phi^s(\iii):=\varphi^s(A_\iii)$. It is not particularly difficult to show that the function $s \mapsto P(\Phi^s)$ is continuous and strictly decreasing with $P(\Phi^0)>0$ and $\lim_{s \to \infty}P(\Phi^s)=-\infty$, so in particular there exists a unique $s >0$ such that $P(\Phi^s)=0$, called the \emph{affinity dimension} of $(T_1,\ldots,T_N)$.  By definition a \emph{K\"aenm\"aki measure} for $(T_1,\ldots,T_N)$ is a $\varphi^s$-equilibrium state for $(T_1,\ldots,T_N)$ where $s$ is the affinity dimension. Crucially every K\"aenm\"aki measure is a generalised matrix equilibrium state, since we have
\[\Phi^s(\iii)= \left\{\begin{array}{cl}\left\|A_\iii^{\wedge \lfloor s\rfloor}\right\|^{1+\lfloor s\rfloor -s} \left\|A_\iii^{\wedge \lceil s\rceil}\right\|^{s-\lfloor s\rfloor}&\text{if $0 \leq s\leq d$,}\\
\left|\det A_\iii\right|^{\frac{s}{d}}&\text{if $s \geq d$,}
\end{array}\right.\]
where $A^{\wedge k}$ denotes the $k^{\mathrm{th}}$ exterior power of the linear map $A$; for details see the following section. (Here $A^{\wedge 0}$ is always understood to equal the identity linear map on $\mathbb{R}$.) It is not difficult to show that there exists a well-defined continuous function $\Pi \colon \Sigma_N \to \mathbb{R}^d$ which satisfies
\[\Pi\left[(x_k)_{k=1}^\infty\right]=\lim_{n \to \infty} T_{x_1}T_{x_2}\cdots T_{x_n}v\]
for all $v \in \mathbb{R}^d$, and indeed the image $\Pi(\Sigma_N)$ is precisely the attractor of $(T_1,\ldots,T_N)$. (It is for this reason that in this article we multiply matrices on the right -- we define $A_\iii:=A_{i_1}\cdots A_{i_n}$ and not $A_\iii:=A_{i_n}\cdots A_{i_1}$ -- and not on the left as is more natural in many other contexts.) It follows from a result of Jordan, Pollicott and Simon (\cite{JoPoSi07}, see also \cite{KaRe14}) that if a shift-invariant measure $\mu$ on $\Sigma_N$ has the property that $\Pi_*\mu$ has Hausdorff dimension equal to the affinity dimension then it is necessarily a K\"aenm\"aki measure for $(T_1,\ldots,T_N)$, and in this sense K\"aenm\"aki measures are the natural candidates for the measures of maximal dimension on self-affine sets.

Theorem \ref{th:main} and Corollary \ref{co:xswain} together yield the following result for K\"aenm\"aki measures:
\begin{corollary}\label{co:xsapples}
Let $T_1,\ldots,T_N \colon \mathbb{R}^d \to \mathbb{R}^d$ be invertible affine maps which are all contracting with respect to some fixed norm on $\mathbb{R}^d$ and let $s>0$ denote the affinity dimension of $(T_1,\ldots,T_N)$. If $\mu$ is a totally ergodic K\"aenm\"aki measure for $(T_1,\ldots,T_N)$ then 
\[\lim_{n \to \infty} \sup_{\iii,\jjj \in \Sigma_N^*} \left|\frac{\mu([\iii] \cap \sigma^{-n-|\iii|} [\jjj])}{\mu([\iii])\mu([\jjj])} -1\right|=0\]
and the natural extension of $\mu$ is measurably isomorphic to a Bernoulli measure. This holds in particular if the tuples $(A_1^{\wedge \lfloor s\rfloor},\ldots,A_N^{\wedge \lfloor s\rfloor})$ and $(A_1^{\wedge \lceil s\rceil},\ldots,A_N^{\wedge \lceil s\rceil})$ are both strongly irreducible.
\end{corollary}
In several works on the dimension theory of K\"aenm\"aki measures it has been possible to obtain stronger results if an additional assumption is made, called the \emph{quasi-Bernoulli property}. A measure $\mu$ on $\Sigma_N$ is called quasi-Bernoulli if there exists a constant $C>0$ such that $C^{-1}\mu([\iii])\mu([\jjj]) \leq \mu([\iii\jjj])\leq C\mu([\iii])\mu([\jjj])$ for all $\iii,\jjj \in \Sigma_N^*$. (In other literatures this property is sometimes called \emph{local product structure}: see for example \cite{BoVi04}). It follows from the results of \cite{BoMo18} that every ergodic generalised matrix equilibrium state satisfies the upper bound $\mu([\iii\jjj])\leq C\mu([\iii])\mu([\jjj])$, but the lower bound does not hold in general (see for example \cite{BaKaMo19}). If $T_1,\ldots,T_N$ are affine contractions of $\mathbb{R}^d$ with respect to some fixed norm, let us say that the \emph{$n$-step recoding} of $(T_1,\ldots,T_N)$ is the $N^n$-tuple $(\hat{T}_1,\ldots,\hat{T}_{N^n}):=(T_1^n,T_1^{n-1}T_2,T_1^{n-1}T_3,\ldots,T_N^{n-1}T_{N-1},T_N^n)$ which lists all compositions of the form $T_{i_1}\cdots T_{i_n}$ in lexicographical order. It is easy to see that if $X=\bigcup_{i=1}^N T_iX$ then $X=\bigcup_{i=1}^{N^n}\hat{T}_iX$, so the tuples $(T_1,\ldots,T_N)$ and $(\hat{T}_1,\ldots,\hat{T}_{N^n})$ describe the same self-affine set. Moreover one may show that the affinity dimensions of $(T_1,\ldots,T_N)$ and $(\hat{T}_1,\ldots,\hat{T}_{N^n})$ are equal. By recoding $(T_1,\ldots,T_N)$ by the integer $n_0$ given by Theorem \ref{th:recoding} we may recode any $(T_1,\ldots,T_N)$ into a new tuple all of whose ergodic K\"aenm\"aki measures are totally ergodic and therefore are $\psi$-mixing. By recoding a second time we may for any prescribed $\varepsilon>0$ guarantee that for every ergodic K\"aenm\"aki measure $\mu$ of the twice-recoded system $(\hat{T}_1,\ldots,\hat{T}_{N^n})$ we have
\[ \sup_{\iii,\jjj \in \Sigma_{N^{n}*}} \left|\frac{\nu([\iii] \cap \sigma^{-1-|\iii|} [\jjj])}{\nu([\iii])\nu([\jjj])} -1\right|<\varepsilon\]
which is to say 
\[(1-\varepsilon)\mu([\iii])\mu([\jjj]) \leq \sum_{\ell=1}^{N^n} \mu([\iii \ell \jjj]) \leq (1+\varepsilon) \mu([\iii])\mu([\jjj]) \]
for all $\iii,\jjj \in \Sigma_{N^n}^*$. It is interesting to ask whether this property may have dimension-theoretic applications similar to those of the quasi-Bernoulli property. 
\subsection{Strategy of proof and structure of the paper}

The fundamental objective in the proof of Theorem \ref{th:main} is to establish, given a totally ergodic generalised matrix equilibrium state $\mu$ on $\Sigma_N$, the following property which we refer to as the \emph{pre-condition for $\psi$-mixing}: there exist an integer $m \geq 1$ and a real number $\delta>0$ depending only on $\mu$ such that 
\begin{equation}\label{eq:clardic-fug}\max_{\substack{\kkk \in \Sigma_N^*\\|\kkk|=m}} \mu([\iii\kkk\jjj])\geq \delta \mu([\iii])\mu([\jjj])\end{equation}
for all $\iii,\jjj \in \Sigma_N^*$. By combining this result with a theorem of R.C. Bradley \cite{Br05} it can easily be deduced that the natural extension $\hat\mu$ is $\psi$-mixing, which implies the same result for $\mu$. A celebrated theorem of N.A. Friedman and D.S. Ornstein \cite{FrOr70} on isomorphism with Bernoulli processes then allows us to pass directly from the $\psi$-mixing property for $\hat\mu$ to the Bernoulli property. This basic strategy for proving $\psi$-mixing and deducing the Bernoulli property follows that used by M. Piraino in \cite{Pi20}.

The route to the condition \eqref{eq:clardic-fug} divides naturally into three principal stages. In the first stage, which is relatively elementary, we show that every ergodic generalised matrix equilibrium state $\mu$ can be represented by a potential $\Phi$ defined in terms of tuples $(A_1^{(j)},\ldots,A_N^{(j)})$ all of which are irreducible and all of which have simple top Lyapunov exponent with respect to $\mu$. This to some extent reprises arguments already used in \cite{BoMo18} but with the additional detail of the top Lyapunov exponent to be considered. In the second stage we use analytic arguments to further show that $\mu$ is the \emph{unique} equilibrium state of a potential of the form $\Phi_{\mathcal{W}}(\iii):=\max_{(W_j)_{j=1}^k \in \mathcal{W}}\prod_{j=1}^k \|A_\iii^{(j)}|_{W_j}\|^{\beta_j}$ where $\mathcal{W}$ is a finite invariant set of tuples $(W_j)_{j=1}^k$ of subspaces $W_j$ of the respective vector space $V_j$, and such that $\mathcal{W}$ has an additional combinatorial property called \emph{primitivity}: this is the stage at which total ergodicity is used.  In the third stage, which is more algebraic, these ingredients are combined to obtain the inequality \eqref{eq:clardic-fug}. We may then deduce Theorem \ref{th:main} from \eqref{eq:clardic-fug} in a fairly straightforward manner. This division into parts is reflected in similar divisions in the proofs of other major results on generalised matrix equilibrium states given in \cite{BoMo18,MoSe19a}: in the first stage of the argument we obtain irreducibility, in the second stage we treat complications arising from the possibility of irreducibility without strong irreducibility, and in the last stage we deal with a reduced case in which the arguments applicable to the strongly irreducible case are available. To illustrate this we remark that in the strongly irreducible case, the arguments in the second stage mostly collapse to trivialities; and in the case where for each $j$ there exists a one-dimensional space with finite orbit under $(A_1^{(j)},\ldots,A_N^{(j)})$, the arguments in the second stage become of fundamental importance whereas those in the third stage become trivial instead.

The remainder of the paper is therefore structured as follows. In the following section we recall various foundational results in linear algebra, ergodic theory and algebraic geometry which will be used in various parts of the proof of Theorem \ref{th:main}. The three stages in the proof of \eqref{eq:clardic-fug} just described are given successively in sections \ref{se:firststage} through \ref{se:thirdstage}. In \S\ref{se:bananas} we combine these results to obtain Theorem \ref{th:main} and also prove the various minor additional results described in this section.




\section{Preliminaries}\label{mango}

\subsection{Linear algebra}\label{sss:la} We first recall some concepts and identities from linear and multilinear algebra which will be used in various sections of this article. Here and throughout the article $\ned(V)$ denotes the vector space of linear endomorphisms of the vector space $V$, and $\rho(A)$ denotes the largest of the absolute values of the eigenvalues of the linear map $A \in \ned(V)$. Proofs of the following statements concerning exterior powers and tensor products may be found in, for example, \cite[\S{XVI}]{MaBi99}; the material on singular values is more commonly found in texts on matrix analysis such as \cite{HoJo94}. 

\subsubsection{Exterior powers} If $V$ is a finite-dimensional real (or complex) vector space of dimension $d$ then for every $k=1,\ldots,d$ there exists a vector space $\wedge^k V$ of dimension ${d \choose k}$, called the $k^{\mathrm{th}}$ exterior power of $V$, which is spanned by all expressions of the form $v_1 \wedge v_2 \wedge \cdots \wedge v_k$ such that $v_1,\ldots,v_k \in V$. These objects are subject to the identities 
\[(\lambda v_1 + u_1) \wedge v_2 \wedge \cdots \wedge v_k  = \lambda (v_1 \wedge v_2 \wedge + \cdots + \wedge v_k) + u_1 \wedge v_2 \wedge + \cdots + \wedge v_k,\]
\[v_1 \wedge \cdots \wedge v_i \wedge v_{i+1} \wedge \cdots \wedge  v_k = -v_1 \wedge \cdots \wedge v_{i+1} \wedge v_{i} \wedge \cdots \wedge  v_k \]
for all $v_1,\ldots,v_k,u_1 \in V$, all $i \in \{1,\ldots,k-1\}$ and all $\lambda$  in $\mathbb{R}$ (or $\mathbb{C}$). If $e_1,\ldots,e_d$ is a basis for $V$ then the vectors $e_{i_1}\wedge \cdots \wedge e_{i_k}$ such that $1 \leq i_1<\cdots<i_k \leq d$ form a basis for $\wedge^kV$. If $A \in \ned(V)$ then  the $k^{\mathrm{th}}$ exterior power of $A$ is the unique linear map $A^{\wedge k} \in \ned(\wedge^k V)$ characterised by the property $A^{\wedge k} (v_1\wedge \cdots \wedge v_k) = Av_1 \wedge \cdots \wedge Av_k$ for every $v_1,\ldots,v_k \in V$. The identity $(AB)^{\wedge k}=A^{\wedge k}B^{\wedge k}$ for all $A,B \in \ned(V)$ is clear. By considering appropriate bases it is not difficult to see that if the eigenvalues of $A$ are $\lambda_1,\ldots,\lambda_d$ then the eigenvalues of $A^{\wedge k}$ are precisely the products $\lambda_{i_1}\wedge \cdots \wedge \lambda_{i_k}$ such that $1 \leq i_1<\cdots<i_k \leq d$. If $V$ is additionally equipped with an inner product $\langle \cdot,\cdot\rangle$ then it induces an inner product on $\wedge^k V$ by $\langle u_1 \wedge \cdots \wedge u_k,v_1 \wedge \cdots \wedge v_k\rangle := \det ([\langle u_i,v_j\rangle]_{i,j=1}^k)$, and with respect to these inner products it is clear that $(A^\top)^{\wedge k}\equiv (A^{\wedge k})^\top$. 

\subsubsection{Tensor products} If $V_1,\ldots,V_k$ are finite-dimensional real (or complex) vector spaces then their tensor product $\bigotimes_{i=1}^k V_i$ is a vector space of dimension $\prod_{i=1}^k \dim V_i$ spanned by all expressions of the form $v_1 \otimes \cdots \otimes v_k$ such that $v_i \in V_i$ for every $i=1,\ldots,k$, subject to the identity
\[v_1 \otimes  \cdots \otimes (\lambda v_i + u)\otimes \cdots \otimes v_k = \lambda(v_1 \otimes  \cdots \otimes  v_i \otimes \cdots \otimes v_k) +  v_1 \otimes \cdots \otimes u \otimes \cdots \otimes v_k\]
for all $v_1 \in V_1$, $v_2 \in V_2$,\ldots,$v_k \in V_k$, all $\lambda$ in $\mathbb{R}$ (or $\mathbb{C}$) and $u \in V_i$, for all $i=1,\ldots,k$. If for each $i=1,\ldots,k$ we are given a basis $e_{1,i},\ldots,e_{d_i,i}$ for $V_i$ then the $\prod_{i=1}^k d_i$ vectors of the form $e_{j_1,1}\otimes e_{j_2,2} \otimes \cdots \otimes e_{j_k,k}$ with $1 \leq j_i \leq d_i$ for each $i=1,\ldots,k$ form a basis for $\bigotimes_{i=1}^k V_i$.  If linear maps $A_1 \in \ned(V_1),\ldots,A_k \in \ned(V_k)$ are given then they induce a linear map $\bigotimes_{i=1}^k A_i$ on $\bigotimes_{i=1}^k V_i$ by $(\bigotimes_{i=1}^k A_i)(u_1\otimes \cdots \otimes u_k):=A_1u_1\otimes \cdots \otimes A_ku_k$. By considering appropriate bases it is not difficult to show that the $\prod_{i=1}^k \dim V_i$ eigenvalues of $\bigotimes_{i=1}^k A_i$ are precisely the products of the form $\prod_{i=1}^k \lambda_i$ where for each $i$ the number $\lambda_i$ is an eigenvalue of $A_i$. In particular we have $\rho(\bigotimes_{i=1}^k A_i)=\prod_{i=1}^k \rho(A_i)$ whenever $A_i \in \ned(A_i)$ for every $i=1,\ldots,k$. If for each $i$ we are given an inner product $\langle \cdot,\cdot\rangle_{V_i}$ on $V_i$ then we may define an inner product on $\bigotimes_{i=1}^k V_i$ by defining $\langle u_1 \otimes \cdots \otimes u_k,v_1 \otimes \cdots \otimes v_k\rangle := \prod_{i=1}^k \langle u_i,v_i\rangle_{V_i}$ for all $v_1 \in V_1$,\ldots,$v_k \in V_k$ and extending linearly. It is not difficult to see that with respect to this inner product we have $(\bigotimes_{i=1}^k A_i)^\top \equiv \bigotimes_{i=1}^k (A_i^\top)$ and the identity $\|\bigotimes_{i=1}^k A_i\| = \prod_{i=1}^k \|A_i\|$ follows.

\subsubsection{Singular values}\label{sss:singular}
If $V$ is a $d$-dimensional real or complex vector space equipped with an inner product, the singular values of a linear map $A \in \ned(V)$ are defined to be the non-negative square roots of the eigenvalues of the positive semidefinite linear map $A^\top A$ listed in decreasing order with repetition in the case of multiple eigenvalues. All vector spaces in this article will be assumed to be equipped with inner products. If $\dim V=d$ we denote the singular values of $A$ by $\sigma_1(A),\ldots,\sigma_d(A)$. The singular values are well known to satisfy the alternative characterisation
\[\sigma_i(A) =\min \{\|A-F\| \colon \rank F <i\}\]
for all $A \in \ned(V)$ and $i=1,\ldots,d$. It may be easily demonstrated using these two descriptions that the singular values satisfy the identities $\prod_{i=1}^d \sigma_i(A)=|\det A|$ and $\|A\|=\sigma_1(A)$ and also satisfy the inequality $\sigma_i(X_1AX_2) \leq \|X_1\| \cdot \sigma_i(A) \cdot \|X_2\|$ for all $A,X_1,X_2 \in \ned(V)$ and $i=1,\ldots,d$. For every $k=1,\ldots,d$ and $A \in \ned(V)$ the singular values of $A^{\wedge k}$ (relative to the inner product on $\wedge^k V$ induced by the inner product on $V$) are the square roots of the eigenvalues of $(A^{\wedge k})^\top A^{\wedge k} = (A^\top A)^{\wedge k}$ and hence are precisely the products $\sigma_{i_1}(A)\cdots \sigma_{i_k}(A)$ such that $1 \leq i_1<\cdots<i_k \leq d$. In particular the largest singular value of $A^{\wedge k}$ is $\sigma_1(A)\cdots \sigma_k(A)$, so we have $\|A^{\wedge k}\|=\sigma_1(A)\cdots \sigma_k(A)$ for all $A \in \ned(V)$ and $k=1,\ldots,d$. The inequality $\prod_{i=1}^k \sigma_i(AB) \leq (\prod_{i=1}^k \sigma_i(A)) (\prod_{i=1}^k \sigma_i(B))$ for all $A,B \in \ned(V)$ and all $k=1,\ldots,d$ follows. 

In general the singular values of $A$ are defined only relative to a specified inner product on $V$ and may change if a different inner product is used. If $\langle \cdot,\cdot\rangle_1$ and $\langle \cdot,\cdot\rangle_2$ are distinct inner products on $V$, $X \in \GL(V)$ is an isometry from $(V,\langle \cdot,\cdot\rangle_1)$ to $(V,\langle \cdot,\cdot\rangle_2)$, and $\sigma_1(A),\ldots,\sigma_d(A)$ denote the singular values of $A \in \GL(V)$ as calculated with respect to $\langle \cdot,\cdot\rangle_1$ then the singular values of $A$ as calculated with respect to $\langle \cdot,\cdot \rangle_2$ are precisely  $\sigma_1(XAX^{-1}),\ldots,\sigma_d(XAX^{-1})$. In particular the ratio between the two values of $\sigma_i(A)$ as calculated according to the two distinct inner products is bounded above by $\|X\|\cdot\|X^{-1}\|$ and below by $\|X\|^{-1}\|X^{-1}\|^{-1}$ for all $A \in \GL(V)$. This will be significant in \S\ref{ss:lyap} below, and in general implies that when we are interested in limits of sequences of the form $\frac{1}{n} \log \sigma_i(A_n)$ for some sequence $(A_n)_{n=1}^\infty$ of elements of $\ned(V)$, the value of the limit will be independent of the choice of inner product on $V$ with respect to which the sequence of terms is calculated.

\subsubsection{Proximality} \label{sss:prox}

Let $V$ be a finite-dimensional real vector space equipped with an inner product.  We will call a linear endomorphism $A \in \ned(V)$ \emph{proximal} if it has a unique eigenvalue of maximum modulus and that eigenvalue is simple. (Note that every linear endomorphism of a one-dimensional space is proximal.) We denote the set of all proximal endomorphisms of $V$ by $\prox(V)$. If $A \in \prox(V)$ then we write $V^+(A)$ for the leading eigenspace of $A$ and $V^-(A)$ for the unique $A$-invariant hyperplane which is complementary to $V^+(A)$. Using the perturbation theory of finite-dimensional linear maps it is not difficult to show that $\prox(V)$ is an open subset of $\ned(V)$ and that the functions $V^+$ and $V^-$ are continuous on it. We also note that $A \in \prox(V)$ if and only if $\lambda A \in \prox(V)$ for every nonzero $\lambda\in \mathbb{R}$, if and only if $A^n \in \prox(V)$ for every $n \geq 1$, and that the identities $V^+(A)=V^+(\lambda A) = V^+(A^n)$ and $V^-(A)=V^-(\lambda A) = V^-(A^n)$ are valid for all $A \in \prox(V)$, all nonzero $\lambda \in \mathbb{R}$ and all positive integers $n$. By considering the Jordan form of $A$ it is not difficult to see that $A \in \prox(V)$ if and only if the limit $P:=\lim_{n \to \infty} \|A^n\|^{-1}A^n$ exists, is not nilpotent, and has rank one. Under these conditions the limit $P$ is clearly also proximal and satisfies $V^+(P)=V^+(A)$ and $V^-(P)=V^-(A)$. We observe that if $A \in \ned(V)$ has rank one then $A \in \prox(V)$ if and only if $A^2 \neq 0$, if and only if $A$ is not nilpotent. We lastly remark that if $A \in \ned(V)$ and $\rho(A)^2 >\sigma_1(A)\sigma_2(A)$ then $A$ is necessarily proximal, since in this case by Gelfand's formula
\[\rho(A^{\wedge 2}) = \lim_{n \to \infty} \left\|\left(A^{\wedge 2}\right)^n\right\|^{\frac{1}{n}} = \inf_{n \geq 1} \left\|\left(A^{\wedge 2}\right)^n\right\|^{\frac{1}{n}}  \leq \left\|A^{\wedge 2}\right\|  =\sigma_1(A)\sigma_2(A)<\rho(A)^2,\]
and since $\rho(A^{\wedge 2})$ is the product of the absolute values of the two largest eigenvalues of $A$ this implies that $A$ has a unique, simple eigenvalue with absolute value $\rho(A)$ as required for $A$ to be proximal.

\subsection{Linear algebraic groups}\label{ss:zariski}

In \S\ref{se:thirdstage} we will need to consider the Zariski topology on the general linear group $\GL(V)$ of invertible linear transformations of a finite-dimensional real vector space $V$. We briefly summarise here, without proofs, the definition and important features of this topology which will be needed later. Proofs of the statements described in this section may be found in standard textbooks on linear algebraic groups such as \cite{Bo91,Hu75}; the introductory treatment of this subject in \cite{BeQu16} may be particularly helpful for readers approaching the subject from a background in ergodic theory.

If $V_1$ and $V_2$ are finite-dimensional real vector spaces then a function $p \colon V_1 \to V_2$ is called a \emph{polynomial} if for some (then for every) choice of basis on $V_1$  and $V_2$, the coefficients of the vector $p(v) \in V_2$ with respect to the basis on $V_2$ are consistent polynomial functions of the coefficients of $v$ with respect to the basis on $V_1$. A subset $Z$ of a finite-dimensional real vector space $V$ is called an \emph{affine variety} if it is the common zero locus of some family of polynomial functions $V \to \mathbb{R}$. In particular $V$ itself is an affine variety. When a (proper) subset of an affine variety is also an affine variety we call it a (proper) subvariety. If $Z_1$ and $Z_2$ are affine varieties which are subvarieties of real vector spaces $V_1$ and $V_2$ then we define a polynomial $Z_1 \to Z_2$ to be a function from $Z_1$ to $Z_2$ which can be realised as the restriction to $Z_1$ of a polynomial $V_1 \to V_2$.

The \emph{Zariski topology} on an affine variety $Z$ is defined to be the topology generated by declaring the affine subvarieties of $Z$ to be the closed sets for the topology. The Zariski topology is much coarser than the standard (Euclidean) topology which $Z$ inherits as a subset of its ambient vector space $V$, having far fewer open sets; in particular, it is not a Hausdorff topology.  An affine variety is called an \emph{irreducible variety} if it cannot be written as the union of a finite collection of proper subvarieties. In an irreducible variety, every Zariski open set is dense. One may show that every affine variety is equal to the union of finitely many irreducible subvarieties.

Importantly for our arguments, if $V$ is a finite-dimensional real vector space then $\GL(V)$ may be given the structure of an affine variety by identifying it with the set of all linear operators on $V \oplus \mathbb{R}$ which have the form $A \oplus x$ for some $A \in \ned(V)$ and $x \in \mathbb{R}$ such that $x\cdot (\det A)=1$.  This condition is clearly polynomial and therefore defines an affine subvariety of $\ned(V \oplus \mathbb{R})$. This gives $\GL(V)$ the structure of an affine variety; in this structure a function $p \colon \GL(V) \to \mathbb{R}$ is a polynomial if $p(A)$ is a polynomial function of the matrix entries of $A$ in some basis together with the additional variable $1/\det A$. We note that for every $B \in \GL(V)$ the maps $A \mapsto AB$ and $A \mapsto BA$ are homeomorphisms in the Zariski topology on $\GL(V)$, as is the map $A \mapsto A^{-1}$.

For the purposes of this article a linear algebraic group will be any Zariski-closed subgroup of $\GL(V)$, where $V$ is a finite-dimensional real vector space. Importantly, the Zariski closure of a \emph{subsemigroup} of $\GL(V)$ is always a linear algebraic group. Every linear algebraic group $G \leq \GL(V)$ has only finitely many connected components in the Zariski topology. These components are disjoint and there exists a unique component of $G$ containing the identity, which we denote by  $G^0$ and call the \emph{identity component} of $G$. We note that since every $A \in G$ induces a Zariski homeomorphism of $G$ by left (or right) multiplication, left or right multiplication by $A$ induces a permutation of the connected components of $G$. It is not difficult to show that the identity component of $G$ is a normal subgroup of $G$. If $G_1\leq \GL(V_1)$ and $G_2\leq \GL(V_2)$ are linear algebraic groups, a \emph{regular representation} $\phi \colon G_1 \to G_2$ will be any group homomorphism which is also a polynomial. We call $\phi$ an \emph{irreducible} representation if there is no proper nonzero linear subspace of $V_2$ which is preserved by every element of $\phi(G_1)$.

We finish this section by highlighting for the reader some important instances of Zariski closed sets which will be used in our arguments. If $G \leq \GL(V)$ is a linear algebraic group and $U_1,U_2 \subseteq V$ are linear subspaces then the set $\{A \in G \colon AU_1 = U_2\}$ is Zariski closed, because if $u_1,\ldots,u_k$ is a basis for $U_1$ and $v_1,\ldots,v_\ell$ a basis for $U_2^\perp$ then this set is equal to the intersection of the sets $\{A \in G \colon \langle Au_i, v_j\rangle=0\}$ over all $i=1,\ldots,k$ and $j=1,\ldots,\ell$, which is clearly a subvariety of $G$. Similarly if $v \in V$ is arbitrary then the set $\{A \in G \colon Av \in U_2\}$ is Zariski closed since $A$ belongs to this set if and only if $\langle Av,v_j\rangle=0$ for every $j=1,\ldots,\ell$. We also note that if $B \in \ned(V)$ is arbitrary then the set $\{A \in G \colon (AB)^2=0\}$ is Zariski closed since each of the finitely many entries of the matrix $(AB)^2$ is a polynomial function of the entries of $A$.

\subsection{Lyapunov exponents}\label{ss:lyap}


In all sections of this article except \S\ref{se:thirdstage} we will have frequent need to refer to Lyapunov exponents. Let $V$ be a finite-dimensional real vector space, let $(A_1,\ldots,A_N) \in \GL(V)^N$ and let $\mu \in \mathcal{M}_\sigma(\Sigma_N)$. If $\langle \cdot,\cdot\rangle$ is any inner product on $V$ then we define the \emph{Lyapunov exponents} of $(A_1,\ldots,A_N)$ to be the quantities
\[\lambda_i(A,\mu):=\lim_{n \to \infty}\frac{1}{n} \int \log \sigma_i(A_{x|_n})\,d\mu(x)\]
for $i=1,\ldots,\dim V$. For every $k=1,\ldots,\dim V$ the limit
\begin{equation}\label{eq:rose-hork}\sum_{i=1}^k \lambda_i(A,\mu) = \lim_{n \to \infty}\frac{1}{n} \int \log \prod_{i=1}^k \sigma_i(A_{x|_n})\,d\mu(x)\end{equation}
exists by subadditivity as a consequence of the inequality
\[\prod_{i=1}^k \sigma_i(AB) \leq \left(\prod_{i=1}^k \sigma_i(A)\right) \left(\prod_{i=1}^k \sigma_i(B)\right)\]
noted in \S\ref{sss:singular}, and it follows that the limit in the definition of $\lambda_i(A,\mu)$ is well-defined for every $i=1,\ldots,k$ since it is a difference of two limits of the form \eqref{eq:rose-hork}. It is clear that $\lambda_1(A,\mu) \geq \lambda_2(A,\mu) \geq \cdots \geq \lambda_{\dim V}(A,\mu)$ as a consequence of the corresponding inequality for singular values. We say that $(A_1,\ldots,A_N) \in \GL(V)^N$ has \emph{simple top Lyapunov exponent with respect to $\mu$} if $\lambda_1(A,\mu)>\lambda_2(A,\mu)$.

The Lyapunov exponents are independent of the choice of inner product on $V$ which is used to define the singular values: if $\sigma_i(A)$ and $\hat\sigma_i(A)$ denote the $i^{\mathrm{th}}$ singular value of $A$ calculated using two different inner products on $V$ then as remarked in \S\ref{sss:singular} there is a constant $C>0$ such that $|\log \sigma_i(A) - \log \hat\sigma_i(A)| \leq C$ for all $A \in \GL(V)$, and consequently
\[\lim_{n \to \infty}\frac{1}{n} \int \log \sigma_i(A_{x|_n})\,d\mu(x)= \lim_{n \to \infty}\frac{1}{n} \int \log \hat\sigma_i(A_{x|_n})\,d\mu(x).\]
 In particular we are at liberty to change the inner product on $V$ without affecting the Lyapunov exponents, if there is advantage in doing so. We will take advantage of this principle in \S\ref{se:pf2} below.
 
We lastly note the following useful result which will be applied in \S\ref{se:firststage} and \S\ref{se:pf2}:
\begin{proposition}[\cite{Mo12}]\label{pr:motwelve}
Let $N \geq 2$, let $V$ be a finite-dimensional real vector space, let $(B_1,\ldots,B_N) \in \GL(V)^N$ and let $\mu \in \mathcal{M}_\sigma(\Sigma_N)$ be ergodic. Then
\[\limsup_{n \to \infty} \frac{1}{n}\log \rho\left(B_{x|_n}\right) = \lambda_1(B,\mu)\]
for $\mu$-a.e. $x \in \Sigma_N$.
\end{proposition}
\begin{proof}
The result follows by applying the subadditive ergodic theorem and \cite[Theorem 1.5]{Mo12} to the cocycle $\mathcal{A} \colon \Sigma_N \times \mathbb{N} \to \GL(V)$ defined by $\mathcal{A}(x,n):=B_{x_n}^\top \cdots B_{x_1}^\top = B_{x|_n}^\top$. 
\end{proof}

\subsection{Prior results on generalised matrix equilibrium states}

Throughout this article we will require various facts on the structure of generalised matrix equilibrium states which were established in \cite{BoMo18} and which we collect here for the reader's convenience.

\subsubsection{Subspace classes}

Ergodic generalised matrix equilibrium states were characterised in \cite{BoMo18} via an algebraic object which we now describe.  If $V$ is a finite-dimensional real vector space then for the purposes of this article the Grassmannian of $V$, denoted $\Gr(V)$, is defined to be the set of all nonzero linear subspaces of $V$. Note that $\Gr(V)$ thus defined includes the space $V$ itself. If $1 \leq \ell \leq \dim V$ then we let $\Gr_\ell(V)$ denote the set of all $\ell$-dimensional linear subspaces of $\Gr(V)$.

Let $k \geq 1$, let $V_1,\ldots,V_k$ be finite-dimensional real vector spaces and let $(A_1^{(j)},\ldots,A_N^{(j)}) \in \GL(V_j)^N$ for every $j=1,\ldots,k$. We define a \emph{subspace class} to be any finite nonempty subset of $\prod_{j=1}^k\Gr(V_j)$. We will say that a subspace class $\mathcal{W} \subseteq \prod_{j=1}^k \Gr(V_j)$ is \emph{equivariant} if for every $(W_j)_{j=1}^k \in \mathcal{W}$ we have
\[\left\{(A_\iii^{(j)}W_j)_{j=1}^k \colon \iii \in \Sigma_N^*\right\}\subseteq \mathcal{W},\]
\emph{transitive} if for every $(W_j)_{j=1}^k \in \mathcal{W}$
\[\left\{(A_\iii^{(j)}W_j)_{j=1}^k \colon \iii \in \Sigma_N^*\right\}= \mathcal{W},\]
and \emph{primitive} if for some integer $p \geq 1$ we have for every $(W_j)_{j=1}^k \in \mathcal{W}$
\[\left\{(A_\iii^{(j)}W_j)_{j=1}^k \colon \iii \in \Sigma_N^*\text{ and }|\iii|=p\right\}= \mathcal{W}.\]
In other words $\mathcal{W}$ is equivariant if and only if for every $\iii \in \Sigma_N^*$ and $(W_j)_{j=1}^k \in \mathcal{W}$ the tuple $(A_\iii^{(j)}W_j)_{j=1}^k$ also belongs to $\mathcal{W}$; is transitive if and only if for every  $(W_j)_{j=1}^k,(W_j')_{j=1}^k \in \mathcal{W}$ there exists $\iii \in \Sigma_N^*$ such that $(A_\iii^{(j)}W_j)_{j=1}^k=(W_j')_{j=1}^k$; and is primitive if and only if the word $\iii$ in this definition of transitivity can be chosen so as to have the same length for all choices of $(W_j)_{j=1}^k$ and $(W_j')_{j=1}^k$.  A third description in terms of Perron-Frobenius theory may also be helpful. Suppose that we were to define a non-negative integer matrix $M$, whose rows and columns are indexed by the elements of $\mathcal{W}$, by placing $1$ at the intersection of the row $(W_j)_{j=1}^k$ and the column $(W_j')_{j=1}^k$ if there exists $i \in \{1,\ldots,N\}$ such that $(A_i^{(j)}W_j)_{j=1}^k=(W_j')_{j=1}^k$, and $0$ otherwise. Equivariance of $\mathcal{W}$ ensures that this definition makes sense and implies that every row of $M$ has at least one nonzero entry; transitivity asserts precisely that $M$ is an irreducible matrix in the standard sense of Perron-Frobenius theory; and primitivity asserts that $M$ is a primitive matrix in the sense of Perron-Frobenius theory. 

\subsubsection{Properties of generalised matrix equilibrium states}

The following result from \cite{BoMo18} characterises ergodic generalised matrix equilibrium states  in terms of subspace classes in the case where every $(A_1^{(j)},\ldots,A_N^{(j)})$ is irreducible:

\begin{theorem}[\cite{BoMo18}]\label{th:bomo}
Let $k \geq 1$ and $N \geq 2$ and for each $j=1,\ldots,k$ let $V_j$ be a finite-dimensional real vector space, $(A_1^{(j)},\ldots,A_N^{(j)}) \in \GL(V_j)^N$ an irreducible $N$-tuple of linear maps and $\beta_j>0$ a real number. For each $j=1,\ldots,k$ let $\ell_j \in \{1,\ldots,\dim V_j\}$ be the smallest integer such that there exists a nonzero linear subspace $U_j \subseteq V_j$ which has finite orbit under the action of the semigroup $\{A_\iii^{(j)} \colon \iii \in \Sigma_N^*\}$. Then:

\begin{enumerate}[(i)]
\item
 If $\mathcal{W}\subseteq \prod_{j=1}^k \Gr_{\ell_j}(V)$ is a transitive subspace class, define a potential $\Phi_{\mathcal{W}} \colon  \Sigma_N^* \to (0,+\infty)$ by
\[\Phi_{\mathcal{W}}(\iii):=\max_{(W_j)_{j=1}^k \in \mathcal{W}}\prod_{j=1}^k \left\|A_\iii^{(j)}|_{W_j}\right\|^{\beta_j}.\]
Then $\Phi_{\mathcal{W}}$ is submultiplicative and  quasimultiplicative and has a unique equilibrium state $\mu \in \mathcal{M}_\sigma(\Sigma_N)$. There exists $C>1$ such that
\[C^{-1}\Phi_{\mathcal{W}}(\iii) \leq e^{|\iii|P(\Phi_{\mathcal{W}})}\mu([\iii]) \leq C\Phi_{\mathcal{W}}(\iii)\]
for all $\iii \in \Sigma_N^*$, and in particular $\mu$ is fully supported on $\Sigma_N$.
\item
If we define a potential $\Phi \colon \Sigma_N^* \to (0,+\infty)$ by
\[\Phi(\iii):=\prod_{j=1}^k \left\|A_\iii^{(j)}\right\|^{\beta_j}\]
for all $\iii \in \Sigma_N^*$, then for every ergodic equilibrium state $\mu$ of $\Phi$ there exists a transitive subspace class $\mathcal{W}\subseteq \prod_{j=1}^k \Gr_{\ell_j}(V)$ such that $\mu$  is the unique equilibrium state of the potential $\Phi_{\mathcal{W}}$ defined as in (i). In particular every equilibrium state of $\Phi$ is fully supported on $\Sigma_N$.
\end{enumerate}
\end{theorem}
In general a potential of the form $\Phi(\iii):=\prod_{j=1}^k \|A_\iii^{(j)}\|^{\beta_j}$ may admit multiple ergodic equilibrium states corresponding to different choices of transitive subspace class $\mathcal{W}$: the number of ergodic equilibrium states is bounded above by the quantity $(\prod_{1 \leq j \leq k} \dim V_j ) / (\max_{1 \leq j \leq k} \dim V_j)$  and in at least some situations this bound can be attained, see \cite{BoMo18}. Moreover, in general there may be infinitely many choices of subspace class $\mathcal{W}$ which generate the same equilibrium state. One of the major components of the proof of Theorem \ref{th:main} will be an extension of Theorem \ref{th:bomo}(ii) in \S\ref{se:pf2}. In this result we will show that if additionally every $(A_1^{(j)},\ldots,A_N^{(j)})$ has simple top Lyapunov exponent with respect to $\mu$ then the choice of transitive subspace class $\mathcal{W}$ is unique, and if $\mu$ is totally ergodic then moreover $\mathcal{W}$ must be primitive. 

Outside the irreducible case, the following additional result of \cite{BoMo18} implies that every ergodic generalised matrix equilibrium state can be expressed as the equilibrium state of a potential satisfying the hypotheses of Theorem \ref{th:bomo}, and will also be applied in the following section. In order to prove Theorem \ref{th:main} we will likewise need to extend the below result to include a statement on simple top Lyapunov exponents.
\begin{theorem}[\cite{BoMo18}]\label{th:bomo2}
Let $k \geq 1$ and $N \geq 2$ and for each $j=1,\ldots,k$ let $V_j$ be a finite-dimensional real vector space, $(A_1^{(j)},\ldots,A_N^{(j)}) \in \GL(V_j)^N$ an $N$-tuple of linear maps and $\beta_j>0$ a real number. Then for each $j=1,\ldots,k$ there exist an integer $r_j \geq 1$, integers $d_{j,1},\ldots,d_{j,r_j} \geq 1$ satisfying $\sum_{t=1}^{r_j} d_{j,t}=\dim V_j$ and a basis for $V_j$ in which we may write 
\[A_i^{(j)} = \begin{pmatrix} A_i^{(j,1)} & *& \cdots & * & * \\ 
0 & A_i^{(j,2)}& \cdots & * & * \\ 
\vdots & \vdots & \ddots   & \vdots & \vdots \\
 0 & 0 &\cdots & A_i^{(j,r_j-1)} & * \\
  0 & 0 & \cdots&0 & A_i^{(j,r_j)}
\end{pmatrix}\]
for every $i=1,\ldots,N$, where $(A_1^{(j,t)},\ldots,A_N^{(j,t)}) \in \GL_{d_{j,t}}(\mathbb{R})^N$ is irreducible for every $t=1,\ldots,r_j$. If $\mu$ is an ergodic equilibrium state of the potential $\Phi \colon \Sigma_N^* \to (0,+\infty)$ defined by
\[\Phi(\iii):=\prod_{j=1}^k \left\|A_\iii^{(j)}\right\|^{\beta_j},\]
then there exist $t_1,\ldots,t_k$ such that $1 \leq t_j \leq r_j$ for every $j=1,\ldots,k$ and such that $\mu$ is an equilibrium state of the potential $\hat\Phi \colon \Sigma_N^* \to (0,+\infty)$ defined by
\[\hat\Phi(\iii):=\prod_{j=1}^k \left\|A_\iii^{(j,t_j)}\right\|^{\beta_j}\]
and satisfies $P(\Phi)=P(\hat\Phi)$. Furthermore, the number of distinct ergodic equilibrium states of $\Phi$ is not greater than $\prod_{j=1}^k \dim V_j$.
\end{theorem}
\emph{Remark.} Theorem \ref{th:bomo2} follows from the statement of \cite[Theorem 5]{BoMo18} except for the fact that $P(\Phi)=P(\hat\Phi)$, which is not made explicit in the statement of that theorem but appears in the theorem's proof. Similarly, in Theorem \ref{th:bomo} the fact that every potential of the form $\Phi_{\mathcal{W}}$ where $\mathcal{W} \subseteq \prod_{j=1}^k \Gr_{\ell_j}(V_j)$ is a transitive subspace class is quasimultiplicative and satisfies a Gibbs inequality was not made explicit in the statement of \cite[Theorem 4]{BoMo18} but features prominently in the proof.




\section{Reduction to the case of simple top Lyapunov exponents}\label{se:firststage}

As was described in the introduction the first, and by far the simplest, step in the proof of Theorem \ref{th:main} is to reduce the problem to the case where the generalised matrix equilibrium state $\mu$ is defined by tuples $(A_1^{(j)},\ldots,A_N^{(j)}) \in \GL(V_j)^N$ which are all irreducible and all have simple top Lyapunov exponent with respect to $\mu$. In this section we prove:
\begin{theorem}\label{th:reducks}
Let $N \geq 2$ and let $\mu \in \mathcal{M}_\sigma(\Sigma_N)$ be an ergodic generalised matrix equilibrium state. Then there exist finite-dimensional real vector spaces $V_j$, irreducible tuples $(A_1^{(j)},\ldots,A_N^{(j)}) \in \GL(V_j)^N$ each having simple top Lyapunov exponent with respect to $\mu$, and real numbers $\beta_j>0$ for each $j=1,\ldots,k$ such that $\mu$ is an ergodic equilibrium state of the potential 
\[\Phi(\iii):=\prod_{j=1}^k \left\|A_\iii^{(j)}\right\|^{\beta_j}.\]
\end{theorem}
By definition $\mu$ admits at least one representation in the above form but without each $(B_1^{(j)},\ldots,B_N^{(j)})$ necessarily being irreducible or having simple top Lyapunov exponent. The result is proved by starting with such a representation, passing to an appropriate exterior power for each $j$ and then finding a block upper triangularisation of each tuple such that for each $j$ one of the tuples of diagonal blocks yields the desired new tuple. We separate the first part of this argument into a proposition as follows:

\begin{proposition}\label{pr:ittstick}
Let $N \geq 2$ and let $\mu \in \mathcal{M}_\sigma(\Sigma_N)$ be a generalised matrix equilibrium state. Then there exist finite-dimensional real vector spaces $U_j$, tuples of linear maps $(B_1^{(j)},\ldots,B_N^{(j)}) \in \GL(U_j)^N$ each having simple top Lyapunov exponent with respect to $\mu$, and real numbers $\gamma_j>0$ for each $j=1,\ldots,k$ such that $\mu$ is an ergodic equilibrium state of the potential 
\[\Psi(\iii):=\prod_{j=1}^k \left\|B_\iii^{(j)}\right\|^{\gamma_j}.\]
\end{proposition}
\begin{proof}
Since $\mu$ is a generalised matrix equilibrium state, by definition there exist $k \geq 1$, finite-dimensional real vector spaces $V_j$, tuples $(A_1^{(j)},\ldots,A_N^{(j)}) \in \GL(V_j)^N$ and real numbers $\beta_j>0$ for each $j=1,\ldots,k$ such that $\mu$ is an equilibrium state for the potential $\Phi \colon \Sigma_N\to(0,+\infty)$ defined by
\[\Phi(\iii):=\prod_{j=1}^k \left\|A_\iii^{(j)}\right\|^{\beta_j}.\]
Choose for each $j=1,\ldots,k$ the largest integer $\ell_j \in \{1,\ldots,\dim V_j\}$ such that $\lambda_1(A^{(j)},\mu)=\cdots =\lambda_{\ell_j}(A^{(j)},\mu)$. Define $U_j:=V_j^{\wedge \ell_j}$ for each $j=1,\ldots,k$ and $B_i^{(j)}:=(A_i^{(j)})^{\wedge \ell_j}$  for each $i=1,\ldots,N$ and $j=1,\ldots,k$. For every $A \in \GL(V_j)$ the singular values of $A^{\wedge \ell_j}$ are precisely the numbers $\sigma_{i_1}(A)\cdots \sigma_{i_{\ell_j}}(A)$ such that $1 \leq i_1<\cdots <i_{\ell_j} \leq \dim V_j$, listed in decreasing order, so in particular the largest singular value is $\sigma_1(A)\cdots \sigma_{\ell_j}(A)$ and the second-largest singular value is $\sigma_1(A)\cdots \sigma_{\ell_j-1}(A)\sigma_{\ell_j+1}(A)$ if $\ell_j<\dim V_j$ and zero otherwise. It follows directly that 
\begin{align*}\lambda_1(B^{(j)},\mu) &=  \sum_{i=1}^{\ell_j} \lambda_i(A^{(j)},\mu),\\
\lambda_2(B^{(j)},\mu) &= \left(\sum_{i=1}^{\ell_j-1} \lambda_i(A^{(j)},\mu)\right) + \lambda_{\ell_j+1}(A^{(j)},\mu)\end{align*}
for each $j=1,\ldots,k$, where $\lambda_{1+\dim V_j }(A^{(j)},\mu)$ is interpreted as being equal to $-\infty$. By the maximality of each $\ell_j$ we have $ \lambda_{\ell_j+1}(A^{(j)},\mu)< \lambda_{\ell_j}(A^{(j)},\mu)$ for each $j=1,\ldots,k$ and consequently $\lambda_1(B^{(j)},\mu)>\lambda_2(B^{(j)},\mu)$ for every $j=1,\ldots,k$ so that every $(B_1^{(j)},\ldots,B_N^{(j)})$ has simple top Lyapunov exponent with respect to $\mu$. Define $\gamma_j:=\beta_j/\ell_j$ for each $j=1,\ldots,k$. We claim that $\mu$ is an equilibrium state for the potential $\Psi \colon \Sigma_N^* \to (0,+\infty)$ defined by
\[\Psi(\iii):=\prod_{j=1}^k \left\|{B}_\iii^{(j)}\right\|^{\gamma_j}.\]
We have
\begin{align*}\Psi(\iii):=\prod_{j=1}^k \left\|B_\iii^{(j)}\right\|^{\gamma_j}&=\prod_{j=1}^k \sigma_1\left(B_\iii^{(j)}\right)^{\gamma_j}\\
&=\prod_{j=1}^k \left(\sigma_1\left(A_\iii^{(j)}\right)\cdots  \sigma_{\ell_j}\left(A_\iii^{(j)}\right)\right)^{\frac{\beta_j}{\ell_j}}\leq \prod_{j=1}^k \sigma_1\left(A_\iii^{(j)}\right)^{\beta_j}=\Phi(\iii)\end{align*}
for every $\iii \in \Sigma_N^*$ so in particular $P(\Psi) \leq P(\Phi)$. Thus
\begin{align*}P(\Psi) \geq h(\mu)+\Lambda(\Psi,\mu) &= h(\mu) + \sum_{j=1}^k \gamma_j\lambda_1(B^{(j)},\mu)\\
& = h(\mu) + \sum_{j=1}^k \frac{\beta_j}{\ell_j} \sum_{i=1}^{\ell_j}\lambda_i(A^{(j)},\mu)\\
& = h(\mu) + \sum_{j=1}^k \beta_j \lambda_1(A^{(j)},\mu)\\
&=  h(\mu)+\Lambda(\Phi,\mu)=P(\Phi)\geq P(\Psi)\end{align*}
where we have used the definition of $\ell_j$ in the third equation, and $\mu$ is an equilibrium state of $\Psi$ as claimed.\end{proof}

\begin{proof}[Proof of Theorem \ref{th:reducks}]
By Proposition \ref{pr:ittstick} we may assume without loss of generality that there exist finite-dimensional real vector spaces $U_j$, tuples of linear maps $(B_1^{(j)},\ldots,B_N^{(j)}) \in \GL(U_j)^N$ each having simple top Lyapunov exponent with respect to $\mu$, and real numbers $\gamma_j>0$ for each $j=1,\ldots,k$ such that $\mu$ is an ergodic equilibrium state of the potential 
\[\Psi(\iii):=\prod_{j=1}^k \left\|B_\iii^{(j)}\right\|^{\gamma_j}.\]
By Theorem \ref{th:bomo2} there exist integers $r_1,\ldots,r_k \geq 1$ such that in a suitable basis for each $U_j$ we may write
\[B_i^{(j)} = \begin{pmatrix} B_i^{(j,1)} & *& \cdots & * & * \\ 
0 & B_i^{(j,2)}& \cdots & * & * \\ 
\vdots & \vdots & \ddots   & \vdots & \vdots \\
 0 & 0 &\cdots & B_i^{(j,r_j-1)} & * \\
  0 & 0 & \cdots&0 & B_i^{(j,r_j)}
\end{pmatrix}\]
for every $i=1,\ldots,N$, where $(B_1^{(j,t)},\ldots,B_N^{(j,t)}) \in \GL_{d_{j,t}}(\mathbb{R})^N$ is irreducible for every $t=1,\ldots,r_j$, and where for some choice of integers $t_1,\ldots,t_k$ satisfying $1 \leq t_j \leq r_j$ for every $j=1,\ldots,k$ the measure $\mu$ is an equilibrium state for the potential 
\[\Phi(\iii):=\prod_{j=1}^k \left\|B_\iii^{(j,t_j)}\right\|^{\gamma_j}\]
and addtionally $P(\Phi)=P(\Psi)$. Define $V_j:=\mathbb{R}^{d_{j,t_j}}$, $(A_1^{(j)},\ldots,A_N^{(j)}):=(B_1^{(j,t_j)},\ldots,B_N^{(j,t_j)}) \in \GL(V_j)^N$ and $\beta_j:=\gamma_j$ for every $j=1,\ldots,k$. To complete the proof of the theorem we must verify that every $(B_1^{(j,t_j)},\ldots,B_N^{(j,t_j)})$ has simple top Lyapunov exponent with respect to $\mu$. We will show that $(B_1^{(j,t_j)},\ldots,B_N^{(j,t_j)})$ inherits this property from $(B_1^{(j)},\ldots,B_N^{(j)})$.

To this end we first claim that $\lambda_1(B^{(j,t_j)},\mu)=\lambda_1(B^{(j)},\mu)$ for every $j=1,\ldots,k$. On the one hand we have by definition
\begin{align*}\lambda_1(B^{(j,t_j)},\mu) &= \lim_{n \to \infty} \frac{1}{n}\int_{\Sigma_N} \log \left\|B_{x|_n}^{(j,t_j)}\right\| \,d\mu(x)\\
& \leq  \lim_{n \to \infty} \frac{1}{n}\int_{\Sigma_N} \log \left\|B_{x|_n}^{(j)}\right\| \,d\mu(x) =\lambda_1(B^{(j)},\mu)  \end{align*}
so that $\lambda_1(B^{(j,t_j)},\mu)\leq \lambda_1(B^{(j)},\mu)$ for every $j=1,\ldots,k$. On the other hand we may apply this estimate to obtain
\begin{align*}P(\Psi) =P(\Phi) &= h(\mu) + \Lambda(\Phi,\mu)\\
& = h(\mu) + \sum_{j=1}^k \gamma_j \lambda_1(B^{(j,t_j)},\mu)\\
& \leq h(\mu) + \sum_{j=1}^k \gamma_j \lambda_1(B^{(j)},\mu)= h(\mu) + \Lambda(\Psi,\mu) = P(\Psi)\end{align*}
and therefore $\sum_{j=1}^k \gamma_j \lambda_1(B^{(j,t_j)},\mu)= \sum_{j=1}^k \gamma_j \lambda_1(B^{(j)},\mu)$. It follows that
\[\sum_{j=1}^k \gamma_j \left(\lambda_1(B^{(j)},\mu)- \lambda_1(B^{(j,t_j)},\mu)\right) \]
is a sum of non-negative terms which is equal to zero, so all of the summands must be zero and therefore $\lambda_1(B^{(j,t_j)},\mu)=\lambda_1(B^{(j)},\mu)$ for every $j=1,\ldots,k$ as claimed.

We secondly claim that
\[\lambda_1(B^{(j,t_j)},\mu)+ \lambda_2(B^{(j,t_j)},\mu)\leq \lambda_1(B^{(j)},\mu)+\lambda_2(B^{(j)},\mu)\]
for every $j=1,\ldots,k$. Indeed, for fixed $j$ we have for $\mu$-a.e. $x \in \Sigma_N$
\[\lambda_1(B^{(j,t_j)},\mu)+ \lambda_2(B^{(j,t_j)},\mu) =\lambda_1\left(\left(B^{(j,t_j)}\right)^{\wedge 2},\mu\right)= \limsup_{n \to \infty} \frac{1}{n}\log \rho\left(\left(B^{(j,t_j)}_{x|_n}\right)^{\wedge 2}\right)\]
and
\[\lambda_1(B^{(j)},\mu)+ \lambda_2(B^{(j)},\mu) =\lambda_1\left(\left(B^{(j)}\right)^{\wedge 2},\mu\right)= \limsup_{n \to \infty} \frac{1}{n}\log \rho\left(\left(B^{(j)}_{x|_n}\right)^{\wedge 2}\right)\]
by Proposition \ref{pr:motwelve}. For every $x$ and $n$ the quantity $\rho((B^{(j,t_j)}_{x|_n})^{\wedge 2})$ is the product of the absolute values of the two largest eigenvalues of $B^{(j,t_j)}_{x|_n}$ and similarly $\rho((B^{(j)}_{x|_n})^{\wedge 2})$ is the product of the absolute values of the two largest eigenvalues of $B^{(j)}_{x|_n}$. But the set of eigenvalues of $B^{(j,t_j)}_{x|_n}$ is a subset of the set of eigenvalues of $B^{(j)}_{x|_n}$ so we have 
\[\limsup_{n \to \infty} \frac{1}{n}\log \rho\left(\left(B^{(j,t_j)}_{x|_n}\right)^{\wedge 2}\right)   \leq \limsup_{n \to \infty} \frac{1}{n}\log \rho\left(\left(B^{(j)}_{x|_n}\right)^{\wedge 2}\right) \]
for all $x \in \Sigma_N$. The claim follows.

We may now show that every $(B_1^{(j,t_j)},\ldots,B_N^{(j,t_j)})$ has simple top Lyapunov exponent with respect to $\mu$. Combining the two claims we deduce that
\[\lambda_2(B^{(j,t_j)},\mu) \leq \lambda_2(B^{(j)},\mu)\]
and consequently
\[\lambda_2(B^{(j,t_j)},\mu) \leq \lambda_2(B^{(j)},\mu)< \lambda_1(B^{(j)},\mu) = \lambda_1(B^{(j,t_j)},\mu)\]
for every $j=1,\ldots,k$ so that every $(B_1^{(j,t_j)},\ldots,B_N^{(j,t_j)})$ has simple top Lyapunov exponent with respect to $\mu$ as required. The proof of the theorem is complete.
\end{proof}




\section{Generalised matrix equilibrium states in the case of simple top Lyapunov exponents}\label{se:pf2}

The result of the previous section shows that every ergodic generalised matrix equilibrium state $\mu$ can be assumed without loss of generality to be generated by a potential $\Phi(\iii):=\prod_{j=1}^k\|A_\iii^{(j)}\|^{\beta_j}$ defined in terms of tuples $(A_1^{(j)},\ldots,A_N^{(j)}) \in \GL(V_j)$ which are all irreducible and have simple top Lyapunov exponent with respect to $\mu$. By Theorem \ref{th:bomo} irreducibility and ergodicity together imply that $\mu$ is the equilibrium state of a potential $\Phi_{\mathcal{W}}(\iii):=\max_{(W_j)_{j=1}^k\in \mathcal{W}}\prod_{j=1}^k\|A_\iii^{(j)}|_{W_j}\|^{\beta_j}$ for some transitive subspace class $\mathcal{W}\subseteq \prod_{j=1}^k \Gr(V_j)$. In this section we extend Theorem \ref{th:bomo} by showing that if $\mu$ is totally ergodic then irreducibility and the simplicity of Lyapunov exponents allow us to choose $\mathcal{W}$ so as to additionally be primitive. The same arguments which yield this result incidentally provide a  pivotal technical lemma on simultaneously proximal words which recalls some results of Abels, Margulis and Soifer (\cite{AbMaSo95}, for related results see also \cite[\S6]{BeQu16}), although in those works an algebraic rather than analytic method is used. In this section we prove:
\begin{theorem}\label{th:prox-unique}
Let $k \geq 1$ and $N \geq 2$. For each $j=1,\ldots,k$ let $V_j$ be a finite-dimensional real vector space, $(A_1^{(j)},\ldots,A_N^{(j)}) \in \GL(V_j)^N$ an irreducible $N$-tuple and $\beta_j>0$ a real number. For each $j=1,\ldots,k$ let $\ell_j \in \{1,\ldots,\dim V_j\}$ be the smallest dimension of any nonzero linear subspace of $V_j$ which has finite orbit under the action of $(A_1^{(j)},\ldots,A_N^{(j)})$. Suppose that $\mu \in \mathcal{M}_\sigma(\Sigma_N)$ is an ergodic equilibrium state of the potential $\Phi\colon \Sigma_N^* \to (0,+\infty)$ defined by
\[\Phi(\iii):=\prod_{j=1}^k \left\|A_\iii^{(j)}\right\|^{\beta_j}\]
such that for every $j=1,\ldots,k$ the tuple $(A_1^{(j)},\ldots,A_N^{(j)})$ has a simple top Lyapunov exponent with respect to $\mu$. Then:
\begin{enumerate}[(i)]
\item
There exists a \emph{unique} transitive subspace class $\mathcal{W}\subseteq \prod_{j=1}^k \Gr_{\ell_j}(V_j)$ such that $\mu$ is an equilibrium state of the potential $\Phi_{\mathcal{W}} \colon \Sigma_N^* \to (0,+\infty)$ defined by
\[\Phi_{\mathcal{W}}(\iii):=\max_{(W_j)_{j=1}^k \in \mathcal{W}} \prod_{j=1}^k \left\|A_\iii^{(j)}|_{W_j}\right\|^{\beta_j}.\]
\item
The transitive subspace class $\mathcal{W}$ defined in (i) has the following additional property: there exist $\kkk \in \Sigma_N^*$ and $(W_j^0)_{j=1}^k \in \mathcal{W}$ such that for every $j=1,\ldots,k$ we have $A_\kkk^{(j)}W_j^0=W_j^0$ and $A_\kkk^{(j)}|_{W_j^0} \in \prox(W_j^0)$.
\item
If $\mu$ is totally ergodic then the transitive subspace class $\mathcal{W}$ defined in (i) is primitive.
\end{enumerate}
\end{theorem}
Theorem \ref{th:prox-unique}(i) is not used in this article, but is provided for interest. This result contrasts strongly with the situation in which simple top Lyapunov exponents are not assumed, where it can be the case that uncountably many choices of $\mathcal{W}$ exist which all have the same equilibrium state: a simple example of this, as mentioned in \cite{BoMo18}, is the case in which $k=1$, $V=\mathbb{R}^2$ and every $A_i$ is a rational rotation matrix, in which case $\mathcal{W}$ can be taken to be the orbit of any one-dimensional subspace whatsoever. Theorem \ref{th:prox-unique}(ii) and (iii) both contribute to the proof of Theorem \ref{th:main}; while the former has more of the character of a lemma than of a main result, we include it in Theorem \ref{th:prox-unique} since it emerges directly from the proof of the other clauses of the theorem. 

The proof of Theorem \ref{th:prox-unique} is heavily inclined towards ergodic theory, unlike the results of the following section which are essentially algebraic. The essential idea of the proof is as follows. The separation of Lyapunov exponents allows us fairly easily to construct words $\jjj \in \Sigma_N^*$ such that $A_\jjj^{(j)} \in \prox(V_j)$ for all $j=1,\ldots,k$, and the existence of these words implies that each $V_j$ can be written as a splitting $V_j = \bigoplus_{i=1}^{r_j} U_j^{i}$ into $\ell_j$-dimensional spaces in such a way that for each $j$ the spaces $U_j^1,\ldots,U_j^{r_j}$ are permuted by the linear maps $A_i^{(j)}$. By replacing the inner product on each $V_i$ if necessary this splitting can without loss of generality be taken to be orthogonal and it follows that the norm of the restriction of each product $A_\iii^{(j)}$ to each of the various spaces $U_j^i$ is a singular value of $A_\iii^{(j)}$. Together with the separation of Lyapunov exponents this implies that for each $j=1,\ldots,k$ the growth of $\|A_{x|_n}^{(j)}\|$ for almost every $x$ is concentrated on a single one of the subspaces $U_j^i$, which in general will depend on both $x$ and $n$. Using this observation we construct measurable functions $\mathsf{U}_j \colon \hat\Sigma_N \to \{U_j^1,\ldots,U_j^{r_j}\}$ which are equivariant with respect to the action of $A_{x|_n}^{(j)}$ and capture the maximal growth of each $\|A_{x|_n}^{(j)}\|$ in the sense that
\begin{align*}\lim_{n \to \infty}\frac{1}{n}\log\left\|A_{x|_n}^{(j)}|_{\mathsf{U}_j(\hat\sigma^nx)}\right\| &= \lambda_1(A^{(j)},\mu),\\\nonumber
\lim_{n \to \infty}\frac{1}{n}\log\left\|A_{x|_n}^{(j)}|_{\mathsf{U}_j(\hat\sigma^nx)^\perp}\right\| &\leq \lambda_2(A^{(j)},\mu)\end{align*}
for $\hat\mu$-a.e. $x \in \hat\Sigma_N$. (These functions may be thought of as resembling Oseledets spaces, although they do not always correspond precisely to Oseledets spaces and it is not clear whether they are necessarily given by direct sums of Oseledets spaces either.) The ergodicity of $\hat\mu$ implies that the tuple $(\mathsf{U}_j(x))_{j=1}^k$ belongs almost everywhere to a single transitive subspace class $\mathcal{W}$; if $\hat\mu$ is also totally ergodic, this can be applied to show that for every integer $n \geq 1$, every two values taken by the tuple $(\mathsf{U}_j(x))_{j=1}^k$ on sets of positive measure can be linked by a word $\iii$ with length divisible by $n$. This allows us to show that 
\[\left\{ (A_\iii^{(j)}W_j)_{j=1}^k \colon \iii \in \Sigma_N^*\text{ and }n\text{ divides }|\iii|\right\} =\mathcal{W}\]
for every $(W_j)_{j=1}^k \in \mathcal{W}$ and $n \geq 1$, and this is sufficient to deduce that $\mathcal{W}$ is primitive in the totally ergodic case. 

Before beginning the proof we require three lemmas, one of which is primarily ergodic-theoretic in character, one primarily algebraic and one somewhat more combinatorial; these respectively treat the existence of proximal elements, the algebraic consequences of their existence for splittings of each $V_j$, and the criterion for primitivity just mentioned. 
\begin{lemma}\label{le:poximal1}
Let $k \geq 1$ and $N \geq 2$ and let $\mu \in \mathcal{M}_\sigma(\Sigma_N)$ be ergodic. For each $j=1,\ldots,k$ let $V_j$ be a finite-dimensional real inner product space and $(A_1^{(j)},\ldots,A_N^{(j)}) \in \GL(V_j)^N$ a tuple which has simple top Lyapunov exponent with respect to $\mu$. Then for $\hat\mu$-a.e. $x \in \hat\Sigma_N$ we have
\begin{equation}\label{eq:turdly}\limsup_{n \to \infty}\min_{1 \leq j \leq k} \left(\frac{1}{n}\log\rho\left(A_{x|_n}^{(j)}\right) -\frac{1}{2n}\log \sigma_1\left(A_{x|_n}^{(j)}\right)\sigma_2\left(A_{x|_n}^{(j)}\right)\right)>0.\end{equation}
Furthermore there exists $\kkk \in \Sigma_N^*$ such that for every $j=1,\ldots,k$ we have  $A_\kkk^{(j)} \in \prox(V_j)$.
\end{lemma}
\begin{proof}
The inequality \eqref{eq:turdly} implies that we may find $\kkk:=x|_n \in \Sigma_N^*$ such that $\rho(A_\kkk^{(j)})^2>\sigma_1(A_\kkk^{(j)})\sigma_2(A_\kkk^{(j)})$ for all $j=1,\ldots,k$ and as noted in \S\ref{sss:prox} this yields $A_\kkk^{(j)} \in \prox(V_j)$ for every $j=1,\ldots,k$. To prove the lemma it is thus sufficient to show that \eqref{eq:turdly} holds for $\hat\mu$-a.e. $x \in \hat\Sigma_N$ and this clearly follows if the same result is shown for $\mu$-a.e. $x\in \Sigma_N$. We therefore proceed to prove the latter statement.

We first claim that for $\mu$-a.e. $x \in \Sigma_N$
\begin{equation}\label{eq:stummy-beige}\limsup_{n \to \infty}  \sum_{j=1}^k \left(\frac{1}{n}\log \rho\left(A_{x|_n}^{(j)}\right)-\frac{1}{n}\log\left\|A_{x|_n}^{(j)}\right\|\right)=0.\end{equation}
To prove this let $V:=\bigotimes_{j=1}^k V_j$ be the tensor product of the vector spaces $V_j$ and define a tuple $(B_1,\ldots,B_N) \in \GL(V)^N$ by $B_i:=\bigotimes_{j=1}^k A_i^{(j)}$ for each $i=1,\ldots,N$. We equip $V$ with the inner product induced by the inner products on the spaces $V_j$. As noted in \S\ref{sss:la} we have $\|B_\iii\|=\prod_{j=1}^k \|A_\iii^{(j)}\|$ for all $\iii \in \Sigma_N^*$ with respect to the corresponding inner product norms and also $\rho(B_\iii)=\prod_{j=1}^k \rho(A_\iii^{(j)})$ for all $\iii \in \Sigma_N^*$. It follows from the first statement in particular that  $\lambda_1(B,\mu)=\sum_{j=1}^k \lambda_1(A^{(j)},\mu)$. Since by the subadditive ergodic theorem
\[ \lim_{n \to \infty} \sum_{j=1}^k \frac{1}{n}\log\left\|A_{x|_n}^{(j)}\right\|= \lim_{n \to \infty} \frac{1}{n}\log \left\|B_{x|_n}\right\| = \lambda_1(B,\mu)\]
for $\mu$-a.e. $x \in \Sigma_N$, and by Proposition \ref{pr:motwelve}
\[\limsup_{n \to \infty}  \sum_{j=1}^k \frac{1}{n}\log \rho\left(A_{x|_n}^{(j)}\right) = \limsup_{n \to \infty} \frac{1}{n}\log \rho\left(B_{x|_n}\right) = \lambda_1(B,\mu)\]
for $\mu$-a.e. $x \in \Sigma_N$, the claimed result \eqref{eq:stummy-beige} follows. Now, each of the summands $\frac{1}{n}\log \rho\left(A_{x|_n}^{(j)}\right)-\frac{1}{n}\log\left\|A_{x|_n}^{(j)}\right\|$ in \eqref{eq:stummy-beige} is non-positive since clearly $\rho(A_\iii^{(j)}) \leq \|A_\iii^{(j)}\|$ for every $\iii \in \Sigma_N^*$, and it follows from this observation that for $\mu$-a.e. $x \in \Sigma_N$
\begin{equation}\label{eq:ronching-blue}\limsup_{n \to \infty}  \min_{1 \leq j \leq k}  \left(\frac{1}{n}\log \rho\left(A_{x|_n}^{(j)}\right)-\frac{1}{n}\log\left\|A_{x|_n}^{(j)}\right\|\right)=0\end{equation}
since this sequence is bounded below by the sequence in \eqref{eq:stummy-beige} and is bounded above by $0$. But for $\mu$-a.e. $x \in \Sigma_N$ we have
\begin{eqnarray}\label{eq:stanky-bean}{\lefteqn{\lim_{n \to \infty} \min_{1 \leq j \leq k}\left(\frac{1}{n}\log\left\|A_{x|_n}^{(j)}\right\| - \frac{1}{2n}\log \sigma_1(A_{x|_n}^{(j)})\sigma_2(A_{x|_n})\right)}}& & \\\nonumber
&=&\min_{1 \leq j\leq k} \left(\frac{1}{2}\lambda_1(A^{(j)},\mu) - \frac{1}{2}\lambda_2(A^{(j)},\mu)\right)>0\end{eqnarray}
because by the subadditive ergodic theorem
\[\lim_{n\to \infty} \frac{1}{n}\log \|A_{x|_n}^{(j)}\| = \lambda_1(A^{(j)},\mu)\]
and
\[\lim_{n \to \infty} \frac{1}{n} \log \sigma_1\left(A_{x|_n}^{(j)}\right)\sigma_2\left(A_{x|_n}^{(j)}\right) = \lambda_1\left(A^{(j)},\mu\right)+\lambda_2\left(A^{(j)},\mu\right)\]
for every $j=1,\ldots,k$, and each of the differences $\lambda_1(A^{(j)},\mu) - \lambda_2(A^{(j)},\mu)$ is positive by the hypothesis of simplicity of the top Lyapunov exponent. Adding the sequences in \eqref{eq:ronching-blue} and \eqref{eq:stanky-bean} gives a sequence which is a lower bound for the left-hand side of \eqref{eq:turdly} and has positive limit superior almost everywhere. This yields \eqref{eq:turdly} for $\mu$-a.e. $x \in \Sigma_N$ as required and the result follows.
\end{proof}
The following algebraic consequence of proximality is perhaps already known but we have been unable to find a reference:
\begin{lemma}\label{le:poximal2}
Let $V$ be a finite-dimensional real vector space, let $(B_1,\ldots,B_N) \in \GL(V)^N$ be irreducible and let $\ell  \in \{1,\ldots,\dim V\}$ be the smallest integer such that there exists an $\ell$-dimensional linear subspace $U \subseteq V$ with finite orbit under the action of $(B_1,\ldots,B_N)$. Suppose that there exists $\kkk \in \Sigma_N^*$ such that $B_\kkk \in \prox(V)$. Then $\dim V = \ell r$ for some integer $r \geq 1$, there are exactly $r$ subspaces $U_1,\ldots,U_r \in \Gr_\ell(V)$ which have finite orbit under the action of $(B_1,\ldots,B_N)$, and these subspaces form a direct sum $\bigoplus_{t=1}^r U_t$ which is equal to $V$. 
\end{lemma}
\begin{proof}
Let $U \subseteq V$ be an $\ell$-dimensional subspace with finite orbit under $(B_1,\ldots,B_N)$ and let $\mathcal{U}:=\{B_\iii U \colon \iii \in \Sigma_N^*\}$ denote its orbit. We claim that $(B_1,\ldots,B_N)$ acts transitively on $\mathcal{U}$ in the following sense: for every $U_1,U_2 \in \mathcal{U}$ there exists $\iii \in \Sigma_N^*$ such that $B_\iii U_1 = U_2$.

To see this we argue as follows. For every $\iii \in \Sigma_N^*$ the map $\Gr_\ell(V) \to \Gr_\ell(V)$ defined by $W \mapsto B_\iii W$ is clearly bijective, so in particular each $B_\iii$ induces an injective map from $\{U\} \cup \mathcal{U}$ to $\mathcal{U}$. Since $\mathcal{U}$ is finite this is only possible if $\{U\} \cup \mathcal{U}=\mathcal{U}$ and thus necessarily $U \in \mathcal{U}$. If $\iii \in \Sigma_N^*$ is arbitrary then the map $B_\iii \colon \mathcal{U} \to \mathcal{U}$ is a bijection of a finite set, so in particular some power of that map is the identity permutation. In particular if $\iii \in \Sigma_N^*$ is arbitrary then $B_\iii^pU=U$ for some integer $p \geq 1$ and hence also $B_\iii^{np}U=U$ for every $n \geq 1$. It follows that if $U_1 \in \mathcal{U}$ is arbitrary then since by definition $U_1=B_\jjj U$ for some $\jjj \in \Sigma_N^*$, the set $\{B_\iii U_1 \colon \iii \in \Sigma_N^*\}$ contains $B_\jjj^n U$ for every $n \geq 2$ and therefore contains $U$ and hence contains $\{B_\iii U \colon \iii \in \Sigma_N^*\}$. But the reverse inclusion is obvious, so $\{B_\iii U_1 \colon \iii \in \Sigma_N^*\}$ is equal to $\mathcal{U}$ for every $U_1 \in \mathcal{U}$ and this proves the result claimed.

We next claim that there exist $r \geq 1$ and subspaces $U_1,\ldots,U_r\in \mathcal{U}$ such that $V=\bigoplus_{t=1}^r U_t$. To see this we observe that there exists at least one list of elements $U_1,\ldots,U_m \in \mathcal{U}$ which forms a direct sum, namely the list of length $1$ consisting of the single space $U$. There therefore exists a \emph{maximal} list of elements $U_1,\ldots,U_r \in \mathcal{U}$ which forms a direct sum. If $U' \in \mathcal{U}$ is arbitrary then the subspace $(U_1 \oplus \cdots \oplus U_r) \cap U'$ cannot have dimension zero since then $U_1,\ldots,U_r,U'$ would form a direct sum, contradicting maximality. It also cannot have dimension nonzero but strictly less than $\ell$, since it clearly has finite orbit under the action of $(B_1,\ldots,B_N)$ and such a space must have dimension at least $\ell$. It cannot have dimension strictly greater than $\ell$ since it is a subset of $U'$, so by elimination its dimension is precisely $\ell$ and therefore it is equal to the subspace $U'$. Since $U' \in \mathcal{U}$ was arbitrary this shows that every element of $\mathcal{U}$ is a subspace of $U_1 \oplus \cdots \oplus U_r$. Hence $U_1 \oplus \cdots \oplus U_r = \spann \bigcup_{W \in \mathcal{U}} W$. The latter space is clearly preserved by every $B_i$ and has nonzero dimension, hence equals $V$ by irreducibility. The claim is proved. We deduce in particular that $r\ell=\dim V$. We have proved every part of the lemma except for the claim that $U_1,\ldots,U_r$ are the \emph{only} subspaces of $V$ with dimension $\ell$ and with finite orbit under the action of $(B_1,\ldots,B_N)$.

In pursuit of this result we now claim that for every $t=1,\ldots,r$ there exists a linear map $A_t \in \prox(V)$ which belongs to the group generated by the maps $B_1,\ldots,B_N$, fixes every $U_1,\ldots,U_r$ and whose leading eigenvalue has an eigenvector in $U_t$. By hypothesis there exists $\kkk \in \Sigma_N^*$ such that $B_\kkk$ is proximal. Since $B_\kkk$ induces a permutation of $\mathcal{U}$ it has a power which induces the identity permutation of $\mathcal{U}$ and hence fixes all of the subspaces $U_1,\ldots,U_t$. Without loss of generality we replace $B_\kkk$ with this power, so we have $B_\kkk U_t=U_t$ for every $t=1,\ldots,r$ and also $B_\kkk \in \prox(V)$. The determinant of $(B_\kkk-\lambda I)$ on $V$ is equal to the product of the determinants of the linear maps $(B_\kkk-\lambda I)|_{U_t}$, so if $\lambda$ is the leading eigenvalue of $B_\kkk$ then one of the restrictions $(B_\kkk-\lambda I)|_{U_t}$ must have zero determinant, which is to say there exists an eigenvalue for this eigenvector in the space $U_t$. Consequently there exists an integer $t_0 \in \{1,\ldots,r\}$ such that we may take $A_{t_0}:=B_\kkk$. By the transitivity property demonstrated earlier, for any other $t \in \{1,\ldots,r\}$ we may choose $\iii \in \Sigma_N^*$ such that $B_\iii U_t = U_{t_0}$ and it is clear that the map $A_t:=B_\iii^{-1} B_\kkk B_\iii$ has the desired properties.

It remains to show that $U_1,\ldots,U_r$ are the \emph{only} subspaces of $V$ with finite orbit under the action of $(B_1,\ldots,B_N)$ and with dimension $\ell$. To this end suppose that $\hat{U}$ is an arbitrary $\ell$-dimensional subspace of $V$ with finite orbit under the action of $(B_1,\ldots,B_N)$. As with the space $U$ we see easily that every element of the semigroup generated by $B_1,\ldots,B_N$ induces a permutation on the orbit of $\hat{U}$, which is finite. It follows directly that every element of the group generated by $B_1,\ldots,B_N$ also induces a permutation on the orbit of $\hat{U}$. We now claim that $\hat{U}$ must be one of the spaces $U_1,\ldots,U_r$. To see this it is clearly sufficient to show that $\hat{U}$ has nontrivial intersection with one of the spaces $U_t$, since then $\hat{U} \cap U_t$ is a nonzero subspace of $V$ with dimension at most $\ell$ and with finite orbit under the action of $(B_1,\ldots,B_N)$; hence it is necessarily $\ell$-dimensional and this yields $U_t=\hat{U}$ as required. Let us therefore show that $\hat{U}$ intersects one of the spaces $U_t$. Since $V=\bigoplus_{t=1}^r U_t$ there exists a subcollection of $U_1,\ldots,U_r$ whose span intersects $\hat{U}$ nontrivially, namely the whole collection. There therefore exists a \emph{minimal} subcollection whose span intersects $\hat{U}$ nontrivially; by relabelling if necessary, call this subcollection $U_1,\ldots,U_m$, say. Since $\hat{U} \cap (U_1\oplus \cdots \oplus U_m)$ is a nonzero subspace of $V$ with finite orbit under $(B_1,\ldots,B_N)$ it has dimension at least $\ell$ and therefore equals $\hat{U}$, and we deduce that $\hat{U} \subseteq U_1\oplus \cdots \oplus U_m$. By minimality the span of the smaller subcollection $U_1,\ldots,U_{m-1}$ must intersect $\hat{U}$ trivially. Therefore $U_1,\ldots,U_{m-1},\hat{U}$ forms a direct sum which is contained in $U_1\oplus \cdots \oplus U_m$ and so by dimension considerations must be equal to $U_1 \oplus \cdots \oplus U_m$. We deduce that $V = U_1 \oplus \cdots \oplus U_{m-1}\oplus \hat{U}\oplus U_{m+1}\oplus \cdots \oplus U_r= U_1\oplus \cdots \oplus U_r$. Now, the linear map $A_m$ constructed earlier induces a permutation on the orbit of $\hat{U}$ under $(B_1,\ldots,B_N)$ and therefore has a power which induces the identity permutation on that set. Replacing $A_m$ with a power of itself if necessary we may therefore choose $A_m \in \prox(V)$ so as to fix every space $U_1,\ldots,U_r$ and also so as to fix $\hat{U}$. Now since the leading eigenspace of $A_m$ is one-dimensional and intersects $U_m$, it does not intersect any of $U_1,\ldots,U_{m-1},U_{m+1},\ldots,U_r$. Since $A_m$  preserves every space in the splitting $V = U_1 \oplus \cdots \oplus U_{m-1}\oplus \hat{U}\oplus U_{m+1}\oplus \cdots \oplus U_r$, by the same argument used before its leading eigenspace must intersect one of those spaces; but this space must be $\hat{U}$ since none of the other spaces in the splitting intersects the leading eigenspace. Therefore $\hat{U} \cap U_m \neq \{0\}$ and we deduce that $\hat{U}=U_m$. The proof is complete.
\end{proof}
We lastly require the following simple combinatorial facts which will be used to construct transitive and primitive subspace classes:
\begin{lemma}\label{le:poximal3}
Let $k \geq 1$ and $N \geq 2$. For each $j=1,\ldots,k$ let $V_j$ be a finite-dimensional real vector space and $(A_1^{(j)},\ldots,A_N^{(j)}) \in \GL(V_j)^N$ an irreducible $N$-tuple. Then:
\begin{enumerate}[(i)]
\item
If  $\mathcal{W}\subseteq \prod_{j=1}^k \Gr(V)$ is an equivariant subspace class then it is equal to the disjoint union of finitely many transitive subspace classes.
\item
Suppose that $\mathcal{W}\subseteq \prod_{j=1}^k \Gr(V)$ is an equivariant subspace class with the following property: for every $(W_j)_{j=1}^k \in \mathcal{W}$ and $n \geq 1$ we have
\[\left\{(A_\iii^{(j)}W_j)_{j=1}^k \colon \iii \in \Sigma_N^*\text{ and }n\text{ divides }|\iii|\right\}=\mathcal{W}.\]
Then $\mathcal{W}$ is primitive.
\end{enumerate}
\end{lemma}
\begin{proof}
It is easy to see that if two transitive subspace classes intersect then they must contain each other and hence be equal, so to establish (i) it is sufficient to show that every $(W_j)_{j=1}^k \in \mathcal{W}$ belongs to a transitive subspace class. Given an arbitrary $(W_j^0)_{j=1}^k \in \mathcal{W}$ define
\[\hat{\mathcal{W}}:=\left\{(A_\iii^{(j)}W_j^0)_{j=1}^k \colon \iii \in \Sigma_N^*\right\}\subseteq \mathcal{W}.\]
We will show that $\hat{\mathcal{W}}$ is a transitive subspace class and contains $(W_j^0)_{j=1}^k$, which by the arbitrariness of $(W_j^0)_{j=1}^k \in \mathcal{W}$ clearly suffices to prove (i).

To show that $\hat{\mathcal{W}}$ is transitive we must show that for every $(W_j)_{j=1}^k \in \hat{\mathcal{W}}$ we have
\begin{equation}\label{eq:sudden-pine} \left\{(A_\iii^{(j)}W_j)_{j=1}^k \colon \iii \in \Sigma_N^*\right\}=\hat{\mathcal{W}}.\end{equation}
The left-hand side above is clearly a subset of the right-hand side, so we need only prove the reverse inclusion. Fix an arbitrary $(W_j)_{j=1}^k = (A_\jjj^{(j)}W_j^0)_{j=1}^k \in \hat{\mathcal{W}}$. The map $\mathcal{W} \to \mathcal{W}$ defined by $(W_j')_{j=1}^k \mapsto (A_\jjj^{(j)}W_j')_{j=1}^k$ is clearly injective and acts on a finite set, hence induces a permutation of $\mathcal{W}$, hence there exists an integer $p \geq 1$ such that the $p^{\mathrm{th}}$ power of this map is the identity permutation of $\mathcal{W}$. We therefore have $(W_j^0)_{j=1}^k = (A_{\jjj^p}^{(j)}W_j^0)_{j=1}^k = (A_{\jjj^{(p-1)}}^{(j)}W_j)_{j=1}^k \in \hat{\mathcal{W}}$. Hence
\begin{align*}\hat{\mathcal{W}}&= \left\{(A_\iii^{(j)}W_j^0)_{j=1}^k \colon \iii \in \Sigma_N^*\right\} \\
&= \left\{(A_{\iii\jjj^{p-1}}^{(j)}W_j)_{j=1}^k \colon \iii \in \Sigma_N^*\right\} \subseteq \left\{(A_\iii^{(j)}W_j)_{j=1}^k \colon \iii \in \Sigma_N^*\right\} \subseteq \hat{\mathcal{W}}.\end{align*}
We have established \eqref{eq:sudden-pine} which proves the transitivity of $\hat{\mathcal{W}}$ and the result (i) follows.

Let us now prove (ii). If $\iii \in \Sigma_N^*$ is arbitrary then it induces a permutation of $\mathcal{W}$ via the map $(W_j)_{j=1}^k \mapsto (A_\iii^{(j)}W_j)_{j=1}^k$ and in particular some power of $\iii$ induces the identity permutation. We may therefore choose $\jjj \in \Sigma_N^*$ such that $ (A_\jjj^{(j)}W_j)_{j=1}^k=(W_j)_{j=1}^k$ for all $(W_j)_{j=1}^k \in \mathcal{W}$. 
Let $\ell:=|\jjj|$. By hypothesis for every $(W_j)_{j=1}^k, (W_j')_{j=1}^k \in \mathcal{W}$ there exists $\iii \in \Sigma_N^*$ such that $(W_j)_{j=1}^k = (A_\iii^{(j)}W_j')_{j=1}^k$ and $\ell$ divides $|\iii|$. Since $\mathcal{W}$ is finite it follows directly that there exists an integer $t$ such that for every $(W_j)_{j=1}^k, (W_j')_{j=1}^k \in \mathcal{W}$ we may choose $\iii \in \Sigma_N^*$ satisfying $(W_j)_{j=1}^k = (A_\iii^{(j)}W_j')_{j=1}^k$, such that $\ell$ divides $|\iii|$, and such that $|\iii|=s\ell$ for some integer $s$ such that $1 \leq s <t$; we then have $(W_j)_{j=1}^k = (A_{\jjj^{t-s}\iii}^{(j)}W_j')_{j=1}^k$ and $|\jjj^{t-s}\iii|=t\ell$. Since the length of the word $\jjj^{t-s}\iii$ is independent of the choice of $(W_j)_{j=1}^k, (W_j')_{j=1}^k \in \mathcal{W}$ we have proved that $\mathcal{W}$ is primitive. 
\end{proof}

\begin{proof}[Proof of Theorem \ref{th:prox-unique}]  By Lemma \ref{le:poximal1} there exists $\kkk \in \Sigma_N^*$ such that for every $j=1,\ldots,k$ we have  $A_\kkk^{(j)} \in \prox(V_j)$, so the hypotheses of Lemma \ref{le:poximal2} are met by $(A_1^{(j)},\ldots,A_N^{(j)})$ and $V_j$ for every $j=1,\ldots,k$. Hence for each $j=1,\ldots,k$ there exist an integer $r_j \geq 1$ and a splitting $V_j=\bigoplus_{i=1}^{r_j} U_j^i$ such that $A_\iii^{(j)}U_j^i \in \{U_j^1,\ldots,U_j^{r_j}\}$ for every $\iii \in \Sigma_N^*$ and $i \in \{1,\ldots,r_j\}$. Without loss of generality we adopt an inner product structure on each $V_j$ such that the splitting $V_j=\bigoplus_{i=1}^{r_j} U_j^i$ is orthogonal.

Fix $\iii \in \Sigma_N^*$ and $j=1,\ldots,k$. The linear map $A_\iii^{(j)}$ induces a permutation of the set $\{U_j^1,\ldots,U_j^{r_j}\}$. We assert that $(A_\iii^{(j)})^{\top}$ also permutes this set and moreover induces the inverse permutation on it. Indeed, if $A_\iii^{(j)}U_j^{i_1}=U_j^{i_2}$ then we have
 \begin{equation}\label{eq:clay-cow}A_\iii^{(j)}\left(\left(U_j^{i_1}\right)^\perp\right) = A_\iii^{(j)} \left(\bigoplus_{\substack{1 \leq i \leq r_j \\ i \neq i_1}} U_j^{i}\right) = \bigoplus_{\substack{1 \leq i \leq r_j \\ i \neq i_2}} U_j^{i} = \left(U_j^{i_2}\right)^\perp.\end{equation}
A standard and easy calculation shows that if $B \in \GL(V)$ and $U,U' \in \Gr(V)$ then $BU = U'$ if and only if $B^\top (U')^\perp=U^\perp$, so \eqref{eq:clay-cow} yields $(A_\iii^{(j)})^\top U_j^{i_2}=U_j^{i_i}$. This demonstrates that $(A_\iii^{(j)})^\top$ induces the inverse permutation on $\{U_j^1,\ldots,U_j^{r_j}\}$ to that induced by $A_\iii^{(j)}$. It follows that $(A_\iii^{(j)})^\top A_\iii^{(j)}$ preserves $U_j^i$ for every $i=1,\ldots,r_j$. Since $V_j=\bigoplus_{i=1}^{r_j} U_j^i$ it follows that every eigenvalue of  $(A_\iii^{(j)})^\top A_\iii^{(j)}$ is realised as an eigenvalue of its restriction to some $U_j^i$. In particular the numbers $\|A_\iii^{(j)}|_{U_j^i}\| = \rho((A_\iii^{(j)})^\top A_\iii^{(j)})^{1/2}$ are distinct singular values of $A_\iii^{(j)}$ for distinct $i$. We conclude that for every  $\iii \in \Sigma_N^*$ and $j=1,\ldots,k$ we have
\begin{equation}\label{eq:parp-green}\left\|A_\iii^{(j)}\right\|=\max_{1 \leq i \leq r_j} \left\|A_\iii^{(j)}|_{U_j^i}\right\|\end{equation}
and moreover for each $i=1,\ldots,r_j$
\begin{equation}\label{eq:shivable-peach}\max\left\{\left\|A_\iii^{(j)}|_{U_j^i}\right\|, \left\|A_\iii^{(j)}|_{\left(U_j^i\right)^\perp}\right\|\right\} =\left\|A_\iii^{(j)}\right\|,\end{equation}
\begin{equation}\label{eq:hathole-gray}\min\left\{\left\|A_\iii^{(j)}|_{U_j^i}\right\|, \left\|A_\iii^{(j)}|_{\left(U_j^i\right)^\perp}\right\|\right\} \leq \sigma_2\left(A_\iii^{(j)}\right).\end{equation}

For the remainder of the proof let $\mu \in \mathcal{M}_{\sigma}(\Sigma_N)$ be an ergodic equilibrium state of $\Phi$ and let $\hat\mu \in \mathcal{M}_{\hat\sigma}(\hat{\Sigma}_N)$ be its natural extension. A large initial part of the proof will be taken up by the following construction: we claim that for every $j=1,\ldots,k$ there exists a Borel measurable function $\mathsf{U}_j \colon \hat\Sigma_N \to \{U_j^1,\ldots,U_j^{r_j}\}$ such that for $\hat\mu$-a.e $x \in \hat{\Sigma}_N$ and every $m \geq 1$,
\begin{align}\label{eq:lemon-nose}\lim_{n \to \infty}\frac{1}{n}\log\left\|A_{(\hat\sigma^{-n}x)|_n}^{(j)}|_{\mathsf{U}_j(x)}\right\| &= \lambda_1(A^{(j)},\mu),\\
\label{eq:copper-panty}\limsup_{n \to \infty}\frac{1}{n}\log\left\|A_{(\hat\sigma^{-n}x)|_n}^{(j)}|_{\mathsf{U}_j(x)^\perp}\right\| &\leq \lambda_2(A^{(j)},\mu)\end{align}
and
\begin{equation}\label{eq:dorkwood}A_{(\hat\sigma^{-m}x)|_m}^{(j)}\mathsf{U}_j(x) = \mathsf{U}_j(\hat\sigma^{-m}x).\end{equation}
To begin this construction fix $j \in \{1,\ldots,k\}$, define $Z$ to be the set
\[\bigcap_{\ell=1}^{\dim V_j} \left\{x \in \hat\Sigma_N \colon \lim_{n \to \infty} \frac{1}{n} \log \sigma_\ell\left(A_{(\hat\sigma^{-n}x)|_n}^{(j)}\right) = \lambda_\ell (A^{(j)},\mu)\right\}\]
and for each $i=1,\ldots,r_j$ define
\[X_i:=\left\{x \in \hat\Sigma_N \colon \lim_{n \to \infty} \frac{1}{n} \log \left\|A_{(\hat\sigma^{-n}x)|_n}^{(j)}|_{U_j^i}\right\| = \lambda_1(A^{(j)},\mu)\right\}.\]
Clearly all of these sets are Borel sets. For each $\ell=1,\ldots,\dim V_j$ the functions $f_n^\ell \colon \hat\Sigma_N \to \mathbb{R}$ defined by
\[f_n^\ell(x):= \log \prod_{i=1}^\ell \sigma_i\left(A_{(\hat\sigma^{-n}x)|_n}^{(j)}\right)\]
satisfy $f_{n+m}^\ell(x) \leq f_n^\ell(\hat\sigma^{-m}x)+f_m^\ell(x)$ for all $x \in \hat\Sigma_N$ and $n,m \geq 1$, and since clearly
\[Z=\bigcap_{\ell=1}^{\dim V_j} \left\{x \in \hat\Sigma_N \colon \lim_{n \to \infty} \frac{1}{n} \log \prod_{i=1}^\ell \sigma_\ell\left(A_{(\hat\sigma^{-n}x)|_n}^{(j)}\right) = \sum_{i=1}^\ell \lambda_i (A^{(j)},\mu)\right\}\]
it follows by the subadditive ergodic theorem applied to the functions $f^\ell_n$ and the transformation $\hat\sigma^{-1}$ that $\mu(Z)=1$. 

We claim that the sets $Z \cap X_i$ are pairwise disjoint for distinct values of $i$. Indeed, if $x \in Z \cap X_{i_1} \cap X_{i_2}$ with $i_1 \neq i_2$ then since $U_j^{i_2} \subseteq (U_j^{i_1})^\perp$,
\begin{align*}\lambda_1(A^{(j)},\mu) &= \lim_{n \to \infty} \frac{1}{n}\log \min\left\{\left\|A_{(\hat\sigma^{-n}x)|_n}^{(j)}|_{U_j^{i_1}}\right\|, \left\|A_{(\hat\sigma^{-n}x)|_n}^{(j)}|_{U_j^{i_2}}\right\|\right\}\\
&\leq \limsup_{n \to \infty} \frac{1}{n}\log \min\left\{\left\|A_{(\hat\sigma^{-n}x)|_n}^{(j)}|_{U_j^{i_1}}\right\|, \left\|A_{(\hat\sigma^{-n}x)|_n}^{(j)}|_{(U_j^{i_1})^\perp}\right\|\right\}\\
 &\leq \lim_{n \to \infty} \frac{1}{n}\log \sigma_2\left(A_{(\hat\sigma^{-n}x)|_n}^{(j)}\right) = \lambda_2(A^{(j)},\mu)<\lambda_1(A^{(j)},\mu)\end{align*}
 where in the first line we have used the definition of $X_{i_1}$ and $X_{i_2}$, and in the third line we have used \eqref{eq:hathole-gray}, the definition of $Z$ and the hypothesis that the top Lyapunov exponent of each $(A_1^{(j)},\ldots,A_N^{(j)})$ is simple. But the above chain of inequalities ends in a contradiction, and we conclude that $Z \cap X_{i_1} \cap X_{i_2}$ must be empty. The sets $Z \cap X_i$ are thus pairwise disjoint for distinct values of $i$ as claimed.
 
We now claim that $Z \cap (X_1 \cup \cdots \cup X_{r_j})$ has full measure. Define
\[C:=\max_{1 \leq i \leq N} \max\left\{\left\|\left(A_i^{(j)}\right)^{-1}\right\|,\left\|A_i^{(j)}\right\|\right\}\geq1.\]
If $x \in Z$ then there exists $n_0$ such that for all $n \geq n_0$
\[\left\|A_{(\hat\sigma^{-n}x)|_n}^{(j)}\right\| =\sigma_1\left(A_{(\hat\sigma^{-n}x)|_n}^{(j)}\right)>C^2\sigma_2 \left(A_{(\hat\sigma^{-n}x)|_n}^{(j)}\right)\]
by the definition of $Z$ and the hypothesis that the top Lyapunov exponent of each $(A_1^{(j)},\ldots,A_n^{(j)})$ is simple. If $x \in Z$ is fixed then for all $n \geq n_0$ we in particular have $\sigma_1(A_{(\hat\sigma^{-n}x)|_n}^{(j)})>\sigma_2 (A_{(\hat\sigma^{-n}x)|_n}^{(j)})$ and in view of \eqref{eq:shivable-peach} and \eqref{eq:hathole-gray} there exists a unique $i(x,n)$ such that $\|A_{(\hat\sigma^{-n}x)|_n}^{(j)}\| =\|A_{(\hat\sigma^{-n}x)|_n}^{(j)}|_{U_j^{i(x,n)}}\|$. We assert that $i(x,n)$ does not depend on the value of $n \geq n_0$. Given arbitrary $n \geq n_0$ we have 
\begin{align*}\left\|A_{(\hat\sigma^{-n-1}x)|_{n+1}}^{(j)}|_{U_j^{i(x,n)}}\right\|&=\left\|A_{x_{-n}}^{(j)}A_{(\hat\sigma^{-n}x)|_n}^{(j)}|_{U_j^{i(x,n)}}\right\| \geq C^{-1}\left\|A_{(\hat\sigma^{-n}x)|_{n}}^{(j)}|_{U_j^{i(x,n)}}\right\|\\
&=C^{-1}\sigma_1\left(A_{(\hat\sigma^{-n}x)|_{n}}^{(j)}\right)>C\sigma_2\left(A_{(\hat\sigma^{-n}x)|_{n}}^{(j)}\right)\\
&\geq C \left\|A_{(\hat\sigma^{-n}x)|_{n}}^{(j)}|_{(U_j^{i(x,n)})^\perp}\right\| = C\max_{i \neq i(x,n)} \left\|A_{(\hat\sigma^{-n}x)|_{n}}^{(j)}|_{U_j^{i}}\right\| \\
&\geq \max_{i \neq i(x,n)} \left\|A_{x_{-n}}^{(j)}A_{(\hat\sigma^{-n}x)|_{n}}^{(j)}|_{U_j^{i}}\right\|= \max_{i \neq i(x,n)} \left\|A_{(\hat\sigma^{-n-1}x)|_{n+1}}^{(j)}|_{U_j^{i}}\right\|\end{align*}
and in view of \eqref{eq:parp-green} it follows that $i(x,n)$ satisfies the definition of $i(x,n+1)$. We conclude that $i(x,n+1)=i(x,n)$ for all $n \geq n_0$ and therefore $i(x,n)$ in fact takes a constant value $i(x)$ for all $n \geq n_0$. We then have
\begin{align}\label{eq:spiced-rope}\lambda_1(A^{(j)},\mu) &=  \lim_{n \to \infty}\frac{1}{n}\log \sigma_1\left(A_{(\hat\sigma^{-n}x)|_n}^{(j)}\right)\\\nonumber
&= \lim_{n \to \infty} \frac{1}{n}\log \left\|A_{(\hat\sigma^{-n}x)|_n}^{(j)}\right\| = \lim_{n \to \infty}\frac{1}{n}\log\left\|A_{(\hat\sigma^{-n}x)|_n}^{(j)}|_{U_j^{i(x)}}\right\|\end{align}
so that $x \in X_{i(x)}$. Combining this with the preceding result we have shown that every element of $Z$ belongs to a unique $X_i$. Since $Z$ has full measure and the sets $X_i$ are Borel measurable we conclude that the sets $X_1,\ldots,X_{r_j}$ partition $\hat\Sigma_N$ up to $\hat\mu$-measure zero. If we define a function $\mathsf{U}_j \colon \hat\Sigma_N \to \{U_j^1,\ldots,U_j^{r_j}\}$ by $\mathsf{U}_j(x):=U^i_j$ if $x \in Z \cap X_i$ and $\mathsf{U}_j(x):=U^1_j$ otherwise, it is obvious that $\mathsf{U}_j$ is a measurable function. The equation \eqref{eq:spiced-rope} precisely asserts the truth of \eqref{eq:lemon-nose} and combining this with \eqref{eq:shivable-peach}, \eqref{eq:hathole-gray} and the definition of $Z$ yields \eqref{eq:copper-panty}. In particular \eqref{eq:lemon-nose} and \eqref{eq:copper-panty} hold for all $x \in Z \cap (X_1 \cup \cdots \cup X_{r_j})$.

It remains to prove \eqref{eq:dorkwood}. Suppose that $x$ belongs to the set $\bigcap_{\ell \in \mathbb{Z}} \hat\sigma^\ell (Z \cap (X_1 \cup \cdots \cup X_{r_j}))$, which clearly has full measure. It is easy to see that for every fixed $m \geq 1$ we have
\begin{align*}{\lefteqn{\lim_{n \to \infty} \frac{1}{n} \log \left\|A_{(\hat\sigma^{-n-m}x)|_{n}}^{(j)}|_{A_{(\hat\sigma^{-m}x)|_m}^{(j)}\mathsf{U}_j(x) }\right\|}}&\\
&= \lim_{n \to \infty} \frac{1}{n} \log \left\|A_{(\hat\sigma^{-n-m}x)|_{n}}^{(j)}A_{(\hat\sigma^{-m}x)|_{m}}^{(j)}|_{\mathsf{U}_j(x) }\right\|\\
&= \lim_{n \to \infty} \frac{1}{n} \log \left\|A_{(\hat\sigma^{-n-m}x)|_{n+m}}^{(j)}|_{\mathsf{U}_j(x) }\right\| = \lambda_1(A^{(j)},\mu)\end{align*}
and
\begin{align*}
{\lefteqn{\limsup_{n \to \infty} \frac{1}{n} \log \left\|A_{(\hat\sigma^{-n-m}x)|_{n}}^{(j)}|_{\left(A_{(\hat\sigma^{-m}x)|_m}^{(j)}\mathsf{U}_j(x)\right)^\perp }\right\|}}&\\
&=\limsup_{n \to \infty} \frac{1}{n} \log \left\|A_{(\hat\sigma^{-n-m}x)|_{n}}^{(j)}|_{A_{(\hat\sigma^{-m}x)|_m}^{(j)}\left(\mathsf{U}_j(x)^\perp\right) }\right\|\\
&=\limsup_{n \to \infty} \frac{1}{n} \log \left\|A_{(\hat\sigma^{-n-m}x)|_{n}}^{(j)}A_{(\hat\sigma^{-m}x)|_m}^{(j)}|_{\mathsf{U}_j(x)^\perp }\right\|\\
&=  \limsup_{n \to \infty} \frac{1}{n} \log \left\|A_{(\hat\sigma^{-n-m}x)|_{n+m}}^{(j)}|_{\mathsf{U}_j(x)^\perp }\right\| \leq \lambda_2(A^{(j)},\mu).\end{align*}
Since $\hat\sigma^{-m}x \in Z \cap (X_1 \cup \cdots \cup X_{r_j})$ both \eqref{eq:lemon-nose} and \eqref{eq:copper-panty} apply with $\hat\sigma^{-m}x$ in place of $x$ and in view of the above we can only have $A_{(\hat\sigma^{-m}x)|_m}^{(j)}\mathsf{U}_j(x) =\mathsf{U}_j(\hat\sigma^{-m}x)$, which is \eqref{eq:dorkwood}. This completes the construction of the measurable functions $\mathsf{U}_j \colon \hat\Sigma_N \to \{U_j^1,\ldots,U_j^{r_j}\}$.

We next claim that for fixed $j \in \{1,\ldots,k\}$, for $\hat\mu$-a.e. $x \in \hat\Sigma_N$ we have
\begin{align}\label{eq:flipper}\lim_{n \to \infty}\frac{1}{n}\log\left\|A_{x|_n}^{(j)}|_{\mathsf{U}_j(\hat\sigma^nx)}\right\| &= \lambda_1(A^{(j)},\mu),\\
\label{eq:shy-bather}\lim_{n \to \infty}\frac{1}{n}\log\left\|A_{x|_n}^{(j)}|_{\mathsf{U}_j(\hat\sigma^nx)^\perp}\right\| &\leq \lambda_2(A^{(j)},\mu).\end{align}
Define two sequences of bounded measurable functions $f_n,g_n \colon \hat\Sigma_N \to \mathbb{R}$ by $f_n(x):=\log \|A_{x|_n}^{(j)}|_{\mathsf{U}_j(\hat\sigma^nx)}\|$ and $g_n(x):=\log\|A_{x|_n}^{(j)}|_{\mathsf{U}_j(\hat\sigma^nx)^\perp}\|$. For all $n,m \geq 1$ we have $f_{n+m}(x) \leq f_n(\hat\sigma^mx) + f_m(x)$ and $g_{n+m}(x) \leq g_n(\hat\sigma^mx) + g_m(x)$ for $\hat\mu$-a.e. $x \in \hat\Sigma_N$, where we have used \eqref{eq:dorkwood}. It follows by the subadditive ergodic theorem that the limits in \eqref{eq:flipper} and \eqref{eq:shy-bather} exist $\hat\mu$-a.e. and are constant $\hat\mu$-a.e, so we need only show that these almost sure values are equal to $\lambda_1(A^{(j)},\mu)$ and bounded by $\lambda_2(A^{(j)},\mu)$ respectively. But from \eqref{eq:lemon-nose} and \eqref{eq:copper-panty} we have
\[\lim_{n \to \infty} \hat\mu\left(\left\{x \in \hat\Sigma_N \colon \left|\frac{1}{n}\log\left\|A_{x|_n}^{(j)}|_{\mathsf{U}_j(\hat\sigma^nx)}\right\| - \lambda_1(A^{(j)},\mu)\right|<\varepsilon\right\}\right) =1\]
and
\[\lim_{n \to \infty} \hat\mu\left(\left\{x \in \hat\Sigma_N \colon \frac{1}{n}\log\left\|A_{x|_n}^{(j)}|_{\mathsf{U}_j(\hat\sigma^nx)^\perp}\right\| \leq  \lambda_2(A^{(j)},\mu)+\varepsilon\right\}\right) =1\]
for all $\varepsilon>0$ since convergence almost everywhere implies convergence in measure and since $\hat\mu$ is invariant with respect to $\hat\sigma$. This is only possible if the almost sure limit in \eqref{eq:flipper} is equal to $\lambda_1(A^{(j)},\mu)$ and the almost sure limit in \eqref{eq:shy-bather} is bounded by $\lambda_2(A^{(j)},\mu)$. We have proved the claim.

Let $\mathcal{U}$ denote the set of all $k$-tuples of the form $(U_j^{i_j})_{j=1}^k$ where $1 \leq i_j \leq r_j$ for every $j=1,\ldots,k$. It is clear that $\mathcal{U}\subseteq \prod_{j=1}^k \Gr_{\ell_j}(V_j)$ is an equivariant subspace class, so by Lemma \ref{le:poximal3}(i) it is equal to the disjoint union of finitely many transitive subspace classes $\mathcal{W}_1,\ldots,\mathcal{W}_r$, say.  Consider the sets $Y_t:=\{x \in \hat\Sigma_N \colon (\mathsf{U}_j(x))_{j=1}^k \in \mathcal{W}_t\}$ for $t=1,\ldots,r$, which clearly form a measurable partition of $\hat\Sigma_N$. For $\hat\mu$-a.e. $x \in \hat\Sigma_N$ we have $(\mathsf{U}_j(x))_{j=1}^k \in \mathcal{W}_t$ if and only if $(\mathsf{U}_j(\hat\sigma x))_{j=1}^k \in \mathcal{W}_t$ as a consequence of \eqref{eq:dorkwood}, so each $Y_t$ is invariant under $\hat\sigma$ up to $\hat\mu$-measure zero. By the ergodicity of $\hat\mu$ it follows that there exists $t_0 \in \{1,\ldots,r\}$ such that $\hat\mu(Y_{t_0})=1$ and $\hat\mu(Y_t)=0$ for all other $t$.

We may now prove (i). If $\mathcal{W} \subseteq  \prod_{j=1}^k \Gr_{\ell_j}(V_j)$ is any transitive subspace class then we define a potential $\Phi_{\mathcal{W}} \colon \Sigma_N^* \to (0,+\infty)$ by
\[\Phi_{\mathcal{W}}(\iii):= \max_{(W_j)_{j=1}^k \in \mathcal{W}} \prod_{j=1}^k \|A_\iii^{(j)}|_{W_j}\|^{\beta_j}.\]
Since $\mu$ is an equilibrium state of $\Phi$ we have
\[P(\Phi)=h(\mu) + \Lambda(\Phi,\mu) = h(\mu) + \sum_{j=1}^k \beta_j \lambda_1(A^{(j)},\mu).\]
Since clearly $\Phi_{\mathcal{W}} \leq \Phi$ for every transitive subspace class $\mathcal{W} \subseteq  \prod_{j=1}^k \Gr_{\ell_j}(V_j)$ we have $P(\Phi_{\mathcal{W}})\leq P(\Phi)$ for all such classes, so if $\mathcal{W} \subseteq  \prod_{j=1}^k \Gr_{\ell_j}(V_j)$ is a transitive subspace class then $\mu$ is an equilibrium state of $\Phi_{\mathcal{W}}$ if and only if $\Lambda(\Phi_{\mathcal{W}},\mu)=\Lambda(\Phi,\mu)$. To prove (i) we will show that this is the case precisely when $\mathcal{W}=\mathcal{W}_{t_0}$. 

Let us first show that $\mu$ is an equilibrium state for $\Phi_{\mathcal{W}_{t_0}}$. We clearly have
\begin{align*}\sum_{j=1}^k \beta_j \lambda_1(A^{(j)},\mu) &= \lim_{n \to \infty} \frac{1}{n}\sum_{j=1}^k \beta_j \log \left\|A_{x|_n}^{(j)}|_{\mathsf{U}_j(\hat\sigma^nx)}\right\| \\
&= \lim_{n \to \infty} \frac{1}{n}\log \prod_{j=1}^k  \left\|A_{x|_n}^{(j)}|_{\mathsf{U}_j(\hat\sigma^nx)}\right\|^{\beta_j}  \\
&\leq \lim_{n \to \infty} \frac{1}{n}\log \max_{(W_j)_{j=1}^k \in \mathcal{W}_{t_0}}\prod_{j=1}^k  \left\|A_{x|_n}^{(j)}|_{W_j}\right\|^{\beta_j}  \\
&=\lim_{n \to \infty} \frac{1}{n}\log \Phi_{\mathcal{W}_{t_0}}(x|_n)\\
& =\Lambda(\Phi_{\mathcal{W}_{t_0}},\mu)\leq  \Lambda(\Phi,\mu) = \sum_{j=1}^k \beta_j \lambda_1(A^{(j)},\mu)\end{align*}
for $\hat\mu$-a.e. $x \in \hat\Sigma_N$, where we have used \eqref{eq:flipper} together with the subadditive ergodic theorem applied to $\Phi_{\mathcal{W}_{t_0}}$ and $\Phi$.  It follows that $\Lambda(\Phi_{\mathcal{W}_{t_0}},\mu)=\Lambda(\Phi,\mu)$ as required for $\mu$ to be an equilibrium state of $\Phi_{\mathcal{W}_{t_0}}$. This proves the existence part of (i).

Now suppose instead that $\mathcal{W} \subseteq \prod_{j=1}^k \Gr_{\ell_j}(V_j)$ is a transitive subspace class such that $\mu$ is an equilibrium state of the potential $\Phi_{\mathcal{W}}$. If $(W_j)_{j=1}^k \in \mathcal{W}$ then by definition each $W_j$ has finite orbit under the action of $(A_1^{(j)},\ldots,A_N^{(j)})$ and therefore must be one of the spaces $U_j^1,\ldots,U_j^{r_j}$, since Lemma \ref{le:poximal2} implies that those are the only $\ell_j$-dimensional subspaces of $V_j$ with finite orbit. Thus necessarily $\mathcal{W} \subseteq \mathcal{U}$. It follows that $\mathcal{W}$ intersects at least one $\mathcal{W}_t$ and since both are transitive we must have $\mathcal{W}=\mathcal{W}_t$. To complete the uniqueness part of the proof of (i) we will show that if $t \neq t_0$ then a contradiction occurs. 

Fix an arbitrary $t \neq t_0$. For almost every $x \in \hat\Sigma_N$ we have $(\mathsf{U}_j(\hat\sigma^nx))_{j=1}^k \in \mathcal{W}_{t_0}$ for all $n \geq 1$. For such $x$ and $n$, if $(W_j)_{j=1}^k \in \mathcal{W}_t $ is arbitrary then $(W_j)_{j=1}^k \notin \mathcal{W}_{t_0}$ and in particular $(W_j)_{j=1}^k \neq (\mathsf{U}_j(\hat\sigma^nx))_{j=1}^k$,  so there exists $j_0$ depending on $x$, $n$ and $(W_j)_{j=1}^k$ such that $W_{j_0} \neq \mathsf{U}_{j_0}(\hat\sigma^nx)$ and therefore $W_{j_0} \subseteq \mathsf{U}_{j_0}(\hat\sigma^nx)^\perp$. Consequently
\[\prod_{j=1}^k \left\|A_{x|_n}^{(j)}|_{W_j}\right\|^{\beta_j} \leq \left(\prod_{\substack{1 \leq j \leq k \\ j \neq j_0}}\left\|A_{x|_n}^{(j)}\right\|^{\beta_j} \right) \left\|A_{x|_n}^{(j_0)}|_{\mathsf{U}_{j_0}(\hat\sigma^nx)^\perp}\right\|^{\beta_{j_0}}\]
for this particular $x$, $n$ and $(W_j)_{j=1}^k$, and therefore 
\[\max_{(W_j)_{j=1}^k \in \mathcal{W}_t}\prod_{j=1}^k \left\|A_{x|_n}^{(j)}|_{W_j}\right\|^{\beta_j} \leq \max_{1 \leq \ell \leq k}\left[ \left(\prod_{\substack{1 \leq j \leq k \\ j \neq \ell}}\left\|A_{x|_n}^{(j)}\right\|^{\beta_j} \right) \left\|A_{x|_n}^{(\ell)}|_{\mathsf{U}_{\ell}(\hat\sigma^nx)^\perp}\right\|^{\beta_\ell}\right]\]
for this particular $x$ and $n$. Thus for $\hat\mu$-a.e. $x \in \hat\Sigma_N$ we have
\[\Phi_{\mathcal{W}_t}(x|_n)\leq \max_{1 \leq \ell \leq k}\left[ \left(\prod_{\substack{1 \leq j \leq k \\ j \neq \ell}}\left\|A_{x|_n}^{(j)}\right\|^{\beta_j} \right) \left\|A_{x|_n}^{(\ell)}|_{\mathsf{U}_{\ell}(\hat\sigma^nx)^\perp}\right\|^{\beta_\ell}\right]\]
for all $n \geq 1$. But for all $j=1,\ldots,k$ by the subadditive ergodic theorem we have for $\hat\mu$-a.e. $x \in \hat\Sigma_N$
\[\lim_{n \to \infty} \frac{1}{n}\log \left(\left\|A_{x|_n}^{(j)}\right\|^{\beta_j}\right)= \beta_j\lambda_1(A^{(j)},\mu)\]
and by \eqref{eq:shy-bather}
\[\lim_{n \to \infty} \frac{1}{n}\log \left(\left\|A_{x|_n}^{(j)}|_{\mathsf{U}_{j}(\hat\sigma^nx)^\perp}\right\|^{\beta_j}\right)\leq \beta_j\lambda_2(A^{(j)},\mu)\]
from which it follows that for $\hat\mu$-a.e. $x\in\hat\Sigma_N$
\begin{align*}\Lambda(\Phi_{\mathcal{W}_t},\mu) &= \lim_{n \to \infty} \frac{1}{n} \log\Phi_{\mathcal{W}_t}(x|_n)\\
& \leq \max_{1 \leq \ell \leq k} \left[ \left(\sum_{\substack{1 \leq j \leq k \\ j \neq \ell}} \beta_j\lambda_1(A^{(j)},\mu)\right) +\beta_\ell \lambda_2(A^{(\ell)},\mu)\right] \\
&<\sum_{j=1}^k \beta_j\lambda_1(A^{(j)},\mu)=\Lambda(\Phi,\mu).\end{align*}
Thus if $t \neq t_0$ we have $\Lambda(\Phi_{\mathcal{W}_t},\mu)<\Lambda(\Phi,\mu)$ and we have completed the proof of (i).

The proof of (ii) is now straightforward. By Lemma \ref{le:poximal1}, for $\hat\mu$-a.e. $x\in \hat\Sigma_N$ we have
\[\limsup_{n \to \infty}\min_{1 \leq j \leq k} \left(\frac{1}{n}\log\rho\left(A_{x|_n}^{(j)}\right) -\frac{1}{2n}\log \sigma_1\left(A_{x|_n}^{(j)}\right)\sigma_2\left(A_{x|_n}^{(j)}\right)\right)>0,\]
and for all $j=1,\ldots,k$ and $\hat\mu$-a.e. $x \in \hat\Sigma_N$ we have by \eqref{eq:shy-bather}
\[\lim_{n \to \infty}\frac{1}{n}\log\left\|A_{x|_n}^{(j)}|_{\mathsf{U}_j(\hat\sigma^nx)^\perp}\right\| \leq \lambda_2(A^{(j)},\mu)\]
and by the subadditive ergodic theorem 
\[\lim_{n \to \infty}  \frac{1}{2n}\log \sigma_1\left(A_{x|_n}^{(j)}\right)\sigma_2\left(A_{x|_n}^{(j)}\right)= \frac{1}{2}\lambda_1(A^{(j)},\mu)+\frac{1}{2}\lambda_2(A^{(j)},\mu)>\lambda_2(A^{(j)},\mu).\]
It follows that there exist $x \in \hat\Sigma_N$ and $n \geq 1$ such that $(\mathsf{U}_j(\hat\sigma^nx))_{j=1}^k \in \mathcal{W}_{t_0}$ and such that for all $j=1,\ldots,k$
\begin{equation}\label{eq:le-cute-white}\rho\left(A_{x|_n}^{(j)}\right)>\sigma_1\left(A_{x|_n}^{(j)}\right)^{\frac{1}{2}}\sigma_2\left(A_{x|_n}^{(j)}\right)^{\frac{1}{2}}>\left\|A^{(j)}_{x|_{n}}|_{\mathsf{U}_j(\hat\sigma^{n}x)^\perp}\right\|.\end{equation}
Let $\jjj:=x|_n$ and $(W_j^0)_{j=1}^k:=(\mathsf{U}_j(\hat\sigma^nx))_{j=1}^k \in \mathcal{W}_{t_0}$. Since $\rho(A_\jjj^{(j)})^2>\sigma_1(A_\jjj^{(j)})\sigma_2(A_\jjj^{(j)})$ for every $j=1,\ldots,k$ we have $A_\jjj^{(j)} \in \prox(V_j)$ for every $j=1,\ldots,k$. The map $(W_j)_{j=1}^k \mapsto (A_\jjj^{(j)}W_j)_{j=1}^k$ induces a permutation of $\mathcal{W}_{t_0}$ and therefore some power of this map induces the identity permutation, so let $\kkk:=\jjj^p$ be a power of $p$ such that $A_\kkk^{(j)}W_j^0=W_j^0$ for every $j=1,\ldots,k$. For every $j=1,\ldots,k$ we clearly have $A_\kkk^{(j)}\in \prox(V_j)$, and since $A_\kkk^{(j)}$ preserves the splitting $V_j = U_j^1\oplus \cdots \oplus U_j^{r_j}$ its leading eigenvector belongs to one of these spaces. In particular this eigenvector belongs either to $W_j^0$ or to $(W_j^0)^\perp$. If it belongs to the latter then the leading eigenvector of $A_\jjj^{(j)}$, being the same vector, belongs to $(W_j^0)^\perp$ and therefore we have $\rho(A_\jjj^{(j)})\leq \|A_\jjj^{(j)}|_{(W_j^0)^\perp}\|$, but this contradicts \eqref{eq:le-cute-white}. Hence the leading eigenvector of $A_\kkk^{(j)}$ belongs to $W_j^0$ for every $j=1,\ldots,k$. For each $j=1,\ldots,k$ the linear map $A_\kkk^{(j)}$ fixes $W_j^0$ and its eigenvalues when restricted to $W_j^0$ are a subset of its eigenvalues on $V_j$ which includes the largest eigenvalue. Thus $A_\kkk^{(j)}|_{W_j^0} \in \prox(W_j^0)$ for every $j=1,\ldots,k$ and we have proved (ii).

It remains to prove (iii). We first claim that for every $(W_j)_{j=1}^k \in \mathcal{W}_{t_0}$ we have
\[ \hat\mu\left(\{x \in \hat\Sigma_N \colon (\mathsf{U}_j(x))_{j=1}^k=(W_j)_{j=1}^k\}\right)>0.\]
To see this we argue as follows. Each $\iii \in \Sigma_N^*$ induces a permutation of $\mathcal{W}_{t_0}$ via the map $(W_j)_{j=1}^k \mapsto (A_\iii^{(j)}W_j)_{j=1}^k$, and some power of the same word consequently induces the inverse permutation. The set of all permutations of $\mathcal{W}_{t_0}$ which can be realised in this way thus forms a finite group. We may therefore find a finite list of words $\iii_1,\ldots,\iii_p \in \Sigma_N^*$ such that every permutation of $\mathcal{W}_{t_0}$ of the form $(W_j)_{j=1}^k \mapsto (A_\iii^{(j)}W_j)_{j=1}^k$ is realised by $(W_j)_{j=1}^k \mapsto (A_{\iii_1 \iii_2 \cdots \iii_q}^{(j)}W_j)_{j=1}^k$ for some $q \in \{1,\ldots,p\}$. (For example, we could take $p$ to be an even number such that the words $\iii_1,\iii_3,\iii_5,\ldots,\iii_{p-1}$ induce all of the various permutations and such that $\iii_{q+1}$ induces the inverse of the permutation induced by $\iii_q$ for each odd number $q$.) The list $\iii_1,\ldots,\iii_p \in \Sigma_N^*$ has the property that if $(W_j)_{j=1}^k \in \mathcal{W}_{t_0}$ is arbitrary, then
\begin{equation}\label{eq:fresh-canding}\left\{((A_{\iii_1\cdots \iii_q}^{(j)})^{-1}W_j)_{j=1}^k \colon 1 \leq q \leq p\right\} = \mathcal{W}_{t_0}.\end{equation}
Now by Theorem \ref{th:bomo} the measure $\mu$ is fully supported on $\Sigma_N$, so $\hat\mu([\iii])>0$ for every $\iii \in \Sigma_N^*$ and in particular $\hat\mu([\iii_1\cdots \iii_p])>0$. Since the set of all $x \in \hat\Sigma_N$ such that $ (\mathsf{U}_j(x))_{j=1}^k \in \mathcal{W}_{t_0}$ has full measure, we may choose $(W_j)_{j=1}^k \in \mathcal{W}_{t_0}$ such that
\[\hat\mu\left(\left\{x \in \hat\Sigma_N \colon (\mathsf{U}_j(x))_{j=1}^k=(W_j)_{j=1}^k\text{ and }x \in [\iii_1\cdots \iii_p]\right\}\right)>0.\]
For every $q=1,\ldots,p$ we have
\[\hat\mu\left(\left\{x \in \hat\Sigma_N \colon (\mathsf{U}_j(\hat\sigma^{-|\iii_1\cdots \iii_q|}x))_{j=1}^k=(W_j)_{j=1}^k\text{ and }\hat\sigma^{-|\iii_1\cdots \iii_q|}x \in [\iii_1\cdots \iii_p]\right\}\right)>0\]
by the $\hat\sigma$-invariance of $\hat\mu$. Now using \eqref{eq:dorkwood} we have
\[\left(\mathsf{U}_j(\hat\sigma^{-|\iii_1\cdots \iii_q|}x)\right)_{j=1}^k= \left(A_{(\hat\sigma^{-|\iii_1\cdots \iii_q|}x)|_{|\iii_1\cdots \iii_q|}}^{(j)}\mathsf{U}_j(x)\right)_{j=1}^k = \left(A_{\iii_1\cdots \iii_q}^{(j)} \mathsf{U}(x)\right)_{j=1}^k \]
for $\hat\mu$-a.e. $x \in \hat\sigma^{|\iii_1\cdots \iii_q|}[\iii_1\cdots \iii_p]$, so 
\[\hat\mu\left(\left\{x \in \hat\Sigma_N \colon (A_{\iii_1\cdots \iii_q}^{(j)} \mathsf{U}(x))_{j=1}^k =(W_j)_{j=1}^k \right\}\right)>0\]
and therefore
\[\hat\mu\left(\left\{x \in \hat\Sigma_N \colon (\mathsf{U}(x))_{j=1}^k =((A_{\iii_1\cdots \iii_q}^{(j)}  )^{-1}W_j)_{j=1}^k \right\}\right)>0\]
for every $q=1,\ldots,p$. In view of \eqref{eq:fresh-canding} this proves the claim.

We may now prove (iii). We will show that the hypothesis of Lemma \ref{le:poximal3}(ii) is satisfied by $\mathcal{W}_{t_0}$. Fix arbitrary $n \geq 1$ and suppose that $\mu$, and hence $\hat\mu$, is totally ergodic. Let $(W_j)_{j=1}^k,(W_j')_{j=1}^k \in \mathcal{W}_{t_0}$ be arbitrary. Since $\hat\mu$ is ergodic with respect to $\hat\sigma^n$, for every two sets $Z_1,Z_2 \subseteq \hat\Sigma_N$ such that $\hat\mu(Z_1),\hat\mu(Z_2)>0$ we may find $\ell \geq 1$ such that $\hat\mu(Z_1 \cap \hat\sigma^{n\ell}Z_2)>0$. In particular using the claim just proved it follows that there exists $\ell \geq 1$ such that
\[ \hat\mu\left(\{x \in \hat\Sigma_N \colon (\mathsf{U}_j(x))_{j=1}^k =(W_j)_{j=1}^k\text{ and } (\mathsf{U}_j(\hat\sigma^{-n\ell}x))_{j=1}^k=(W_j')_{j=1}^k\}\right)>0.\]
Hence for a positive-measure set of $x \in \hat\Sigma_N$ we have
\[(W_j')_{j=1}^k = (\mathsf{U}_j(\hat\sigma^{-n\ell} x))_{j=1}^k = (A_{(\hat\sigma^{-n\ell}x)|_{n\ell}}^{(j)}\mathsf{U}_j(x))_{j=1}^k= (A_{(\hat\sigma^{-n\ell}x)|_{n\ell}}^{(j)}W_j)_{j=1}^k\]
using \eqref{eq:dorkwood}. In particular $(W_j')_{j=1}^k = (A_\iii^{(j)}W_j)_{j=1}^k$ for some $\iii \in \Sigma_N^*$ such that $n$ divides $|\iii|$. Since $(W_j')_{j=1}^k \in \mathcal{W}_{t_0}$ was arbitrary, we have shown that for every $(W_j)_{j=1}^k \in \mathcal{W}_{t_0}$ we have 
\[\left\{(A_\iii^{(j)}W_j)_{j=1}^k \colon \iii \in \Sigma_N^* \text{ and }n \text{ divides }|\iii|\right\} = \mathcal{W}_{t_0}.\]
Since $n$ and $(W_j)_{j=1}^k \in \mathcal{W}_{t_0}$ were also arbitrary this shows that the hypothesis of Lemma \ref{le:poximal3}(ii) is satisfied by $\mathcal{W}_{t_0}$. Hence $\mathcal{W}_{t_0}$ is primitive as required to prove (iii) and complete the proof of the theorem.
 \end{proof}




\section{Towards the precondition for $\psi$-mixing in the primitive proximal case}\label{se:thirdstage}

In this section we prove the following final major component of Theorem \ref{th:main}:
\begin{theorem}\label{th:core0}
Let $N \geq 2$ and $k \geq 1$ and for each $j=1,\ldots,k$ let $V_j$ be a finite-dimensional real vector space, let $(A^{(j)}_1,\ldots,A_N^{(j)}) \in \GL(V_j)^N$ be irreducible and let $\beta_j>0$. For each $j=1,\ldots,k$ let $\ell_j\geq 1$ be the least possible dimension of a nonzero subspace of $V_j$ which has finite orbit under the action of $(A_1^{(j)},\ldots,A_N^{(j)})$. Let $\mathcal{W} \subseteq \prod_{j=1}^k \Gr_{\ell_j}(V_j)$ be a primitive subspace class, and for every $\iii \in \Sigma_N^*$ define 
\[\Phi_{\mathcal{W}}(\iii):=\prod_{j=1}^k \left\|A_\iii^{(j)}|_{W_j}\right\|^{\beta_j}.\]
Suppose that there exist $\kkk \in \Sigma_N^*$ and $(W_j^0)_{j=1}^k \in \mathcal{W}$ such that $A_\kkk^{(j)}|_{W_j^0} \in \prox(W_j^0)$ for all $j=1,\ldots,k$. (In particular we suppose that $A_\kkk^{(j)}W_j^0=W_j^0$ for all $j=1,\ldots,k$.) Then there exist $m \geq 1$ and $\delta>0$ such that
\begin{equation} \label{eq:ghasty-pink}\max_{|\kkk|=m} \Phi_{\mathcal{W}}(\iii\kkk\jjj) \geq \delta \Phi_{\mathcal{W}}(\iii)\Phi_{\mathcal{W}}(\jjj)\end{equation}
for all $\iii,\jjj \in \Sigma_N^*$.
\end{theorem}

The inequality \eqref{eq:ghasty-pink} refines the statement
\begin{equation}\label{eq:ferry-purple}\max_{|\kkk|\leq m} \Phi(\iii\kkk\jjj) \geq \delta \Phi(\iii)\Phi(\jjj)\end{equation}
which was the core result required in the proof of Theorem \ref{th:bomo} above in the earlier article \cite{BoMo18}. To some extent, therefore, Theorem \ref{th:core0} above adapts the ideas used in  \cite{BoMo18}  in such a manner as to obtain a stronger conclusion under stronger hypotheses. In particular this argument makes heavy use of the properties of the Zariski topology described for the reader in \S\ref{ss:zariski}. We will briefly compare these two results and their proofs.

The earlier result \cite[Theorem 6]{BoMo18} began by defining $V:=\bigoplus_{j=1}^kV_j$ and $g_i:=\bigoplus_{j=1}^k A_i^{(j)} \in \GL(V)$ for each $i=1,\ldots,N$, defining $g_\iii:=g_{i_1}\cdots g_{i_n}$ for every $\iii=(i_\ell)_{\ell=1}^n \in \Sigma_N^*$, defining  $\Gamma \subset \GL(V)$ to be the semigroup generated by the linear maps $g_1,\ldots,g_N$, taking $G\leq \GL(V)$ to be the Zariski closure of $\Gamma$, defining homomorphisms $\phi_j \colon \Gamma \to \GL(V_j)$ by $\phi_j(g_{i_1}\cdots g_{i_n}):=A_{i_1}^{(j)}\cdots A_{i_n}^{(j)}$ for each $j=1,\ldots,k$ and observing that these extend to regular representations $\phi_j \colon G \to \GL(V_j)$ for each $j$. The key objective was then to show that there exist an integer $m \geq 1$ and real number $\kappa>0$ such that for every $\iii,\jjj \in \Sigma_N^*$ there exists $\kkk \in \Sigma_N^*$ such that $|\kkk|\leq m$ and $\|\phi_j(g_\iii g_\kkk g_\jjj)\| \geq \kappa \|\phi_j(g_\iii)\|\cdot \|\phi_j(g_\jjj)\|$ simultaneously for $j=1,\ldots,k$; the main result \eqref{eq:ferry-purple} then followed directly. In order to find $\kkk$ satisfying these simultaneous conditions it is useful to be able to pass to the identity component $G^0$ of $G$, which has the advantageous property of being an irreducible variety: in an irreducible variety every nonempty open set is dense, and so to find a word solving $k$ algebraic conditions simultaneously it is sufficient to show that each condition separately is satisfied on a nonempty open subset of $G^0$, since the intersection of these dense open sets must be nonempty. The analytic result $\|\phi_j(g_\iii g_\kkk g_\jjj)\| \geq \kappa \|\phi_j(g_\iii)\|\cdot \|\phi_j(g_\jjj)\|$ is proved by applying the previous reasoning to an algebraic property of $g_\kkk$ and combining this with a compactness argument. In order to reduce this task to that of studying elements of $G^0$ only, the proof of \cite[Theorem 6]{BoMo18} exploited the transitivity of $\mathcal{W}$: by appending a small word to $\iii$ and prepending a small word to $\jjj$ if necessary, we could assume that $\|\phi_j(g_\iii)\|$ and $\|\phi_j(g_\jjj)\|$ could be understood using the behaviour of $\phi_j(g_\iii)$ and $\phi_j(g_\jjj)$ restricted to $W_j\subseteq V_j$ for a consistent choice of $(W_j)_{j=1}^k \in \mathcal{W}$ not depending on $\iii$ or $\jjj$. This allowed us the freedom to specialise to considering only those $g_\kkk$ such that $\phi_j(g_\kkk)W_j=W_j$, in particular allowing us to work only in $G^0$ and not in the whole of $G$, making arguments based on the irreducibility of the variety $G^0$ available. Crucially, since only transitivity of $\mathcal{W}$ was assumed, the lengths of the appended and prepended words could be bounded \emph{a priori} via the finiteness and transitivity of $\mathcal{W}$ but their precise length could not be specified in advance.

In order to prove Theorem \ref{th:core0} we improve this argument in two respects so as to make the length of the word $\kkk$ interposed between $\iii$ and $\jjj$ consistent across all $\iii$ and $\jjj$. Firstly, by assuming primitivity of $\mathcal{W}$ instead of transitivity we can control the length of the small word appended to $\iii$ and the small word prepended to $\jjj$ so as to make their lengths equal to some \emph{a priori} constant $p$ which is independent of the choice of $\iii$ and $\jjj$. This reduces the problem to that of controlling the length of the word $\kkk$ which is chosen so as to satisfy 
$\|\phi_j(g_\iii g_\kkk g_\jjj)\| \geq \kappa \|\phi_j(g_\iii)\|\cdot \|\phi_j(g_\jjj)\|$ simultaneously for $j=1,\ldots,k$. To solve the latter problem we use the following intuition: if we knew that $\kkk$ could be chosen with length less than or equal to some number $m_0$ but also with the property that each $\phi_j(g_\kkk)$ was of rank one and non-nilpotent, then we would know that every power of $\phi_j(g_\kkk)$ is just a scalar multiple of $\phi_j(g_\kkk)$ and satisfies a relation similar to $\|\phi_j(g_\iii g_\kkk g_\jjj)\| \geq \kappa \|\phi_j(g_\iii)\|\cdot \|\phi_j(g_\jjj)\|$ but with an additional scalar constant. Thus if we took $m$ to be the least common multiple of the lengths of the finitely many different words $\kkk$ needed to connect the full range of possible pairs of words $\iii,\jjj \in \Sigma_N^*$, and replaced each $\kkk$ with a power of itself having length $m$, we could obtain the desired inequality \eqref{eq:ghasty-pink} with the word in the middle having length $m+2p$. But it is obvious that $\phi_j(g_\kkk)$ actually has full rank, so this idea must be modified. The key modification is to choose each $\kkk$ so that each $\phi_j(g_\kkk)$ is \emph{proximal}, and hence is close in norm to a scalar multiple of its powers. This explains the additional hypothesis of Theorem \ref{th:core0} regarding the existence of a simultaneously proximal word. 

\begin{proof}[Proof of Theorem \ref{th:core0}]

 Define $V:=\bigoplus_{j=1}^kV_j$ and for each $i=1,\ldots,N$ define $g_i:=\bigoplus_{j=1}^k A_i^{(j)} \in \GL(V)$. Define $g_\iii:=g_{i_1}\cdots g_{i_n}$ for every $\iii =(i_\ell)_{\ell=1}^n \in \Sigma_N^*$ in the obvious fashion. Let $\Gamma:=\{g_\iii \colon \iii \in \Sigma_N^*\}$, which is a semigroup. Let $G \leq \GL(V)$ denote the Zariski closure of $\Gamma$, which is a group. For each $j=1,\ldots,k$ we may define a regular representation $\phi_j \colon \Gamma \to \GL(V_j)$ by $\phi_j(g_\iii):=A_\iii^{(j)}$ and this extends to a regular representation $\phi_j \colon G \to \GL(V_j)$ which is irreducible since its image contains the linear maps $A_1^{(j)},\ldots,A_N^{(j)}$.
 
Let $(W_j^0)_{j=1}^k \in \mathcal{W}$ be as in the statement of the theorem. For each $j=1,\ldots,k$ the subspace $W_j^0$ has finite orbit under the action of $\phi_j(G)$. Denote this orbit by $\{U_j^1,\ldots,U_j^{r_j}\}$.  If $i_0 \in \{1,\ldots,r_j\}$ is fixed then each of the sets $\{g \in G \colon \phi_j(g)U_j^{i_0} = U_j^{i}\}$ is closed in the Zariski topology, and since these sets are disjoint and their union over $i=1,\ldots,r_j$ is equal to $G$, these sets are also open. It follows that they are unions of connected components of $G$ and in particular every connected component of $G$ is contained in a unique set of the form $\{g \in G \colon \phi_j(g)U_j^{i_0} = U_j^{i}\}$ for some $i \in \{1,\ldots,r_j\}$. Since the identity is in $G^0$ we must have $G^0 \subseteq \{g \in G \colon \phi_j(g)U_j^{i_0} = U_j^{i_0}\}$, and it follows that every $U^i_j$ is stabilised by $\phi_j(G^0)$. In particular we have $\phi_j(g)W_j^0=W_j^0$ for every $j=1,\ldots,k$ for every $g \in G^0$. For each $j=1,\ldots,k$ let us define a regular representation $\hat\phi_j \colon G^0 \to \GL(W_j^0)$ by $\hat\phi_j(g):=\phi_j(g)|_{W_j^0}$. By hypothesis there exists $\kkk \in \Sigma_N^*$ such that $\hat\phi_j(g_\kkk) \in \prox(W_j^0)$ for every $j=1,\ldots,k$. The map $g \mapsto g_\kkk g$ is a Zariski homeomorphism of $G$ and induces a permutation of the connected components of $G$; in particular there exists an iterate of this map which induces the identity permutation. By replacing $\kkk$ with a suitable power of $\kkk$ if necessary we may therefore assume without loss of generality that $g_\kkk \in G^0$.  Define $q_{\kkk,j}:=\lim_{n \to \infty} \|\hat\phi_j(g_\kkk^n)\|^{-1}\hat\phi_j(g_\kkk^n) \in \ned(W_j^0)$ for each $j=1,\ldots,k$; by proximality each $q_{\kkk,j}$ is well-defined, of rank one, and not nilpotent. 

We make the following first claim: if nonzero elements $b_{1,j},b_{2,j}$ of $\ned(W_j^0)$ are given for each $j=1,\ldots,k$ then there exists $\iii \in \Sigma_N^*$ such that $g_\iii \in G^0$, such that $\hat\phi_j(g_\iii) \in \prox(W_j^0)$ for each $j=1,\ldots,k$ and such that for each $j=1,\ldots,k$ the linear map $p_{\iii,j}:=\lim_{n \to \infty} \|\hat\phi_j(g_\iii^n)\|^{-1}\hat\phi_j(g_\iii^n)\in \ned(W_j^0)$ satisfies $b_{1,j}p_{\iii,j}b_{2,j} \neq 0$. Clearly it will be sufficient to choose a nonzero vector $v_j \in b_{2,j}W_j^0$ for each $j=1,\ldots,k$ and find $\iii \in \Sigma_N^*$ such that $g_\iii \in G^0$, such that $\hat\phi_j(g_\iii)$ is proximal for all $j=1,\ldots,k$ and such that $p_{\iii,j}v_j \notin \ker b_{1,j}$ for all $j=1,\ldots,k$. We therefore fix nonzero linear maps $b_{1,j}$ and $b_{2,j}$ and nonzero vectors $v_j$ for each $j=1,\ldots,k$ and prove the claim in this form. 

 We assert that there exists $\jjj_2 \in \Sigma_N^*$ such that $g_{\jjj_2} \in G^0$ and such that for all $j=1,\ldots,k$ we have $q_{\kkk,j}\hat\phi_j(g_{\jjj_2}) v_j \neq 0$. To see this it suffices to show that 
\[\bigcap_{j=1}^k \left\{g \in G^0 \colon q_{\kkk,j}\hat\phi_j(g) v_j \neq 0\right\}\]
is nonempty and Zariski open, since by the Zariski density of $\Gamma$ in $G$ it must then contain some $g_\iii \in \Gamma$. This set is the intersection of the sets
\begin{equation}\label{eq:burf-pink}\left\{g \in G^0 \colon q_{\kkk,j}\hat\phi_j(g) v_j \neq 0\right\}\end{equation}
over $j=1,\ldots,k$ and each of these sets is clearly Zariski open. Since $G^0$ is an irreducible variety, all of its nonempty open subsets are also dense, so the intersection of the sets \eqref{eq:burf-pink} will be nonempty and open as long as each individual set is nonempty. The assertion will therefore be proved if each of the sets in \eqref{eq:burf-pink} is shown to be nonempty.  But if this set is empty for some $j$ then the vector space $U:=\spann \{\hat\phi_j(g)v_j\colon g \in G^0\}$ is a subspace of $W_j^0$ (and hence of $V_j)$ which is invariant under $\phi_j(G^0)$ and has smaller dimension than $W_j^0$, since it is contained in the proper subspace $\ker q_{\kkk,j}$ of $W_j^0$. If $h_1,h_2 \in G$ belong to the same connected component $G_i$ of $G$ then $h_1^{-1}G_i$ is a connected component of $G$ which contains the identity, so $h_1^{-1}h_2 \in h_1^{-1}G_i=G^0$, hence $\phi_j(h_1)^{-1}\phi_j(h_2)U=\phi_j(h_1^{-1}h_2)U=U$ by the $\phi_j(G^0)$-invariance of $U$. Thus $g\mapsto \phi_j(g)U$ is constant on each connected component of $G$ and therefore $\{\phi_j(g)U \colon g \in G\}$ is finite. But then $U$ has finite orbit under $\phi_j(G)$ and in particular has finite orbit under $\phi_j(\Gamma)=\{A_\iii^{(j)} \colon \iii \in \Sigma_N^*\}$ whilst having dimension smaller than $\ell_j=\dim W_j^0$; this contradicts the definition of $\ell_j$. We conclude that such a subspace $U$ cannot exist, so the set \eqref{eq:burf-pink} must be nonempty for every $j=1,\ldots,k$ and we deduce the existence of the claimed element $\jjj_2 \in \Sigma_N^*$.

We next assert that there exists $\jjj_1 \in \Sigma_N^*$ such that $g_{\jjj_1}\in G^0$ and such that for all $j=1,\ldots,k$ the endomorphism $\hat\phi_j(g_{\jjj_1}) q_{\kkk,j}\hat\phi_j(g_{\jjj_2})$ is proximal and satisfies $\hat\phi_j(g_{\jjj_1})q_{\kkk,j}\hat\phi_j(g_{\jjj_2}) v_j \notin \ker b_{1,j}$. Clearly $\hat\phi_j(g) q_{\kkk,j}\hat\phi_j(g_{\jjj_2}) \in \ned(W_j^0)$ has rank one  for every $g \in G^0$ since $q_{\kkk,j}$ is of rank one and $\hat\phi_j(g)$ and $\hat\phi_j(g_{\jjj_2})$ are invertible, so for $\hat\phi_j(g) q_{\kkk,j}\hat\phi_j(g_{\jjj_2})$ to be proximal it is necessary and sufficient that it be non-nilpotent, which by rank considerations is equivalent to the condition $(\hat\phi_j(g) q_{\kkk,j}\hat\phi_j(g_{\jjj_2}))^2\neq 0$. Thus to obtain the existence of $\jjj_1$ it suffices to show that 
\begin{equation}\label{eq:black-hand}\bigcap_{j=1}^k \left\{g \in G^0 \colon \hat\phi_j(g) q_{\kkk,j}\hat\phi_j(g_{\jjj_2}) v_j \notin \ker b_{1,j}\right\} \cap \left\{g \in G^0 \colon\hat\phi_j(g) q_{\kkk,j}\hat\phi_j(g_{\jjj_2}))^2\neq 0\right\} \end{equation}
is nonempty and Zariski open. We must therefore likewise show that each set
\[\left\{g \in G^0 \colon \hat\phi_j(g) q_{\kkk,j}\hat\phi_j(g_{\jjj_2}) v_j \notin \ker b_{1,j}\right\}\]
is nonempty and Zariski open, and that each set
\[\left\{g \in G^0 \colon(\hat\phi_j(g) q_{\kkk,j}\hat\phi_j(g_{\jjj_2}))^2\neq 0\right\}\]
is also nonempty and Zariski open.  A vector belongs to $\ker b_{1,j}$ if and only if it orthogonal to every element of a basis for the orthogonal complement of $\ker b_{1,j}$, and the latter is obviously a Zariski closed condition, so the first of the two sets is Zariski open. The Zariski openness of the second set is obvious. 

If the first set is empty for some $j$, then by the same arguments as were used previously the vector space $\spann \{\hat\phi_j(g)q_{\kkk,j}\hat\phi_j(g_{\jjj_2})v_j\colon g \in G^0\}$ would be a $\phi_j(G^0)$-invariant proper subspace of $W_j^0$, the existence of which would contradict the definition of $\ell_j$. We conclude that the first set is nonempty for each $j=1,\ldots,k$. If $(\hat\phi_j(g) q_{\kkk,j}\hat\phi_j(g_{\jjj_2})^2=0$ for some $g \in G^0$ then since $q_{\kkk,j}$ has rank one and $\hat\phi_j(g)$ and $\hat\phi_j(g_{\jjj_2})$ are invertible, it must be the case that the one-dimensional image subspace $q_{\kkk,j}W_j^0$ is mapped into the kernel of $q_{\kkk,j}\hat\phi_j(g_{\jjj_2})$ by $\hat\phi_j(g)$. If this holds for every $g \in G^0$ then we deduce that $\spann \bigcup_{g \in G^0} \hat\phi_j(g)q_{\kkk,j}W_j^0\subseteq \ker q_{\kkk,j}\hat\phi_j(g_{\jjj_2}) \neq W_j^0$ is a $\phi_j(G^0)$-invariant proper subspace of $W_j^0$, which is again impossible. We deduce the nonemptiness and Zariski openness of the set \eqref{eq:black-hand} and the existence of $\jjj_1$ follows.

We have shown that there exist $\jjj_1,\jjj_2 \in \Sigma_N^* $ such that $g_{\jjj_1},g_{\jjj_2} \in G^0$ and such that for every $j=1,\ldots,k$, $\hat\phi_j(g_{\jjj_1}) q_{\kkk,j} \hat\phi_j(g_{\jjj_2}) \in \ned(W_j^0)$ is proximal and of rank one and satisfies $\hat\phi_j(g_{\jjj_1}) q_{\kkk,j} \hat\phi_j(g_{\jjj_2})v_j \notin \ker b_{1,j}$. In particular we necessarily have $V^+(\hat\phi_j(g_{\jjj_1}) q_{\kkk,j} \hat\phi_j(g_{\jjj_2})) \cap \ker b_{1,j} = \{0\}$ and $v_j \notin V^-(\hat\phi_j(g_{\jjj_1}) q_{\kkk,j} \hat\phi_j(g_{\jjj_2}))$. By the openness of the set of proximal endomorphisms in the analytic topology on $\ned(W_j^0)$ together with the continuity of $V^+$ and $V^-$ on that set, it follows that for all sufficiently large $n$ the element $\|\hat\phi_j(g_{\kkk}^n)\|^{-1}\hat\phi_j(g_{\jjj_1}g_{\kkk}^n g_{\jjj_2})$ is proximal and satisfies $V^+(\|\hat\phi_j(g_{\kkk}^n)\|^{-1} \hat\phi_j(g_{\jjj_1}g_{\kkk}^n g_{\jjj_2}))\cap \ker b_{1,j} =\{0\}$ and $ v_j \notin V^-(\|\hat\phi_j(g_{\kkk}^n )\|^{-1}  \hat\phi_j(g_{\jjj_1}g_{\kkk}^n g_{\jjj_2}))$
for all $j=1,\ldots,k$. Fix $n$ large enough that these properties hold and define $\iii:=\jjj_1 \kkk^n \jjj_2$. We then have $\hat\phi_j(g_{\jjj_1} g_\kkk^n g_{\jjj_2}) \in \prox(W_j^0)$ for all $j=1,\ldots,k$, and for every $j=1,\ldots,k$ we also have $V^+(\hat\phi_j(g_{\jjj_1}g_{\kkk}^n g_{\jjj_2}))\cap \ker b_{1,j}=\{0\}$ and  $v_j \notin V^-(\hat\phi_j(g_{\jjj_1}g_{\kkk}^n g_{\jjj_2}))$.
The  limit  $p_{\iii,j}:=\lim_{n \to \infty} \|\hat\phi_j(g_\iii^n)\|^{-1}\hat\phi_j(g_\iii^n)\in \prox(W_j^0)$  has image $V^+(\hat\phi_j(g_\iii))$ and kernel $V^-(\hat\phi_j(g_\iii))$ for each $j=1,\ldots,k$, so in particular $p_{\iii,j}v_j \notin \ker b_{1,j}$ and it follows that $b_{1,j}p_{\iii,j}b_{2,j}\neq 0$. The proof of the first claim is complete.

 We secondly claim that there exist $\kappa>0$ and $m \geq 1$ with the following property: if for each $j=1,\ldots,k$ we are given $b_{1,j},b_{2,j} \in \ned(W_j^0)$, then
\[\max_{\substack{ |\kkk|=m\\g_\kkk \in G^0}} \left\|b_{1,j}\hat\phi_j(g_\kkk) b_{2,j}\right\| \geq \kappa \|b_{1,j}\|\cdot\|b_{2,j}\|\] 
for every $j=1,\ldots,k$. By homogeneity it is clearly sufficient to consider only the case where $\|b_{1,j}\|=\|b_{2,j}\|=1$ for every $j=1,\ldots,k$, and we will prove the claim in this form.

For each $j=1,\ldots,k$ let $S_{\ned(W_j^0)}$ denote the unit sphere of $\ned(W_j^0)$. If $((b_{1,j},b_{2,j}))_{j=1}^k \in \prod_{j=1}^k S_{\ned(W_j^0)}\times S_{\ned(W_j^0)}$ is given, then by the preceding step there exists $\iii \in \Sigma_N^*$ such that $g_\iii \in G^0$ and such that for each $j=1,\ldots,k$ the element $p_{\iii,j}:=\lim_{n\to \infty} \|\hat\phi_j(g_\iii^n)\|^{-1}\hat\phi_j(g_\iii^n) \in \ned(W_j^0)$ is well-defined and satisfies $b_{1,j}p_{\iii,j}b_{2,j} \neq 0$. If $((b_{1,j}',b_{2,j}'))_{j=1}^k\in \prod_{j=1}^k S_{\ned(W_j^0)}\times S_{\ned(W_j^0)}$ is chosen in a sufficiently small open neighbourhood of $((b_{1,j},b_{2,j}))_{j=1}^k$ then we clearly also have $b_{1,j}'p_{\iii,j}b_{2,j}' \neq 0$ for every $j=1,\ldots,k$ for the same word $\iii$. By the compactness of $\prod_{j=1}^k S_{\ned(W_j^0)}\times S_{\ned(W_j^0)}$ it follows that there exist finitely many words $\iii_1,\ldots,\iii_r\in  \Sigma_N^*$ such that $g_{\iii_t} \in G^0$ for every $t=1,\ldots,r$, such that $p_{\iii_t,j}:=\lim_{n\to \infty} \|\hat\phi_j(g_{\iii_t}^n)\|^{-n}\hat\phi_j(g_{\iii_t}^n) \in \ned(W_j^0)$ is well-defined for each $j=1,\ldots,k$ and $t=1,\ldots,r$ and such that for every $((b_{1,j},b_{2,j}))_{j=1}^k \in S_{\ned(W_j^0)}$ there exists $t \in \{1,\ldots,r\}$ such that $b_{1,j}p_{\iii_t,j}b_{2,j} \neq 0$ for all $j=1,\ldots,k$. By compactness and continuity the function $\prod_{j=1}^k S_{\ned(W_j^0)}\times S_{\ned(W_j^0)}\to \mathbb{R}$ defined by
\[((b_{1,j},b_{2,j}))_{j=1}^k \mapsto \min_{1 \leq j \leq k}\max_{1 \leq t \leq r} \|b_{1,j}p_{\iii_t,j}b_{2,j}\|\]
therefore has a nonzero minimum value $\tau>0$, say. 

Let $m\geq 1$ be a natural number divisible by each of $|\iii_1|,\ldots,|\iii_t|$ and choose natural numbers $n_1,\ldots,n_r$ such that $m=n_1|\iii_1|=n_2|\iii_2|=\cdots =n_r |\iii_r|$. By choosing a large integer $\ell\geq 1$ and replacing $m$ with $\ell m$ and each $n_t$ with $\ell n_t$ if required, we may without loss of generality suppose that $n_1,\ldots,n_r$ are large enough that
\[ \max_{1 \leq j \leq k}\max_{1 \leq t \leq r} \left\|p_{\iii_t,j} - \|\hat\phi_j(g_{\iii_t}^{n_t})\|^{-1}\hat\phi_j(g_{\iii_t}^{n_t}) \right\|<\frac{\tau}{2}.\]
Define $\kkk_t:=\iii_t^{n_t} \in \Sigma_N^*$ for each $t$ and observe that each $\kkk_t$ has the same length $m$ and satisfies $g_{\kkk_t}=g_{\iii_t}^{n_t}\in G^0$. We easily see that
\[ \min_{1 \leq j \leq k} \max_{1 \leq t \leq r}\|b_{1,j} \hat\phi_j(g_{\kkk_t}) b_{2,j}\| > \frac{\tau}{2} \|\hat\phi_j(g_{\kkk_t})\|\]
for all $((b_{1,j},b_{2,j}))_{j=1}^k \in \prod_{j=1}^k S_{\ned(W_j^0)}\times S_{\ned(W_j^0)}$, so if we define
\[\kappa:= \frac{\tau}{2} \min_{1 \leq j\leq k}\min_{1 \leq t \leq r}\|\hat\phi_j(g_{\kkk_t})\|\]
then we have proved our second claim.

We may now prove the theorem. Let $\iii,\jjj \in \Sigma_N^*$ be arbitrary. There exist $(W_j)_{j=1}^k, (W_j')_{j=1}^k \in \mathcal{W}$ such that
\[\Phi(\iii)=\prod_{j=1}^k \left\|\phi_j(g_\iii)|_{W_j}\right\|^{\beta_j},\qquad \Phi(\jjj)=\prod_{j=1}^k \left\|\phi_j(g_\jjj)|_{W_j'}\right\|^{\beta_j}.\]
By the hypothesis that $\mathcal{W}$ is primitive there exists $p \geq 1$ not depending on $\iii$, $\jjj$, $(W_j)_{j=1}^k$ or $(W_j')_{j=1}^k$ such that we may choose $\iii_1,\iii_2,\jjj_1,\jjj_2 \in \Sigma_N^*$ with $|\iii_1|=|\iii_2|=|\jjj_1|=|\jjj_2|=p$ satisfying
\begin{align*}(\phi_j(g_{\jjj_2})W_j^0)_{j=1}^k &= (W_j')_{j=1}^k,\\
(\phi_j(g_{\jjj_1})\phi_j(g_\jjj)W_j')_{j=1}^k &= (W_j^0)_{j=1}^k,\\
(\phi_j(g_{\iii_2})W_j^0)_{j=1}^k &= (W_j)_{j=1}^k,\\
(\phi_j(g_{\iii_1})\phi_j(g_\iii)W_j))_{j=1}^k&=(W_j^0)_{j=1}^k.\end{align*}
In particular $\phi_j(g_{\iii_1}g_{\iii}g_{\iii_2})|_{W_j^0}$ and $\phi_j(g_{\jjj_1}g_{\jjj}g_{\jjj_2})|_{W_j^0}$ are endomorphisms of $W_j^0$ for each $j=1,\ldots,k$, and
\[\prod_{j=1}^k \|\phi_j(g_{\iii_1}g_{\iii}g_{\iii_2})|_{W_j^0}\|^{\beta_j} \geq \left(\min_{|\ellellell|=p} \prod_{j=1}^k \|\phi_j(g_{\ellellell})^{-1}\|^{-\beta_j}\right)^2 \prod_{j=1}^k \left\|\phi_j(g_{\iii})|_{W_j}\right\|^{\beta_j} = \varepsilon \Phi(\iii)\]
and
\[\prod_{j=1}^k \|\phi_j(g_{\jjj_1}g_{\jjj}g_{\jjj_2})|_{W_j^0}\|^{\beta_j} \geq \left(\min_{|\ellellell|=p} \prod_{j=1}^k \|\phi_j(g_{\ellellell})^{-1}\|^{-\beta_j}\right)^2 \prod_{j=1}^k \left\|\phi_j(g_{\jjj})|_{W_j'}\right\|^{\beta_j} = \varepsilon \Phi(\jjj),\]
say, where
\[\varepsilon:=\left(\min_{|\ellellell|=p} \prod_{j=1}^k \|\phi_j(g_{\ellellell})^{-1}\|^{-\beta_j}\right)^2.\]
Now by the previous step there exists $\kkk \in \Sigma_N^*$ with $|\kkk|=m$ such that $g_\kkk \in G^0$ (and hence $\phi_j(g_\kkk)W_j^0=W_j^0$ for every $j=1,\ldots,k$) and such that 
\begin{eqnarray*}\lefteqn{\prod_{j=1}^k \left\|\left(\phi_j(g_{\iii_1}g_{\iii} g_{\iii_2})|_{W_j^0}\right)\left(\phi_j(g_{\kkk})|_{W_j^0}\right)\left( \phi_j(g_{\jjj_1}g_{\jjj} g_{\jjj_2})|_{W_j^0}\right) \right\|^{\beta_j}} & &\\
& & \geq \kappa^{\sum_{j=1}^k\beta_j}  \left(\prod_{j=1}^k \left\|\phi_j(g_{\iii_1}g_{\iii} g_{\iii_2})|_{W_j^0}\right\|^{\beta_j}\right)\left(\prod_{j=1}^k \left\|  \phi_j(g_{\jjj_1}g_{\jjj} g_{\jjj_2})|_{W_j^0} \right\|^{\beta_j}\right)\\
& & \geq \kappa^{\sum_{j=1}^k \beta_j} \varepsilon^2\Phi(\iii)\Phi(\jjj).\end{eqnarray*}
Defining $K:=\max_{|\ellellell|=p} \Phi(\ellellell)<\infty$ we have 
\begin{eqnarray*}
\lefteqn{\prod_{j=1}^k \left\|\left(\phi_j(g_{\iii_1}g_{\iii} g_{\iii_2})|_{W_j^0}\right)\left(\phi_j(g_{\kkk})|_{W_j^0}\right)\left( \phi_j(g_{\jjj_1}g_{\jjj} g_{\jjj_2})|_{W_j^0}\right) \right\|^{\beta_j}}& &\\
& &=\prod_{j=1}^k \left\|\phi_j(g_{\iii_1}g_{\iii} g_{\iii_2} g_{\kkk} g_{\jjj_1}g_{\jjj} g_{\jjj_2})|_{W_j^0} \right\|^{\beta_j}\\
& & \leq \max_{(W_j)_{j=1}^k \in \mathcal{W}} \prod_{j=1}^k \left\|\phi_j(g_{\iii_1}g_{\iii} g_{\iii_2} g_{\kkk} g_{\jjj_1}g_{\jjj} g_{\jjj_2})|_{W_j} \right\|^{\beta_j}\\
& &=\Phi(\iii_1\iii \iii_2 \kkk \jjj_1\jjj \jjj_2) \leq \Phi(\iii_1)\Phi(\jjj_2) \Phi(\iii \iii_2\kkk\jjj_1\jjj) \leq K^2 \Phi(\iii \iii_2\kkk \jjj_1 \jjj),\end{eqnarray*}
and since $|\iii_2 \kkk \jjj_1|=m+2p$, we have obtained
\[\max_{|\ellellell|=m+2p} \Phi(\iii \ellellell \jjj) \geq \Phi(\iii\iii_2\kkk\jjj_1\jjj) \geq K^{-2}\varepsilon^2\kappa^{\sum_{j=1}^k \beta_j} \Phi(\iii)\Phi(\jjj)\]
where $m$, $p$, $K$, $\kappa$ and $\varepsilon$ do not depend on $\iii,\jjj \in \Sigma_N^*$. The theorem is proved.
\end{proof}




\section{Proofs of main results}\label{se:bananas}

\subsection{Proof of Theorem \ref{th:main}}

Let the integers $k$ and $N$, vector spaces $V_j$, tuples $(A_1^{(j)},\ldots,A_N^{(j)})\in \GL(V_j)^N$, real numbers $\beta_j>0$, potential $\Phi \colon \Sigma_N^* \to (0,+\infty)$ and totally ergodic equilibrium state $\mu \in \mathcal{M}_\sigma(\Sigma_N)$ be as in the statement of Theorem \ref{th:main}. By Theorem \ref{th:reducks} we may without loss of generality assume that every $(A_1^{(j)},\ldots,A_N^{(j)})\in \GL(V_j)^N$ is irreducible and has simple top Lyapunov exponent with respect to $\mu$. By Theorem \ref{th:prox-unique}(i) we find that $\mu$ is the equilibrium state of a unique potential $\Phi_{\mathcal{W}}$ such that $\mathcal{W} \subseteq \prod_{j=1}^k \Gr_{\ell_j}(V_j)$ is a transitive subspace class, where for each $j=1,\ldots,k$ the integer $\ell_j$ is the dimension of the smallest nonzero subspace of $V_j$ with finite orbit under the action of $(A_1^{(j)},\ldots,A_N^{(j)})$. It follows by Theorem \ref{th:bomo}(i) that there exists $C>0$ such that
\[C^{-1}\Phi_{\mathcal{W}}(\iii) \leq e^{|\iii|P(\Phi_{\mathcal{W}})}\mu([\iii]) \leq C\Phi_{\mathcal{W}}(\iii)\]
for every $\iii \in \Sigma_N^*$. By Theorem \ref{th:prox-unique}(ii)--(iii) the potential $\Phi_{\mathcal{W}}$ satisfies the hypotheses of Theorem \ref{th:core0}, so there exist an integer $m$ and constant $\delta>0$ such that 
\[\max_{|\kkk|=m} \Phi_{\mathcal{W}}(\iii\kkk\jjj) \geq \delta \Phi_{\mathcal{W}}(\iii)\Phi_{\mathcal{W}}(\jjj)\]
for all $\iii,\jjj \in \Sigma_N^*$. Hence for every $\iii,\jjj \in \Sigma_N^*$ we have
\begin{align*}\delta \mu([\iii])\mu([\jjj]) &\leq C^{2}\delta e^{-(|\iii|+|\jjj|)P(\Phi_{\mathcal{W}})} \Phi_{\mathcal{W}}(\iii)\Phi_{\mathcal{W}}(\jjj)\\
& \leq C^{2}e^{-(|\iii|+|\jjj|)P(\Phi_{\mathcal{W}})} \max_{|\kkk|=m} \Phi_{\mathcal{W}}(\iii\kkk\jjj)\\
&\leq C^3e^{|\kkk|P(\Phi_{\mathcal{W}})}  \max_{|\kkk|=m} \mu([\iii\kkk\jjj]) \\
&\leq C^3e^{mP(\Phi_{\mathcal{W}})}  \sum_{|\kkk|=m} \mu([\iii\kkk\jjj])\\
&= C^3e^{mP(\Phi_{\mathcal{W}})} \mu([\iii]\cap \sigma^{-m-|\iii|}[\jjj])
\end{align*}
so that
\begin{equation}\label{eq:farty-red}\mu([\iii]\cap \sigma^{-m-|\iii|}[\jjj])\geq  \kappa \mu([\iii])\mu([\jjj])  \end{equation}
where $\kappa:=C^{-3}\delta e^{-mP(\Phi_{\mathcal{W}})}$, and also 
\begin{align*}\mu([\iii]\cap \sigma^{-m-|\iii|}[\jjj]) &=  \sum_{|\kkk|=m} \mu([\iii\kkk\jjj]) \\
&\leq C\sum_{|\kkk|=m}e^{-(|\iii|+|\kkk|+|\jjj|)P(\Phi_{\mathcal{W}})} \Phi_{\mathcal{W}}(\iii\kkk\jjj)\\
&\leq Ce^{-(|\iii|+|\jjj|)P(\Phi_{\mathcal{W}})} \Phi_{\mathcal{W}}(\iii)\Phi_{\mathcal{W}}(\jjj) \left(\sum_{|\kkk|=m} e^{-|\kkk|P(\Phi_{\mathcal{W}})}   \Phi_{\mathcal{W}}(\kkk)\right) \\
&\leq C^4\mu([\iii])\mu([\jjj]) \left(\sum_{|\kkk|=m} \mu([\kkk])\right)\\
& = C^4\mu([\iii])\mu([\jjj])\end{align*}
so that
\begin{equation}\label{eq:sandbork}\mu([\iii]\cap \sigma^{-m-|\iii|}[\jjj]) \leq K \mu([\iii])\mu([\jjj])\end{equation}
where $K:= C^4$. To prove the theorem we will combine inequalities \eqref{eq:farty-red} and \eqref{eq:sandbork} with theorems of R.C. Bradley, N.A. Friedman and D.S. Ornstein in a manner similar to earlier works such as \cite{Pi20,Wa05}.

For every integer $n \geq 1$ let $\mathcal{B}_{1,n}$ be the finite $\sigma$-algebra on $\hat\Sigma_N$ generated by the set
\[\left\{[\iii]\subset \hat\Sigma_N \colon |\iii|=n\right\}.\]
For every pair of integers $n,m \in \mathbb{Z}$ such that $n \leq m$ define $\mathcal{B}_{n,m}:=\hat\sigma^{n-1}\mathcal{B}_{1,m+1-n}$. Thus $\mathcal{B}_{n,m}$ is precisely the $\sigma$-algebra generated by cylinders of the form
\[\left\{(x_\ell)_{\ell \in \mathbb{Z}} \colon x_i = y_i\text{ for all }i = n,\ldots,m\right\}\]
where the finite sequence $(y_i)_{i=n}^m$ varies over $\{1,\ldots,N\}^{m-n+1}$. For every $n \in \mathbb{Z}$ define also 
\[\mathcal{B}_{-\infty,n}:=\bigvee_{m=-\infty}^n \mathcal{B}_{m,n},\qquad \mathcal{B}_{n,+\infty}:=\bigvee_{m=n}^\infty \mathcal{B}_{n,m}.\] 
The following is a special case of a theorem of R.C. Bradley (\cite[Theorem 4.1(2)]{Br05}):
\begin{theorem}\label{th:excuse-me-doctor-gunson-but-why-do-you-have-a-slice-of-ham-in-your-handbag}
Let $\hat\mu$ be a $\hat\sigma$-invariant measure on $\hat\Sigma_N$ such that for some integer $m \geq1$ the conditions
\[\inf_{\substack{A \in \mathcal{B}_{-\infty}^0, B \in \mathcal{B}_m^\infty\\ \hat\mu(A),\hat\mu(B) \neq 0}} \frac{\hat\mu(A\cap B)}{\hat\mu(A)\hat\mu(B)}>0,\qquad \sup_{\substack{A \in \mathcal{B}_{-\infty}^0, B \in \mathcal{B}_{m}^\infty\\ \hat\mu(A),\hat\mu(B) \neq 0}} \frac{\hat\mu(A\cap B)}{\hat\mu(A)\hat\mu(B)} <\infty\]
are both satisfied.  Then
\[\lim_{n \to \infty}  \sup_{\substack{A \in \mathcal{B}_{-\infty}^0, B \in \mathcal{B}_{n}^\infty\\ \hat\mu(A),\hat\mu(B) \neq 0}} \left|\frac{\hat\mu(A \cap B)}{\hat\mu(A)\hat\mu(B)}-1\right|=0.\]
\end{theorem}
Now, the natural extension $\hat\mu \in \mathcal{M}_{\hat\sigma}(\hat\Sigma_N)$ of the equilibrium state $\mu$ satisfies
\begin{align*}\inf_{\substack{A \in \mathcal{B}_{-\infty}^0, B \in \mathcal{B}_{m+1}^\infty\\ \hat\mu(A),\hat\mu(B) \neq 0}} \frac{\hat\mu(A \cap B)}{\hat\mu(A)\hat\mu(B)}
& = \inf_{\substack{n_1, n_2 \geq 1}}\inf_{\substack{A \in \mathcal{B}_{1-n_1,0}\\ B \in \mathcal{B}_{m+1,m+n_2} } }\frac{\hat\mu(A \cap B)}{\hat\mu(A)\hat\mu(B)}\\
& = \inf_{n_1,n_2 \geq 1}\inf_{\substack{A \in \mathcal{B}_{1,n_1}\\ B \in \mathcal{B}_{m+n_1+1,m+n_1+n_2} } }\frac{\hat\mu(A \cap B)}{\hat\mu(A)\hat\mu(B)}\\
&=\inf_{n_1,n_2 \geq 1}\inf_{\substack{|\iii|=n_1\\|\jjj|=n_2}} \frac{\hat\mu([\iii]\cap\sigma^{-m-|\iii|}[\jjj])}{\hat\mu([\iii])\hat\mu([\jjj])} \geq \kappa>0\end{align*}
by \eqref{eq:farty-red}, and likewise
\[\sup_{\substack{A \in \mathcal{B}_{-\infty}^0, B \in \mathcal{B}_{m+1}^\infty\\ \hat\mu(A),\hat\mu(B) \neq 0}} \frac{\hat\mu(A \cap B)}{\hat\mu(A)\hat\mu(B)}=\sup_{n_1,n_2 \geq 1}\sup_{\substack{|\iii|=n_1\\|\jjj|=n_2}} \frac{\hat\mu([\iii]\cap\sigma^{-m-|\iii|}[\jjj])}{\hat\mu([\iii])\hat\mu([\jjj])}\leq K<\infty\]
by \eqref{eq:sandbork}. Theorem \ref{th:excuse-me-doctor-gunson-but-why-do-you-have-a-slice-of-ham-in-your-handbag} therefore applies and yields 
\[\lim_{n \to \infty}  \sup_{\substack{A \in \mathcal{B}_{-\infty}^0, B \in \mathcal{B}_{n}^\infty\\ \hat\mu(A),\hat\mu(B) \neq 0}} \left|\frac{\hat\mu(A \cap B)}{\hat\mu(A)\hat\mu(B)}-1\right|=0\]
which clearly implies the result 
\[\lim_{n \to \infty} \sup_{\iii,\jjj \in \Sigma_N^*} \left|\frac{\mu([\iii] \cap \sigma^{-n-|\iii|} [\jjj])}{\mu([\iii])\mu([\jjj])} -1\right|=0\]
which is the first assertion of Theorem \ref{th:main}.

To deduce the Bernoulli property of $\hat\mu$ we will apply a theorem of N.A. Friedman and D.S. Ornstein. We recall that a measure space $(X,\mathcal{F},m)$ is called a Lebesgue space if it there exists a measure space isomorphism between $(X,\mathcal{F},m)$ and Lebesgue measure on a bounded interval equipped with the $\sigma$-algebra of Lebesgue measurable sets. If $(X,\mathcal{F},m)$ is a Lebesgue space, $Z \subseteq X$ has nonzero measure, $\mathcal{F}_Z:=\{A \cap Z \colon A \in \mathcal{F}\}$ and $m_Z$ is the measure on $(X,\mathcal{F}_Z)$ defined by $m_Z(A):=m(Z \cap A)$ then $(Z,\mathcal{F}_Z,m_Z)$ is also a Lebesgue space. 

If $T$ is an invertible measure-preserving transformation of a Lebesgue probability space $(X,\mathcal{F},m)$ then a partition $\mathcal{P}$ of $X$ is defined to be any finite set $\mathcal{P}=\{P_1,\ldots,P_n\}\subset \mathcal{F}$ such that, up to measure zero, $X$ is the disjoint union of the sets $P_1,\ldots,P_n$. Given a partition $\mathcal{P}$, for each $k \in \mathbb{Z}$ we may define a new partition $T^k \mathcal{P}:=\{T^kP_1,\ldots,T^kP_n\}$ of $X$ in the obvious fashion. If $\mathcal{P}_1,\ldots,\mathcal{P}_k$ are partitions then we let $\bigvee_{j=1}^k \mathcal{P}_j$ denote the partition $\{A_1 \cap A_2 \cap \cdots \cap A_k \colon A_j \in \mathcal{P}_j\}$. We will say that a partition $\mathcal{P}$ of $X$ is \emph{$\varepsilon$-independent} of a partition $\mathcal{Q}$ of $X$ if there exists a subset $\mathcal{Q}_\varepsilon$ of $\mathcal{Q}$ such that
\[m\left(\bigcup_{Q \in \mathcal{Q}_\varepsilon} Q\right)>1-\varepsilon\]
and
\[\max_{P \in \mathcal{P}} \max_{Q \in \mathcal{Q}_{\varepsilon}} \left|\frac{m(P \cap Q)}{m(Q)} - m(P)\right|<\varepsilon.\]
A partition $\mathcal{P}$ is called a \emph{weak Bernoulli partition} if for every $\varepsilon>0$ there exists an integer $k_\varepsilon \geq 0$ such that for all $n \geq 1$ the partition $\bigvee_{j=k_\varepsilon+1}^{k_\varepsilon +n} T^j \mathcal{P}$ is $\varepsilon$-independent of $\bigvee_{j=1-n}^0 T^j\mathcal{P}$, or equivalently if $\bigvee_{j=1-n}^{0}T^{j}\mathcal{P}$ is $\varepsilon$-independent of $\bigvee_{j=-2n-k_\varepsilon+1}^{-n-k_\varepsilon} T^{j} \mathcal{P}$.
The following theorem paraphrases a celebrated result of N.A. Friedman and D.S. Ornstein \cite{FrOr70}:
\begin{theorem}
Let $\mathbb{P}=(\sum_{i=1}^n p_i \delta_i)^{\mathbb{Z}}$ be a Bernoulli measure on $\hat\Sigma_N$ for some $N \geq 2$ and some nondegenerate probability vector $(p_1,\ldots,p_N)$ and let $\overline{\mathcal{B}}_{\mathbb{P}}$ denote the completion of the Borel $\sigma$-algebra on $\hat\Sigma_N$ with respect to $\mathbb{P}$. Let $T$ be an invertible measure-preserving transformation of a Lebesgue probability space $(X,\mathcal{F},m)$ which admits a weak Bernoulli partition. Then there exists a measure space isomorphism $\phi \colon X \to \hat\Sigma_N$ such that $\phi \circ T = \hat\sigma \circ \phi$.
 \end{theorem}
 If $\mu$ is a totally ergodic generalised matrix equilibrium state on $\Sigma_N$ for some $N \geq 2$, let $\hat\mu$ denote its natural extension to $\hat\Sigma_N$ and let $\overline{\mathcal{B}}_{\hat\mu}$ denote the completion of the Borel $\sigma$-algebra on $\hat\Sigma_N$ with respect to $\hat\mu$. Let $\mathcal{P}$ denote the partition of $\hat\Sigma_N$ into  the $N$ sets $[i]:=\{(x_\ell)_{\ell \in \mathbb{Z}} \colon x_1=i\}$ for $i=1,\ldots,N$. Since
\[\lim_{n \to \infty} \sup_{\iii,\jjj \in \Sigma_N^*} \left|\frac{\hat\mu([\iii] \cap \hat\sigma^{-n-|\iii|} [\jjj])}{\hat\mu([\iii])\hat\mu([\jjj])} -1\right|=0,\]
for every $\varepsilon>0$ we may choose $k_\varepsilon \geq 1$ such that  
\[ \sup_{\iii,\jjj \in \Sigma_N^*} \left|\frac{\hat\mu([\iii] \cap \hat\sigma^{-k_\varepsilon-|\iii|} [\jjj])}{\hat\mu([\iii])\hat\mu([\jjj])} -1\right|<\varepsilon.\]
For each $n \geq 1$ the partitions  $\bigvee_{j=1-n}^{0} \hat\sigma^{j} \mathcal{P}$ and $\bigvee_{j=-2n-k_\varepsilon+1}^{-n-k_\varepsilon}\hat\sigma^{j}\mathcal{P}$ are simply the partitions into sets of the form $[\iii]$ and into sets of the form $\sigma^{-k_\varepsilon-|\iii|} [\iii]$ respectively, where $|\iii|=n$. The $\varepsilon$-independence of the first partition from the second is immediate and we conclude that $\mathcal{P}$ is a weak Bernoulli partition for the transformation $\hat\sigma$ of $(\hat\Sigma_N,\overline{\mathcal{B}}_{\hat\mu},\hat\mu)$. Applying the theorem of Friedman and Ornstein proves the second assertion of Theorem \ref{th:main}.

\subsection{Proof of Theorem \ref{th:recoding}}

Before starting the proof we require the following simple lemma:
\begin{lemma}\label{le:elric-the-hedgehog}
Let $T \colon X \to X$ be an ergodic  measure-preserving transformation of a probability space $(X,\mathcal{F},\mu)$ which is not totally ergodic, and let $n>1$ be the smallest integer such that $T^n$ is not ergodic. Then there exists a measurable set $Z \subset X$ such that $Z, T^{-1}Z, \ldots, T^{-(n-1)}Z$ partitions $X$ up to measure zero and satisfies $T^{-n}Z=Z$ up to measure zero.\end{lemma}
\begin{proof}
Since $T^n$ is not ergodic there by definition exists a measurable set $Y \subset X$ such that $0<\mu(Y)<1$ and $T^{-n}Y=Y$ up to $\mu$-measure zero. Consider a set $R\subseteq \{0,\ldots,n-1\}$ with the properties $0 \in R$ and $\mu(\bigcap_{i \in R} T^{-i}Y)>0$ and which has maximum cardinality of all such sets. (Clearly at least one set with those two properties exists, namely $\{0\}$, so $R$ is well-defined.) Define $Z:=\bigcap_{i \in R} T^{-i}Y$ and note that clearly $T^{-n}Z=Z$ up to $\mu$-measure zero. 
For each $i \in \{1,\ldots,n-1\}$ we must have either $\mu(Z \cap T^{-i}Z)=0$ or $R=R+i \mod n$, since if neither of these holds then $R \cup (R+i \mod n)$ would have larger cardinality than $R$ while having the same characteristic properties, contradicting maximality. But if $R=R+i \mod n$ then $T^{-i}Z=Z$ up to measure zero which implies that $T^i$ is not ergodic, contradicting the definition of $n$. It follows that $\mu(Z \cap T^{-i}Z)=0$ for all $i=1,\ldots,n-1$ and hence by the $T$-invariance of $\mu$ we deduce that $T^{-i}Z$ and $T^{-j}Z$ are pairwise disjoint up to measure zero whenever $0 \leq i<j<n$. We also have $\bigcup_{i=0}^{n-1}T^{-i}Z=X$ up to measure zero since this set is $T$-invariant and has positive measure and $T$ is ergodic with respect to $\mu$, and this completes the proof.
\end{proof}
The core of the proof of Theorem \ref{th:recoding} is contained in the following result, which will be used twice in the proof.
\begin{proposition}\label{pr:od}
Let $k \geq 1$ and $N \geq 2$. For each $j=1,\ldots,k$ let $V_j$ be a finite-dimensional real vector space and let $(A_1^{(j)},\ldots,A_N^{(j)}) \in \GL(V_j)^N$ and $\beta_j>0$. For all $\iii \in \Sigma_N^*$ define
\[\Phi(\iii):=\prod_{j=1}^k \left\|A_\iii^{(j)}\right\|^{\beta_j}\]
and let $\mu$ be an ergodic equilibrium state of $\Phi$. 

Suppose that there exist an integer $n>1$ and Borel set $Z \subset \Sigma_N$ such that $T^{-n}Z=Z$ and $Z,T^{-1}Z,\ldots,T^{-(n-1)}Z$ is a partition of $\Sigma_N$, both up to $\mu$-measure zero. Define a measure $\nu$ on $\Sigma_N$ by $\nu(A):=\mu(A \cap Z)/\mu(Z)$ for all Borel sets $A \subseteq \Sigma_N$. Let $\eta \colon  \{\iii \in \Sigma_N^* \colon |\iii|=n\}\to \{1,\ldots,N^n\}$ be the map which takes each word $\iii \in \Sigma_N^*$ of length $n$ to the integer representing its position in the lexicographical ordering on $\{\iii \in \Sigma_N^* \colon |\iii|=n\}$, and define a homeomorphism $\iota \colon \Sigma_{N} \to \Sigma_{N^n}$ by $\iota[(x_\ell)_{\ell=1}^\infty]:=(\eta(x_{(q-1)n+1}\cdots x_{qn}))_{q=1}^\infty$.  For each $j=1,\ldots,k$ define an $N^n$-tuple $(B_1^{(j)},\ldots,B_{N^n}^{(j)}) \in \GL(V_j)^{N^n}$ by $B_i^{(j)}:=A_{\eta^{-1}(i)}^{(j)}$ for every $i=1,\ldots,N^n$ and $j=1,\ldots,k$, and define a potential $\Psi \colon \Sigma_{N^n}^* \to(0,+\infty)$ by
\[\Psi(\jjj)=\prod_{j=1}^k \left\|B_\jjj^{(j)}\right\|^{\beta_j}\]
for all $\jjj \in \Sigma_{N^n}^*$.

Then $\mu = \frac{1}{n}\sum_{i=0}^{n-1} \sigma^i_* \nu$, each measure $\sigma^i_*\nu$ is $\sigma^n$-invariant, the measures $(\iota \circ \sigma^i)_*\nu \in \mathcal{M}_\sigma(\Sigma_{N^n})$ are pairwise mutually singular equilibrium states of $\Psi$, and $n \leq \prod_{j=1}^k \dim V_j$.
\end{proposition}
\begin{proof}
Clearly the properties of $Z$ imply $\mu(Z)=1/n$. For every Borel set $A \subseteq \Sigma_N$ we have
\begin{align*}\frac{1}{n}\sum_{i=0}^{n-1} (\sigma^i_*\nu)(A)= \frac{1}{n}\sum_{i=0}^{n-1} \frac{\mu(\sigma^{-i} A\cap Z)}{\mu(Z)} & = \sum_{i=0}^{n-1} \mu(\sigma^{-i}A \cap Z)\\
&= \sum_{i=0}^{n-1}\mu(A \cap \sigma^{-(n-i)}Z) = \mu(A) \end{align*}
so that $\mu = \frac{1}{n}\sum_{i=0}^{n-1} \sigma^i_* \nu$, and similarly
\[(\sigma^i_*\nu)(\sigma^{-n}A)=\frac{\mu(\sigma^{-i-n}A \cap Z)}{\mu(Z)}=\frac{\mu(\sigma^{-i-n} \cap \sigma^{-n}Z)}{\mu(Z)} =\frac{\mu(\sigma^{-i} \cap Z)}{\mu(Z)} =(\sigma^i_*\nu)(A)\]
so that each  $\sigma^i_*\nu$ is $\sigma^n$-invariant. The equation $\iota \circ \sigma^n = \sigma \circ \iota$ is obvious from the definition of $\iota$.
 It follows directly that each $(\iota \circ \sigma^i)_*\nu$ is a $\sigma$-invariant measure on $\Sigma_{N^n}$. It is easy to check from the construction of $\nu$ that the sets which have nonzero measure with respect to $\sigma^i_*\nu$ are precisely those which intersect $\sigma^{-(n-i)}Z$ in a set of nonzero $\mu$-measure, and since the sets $Z,\sigma^{-1}Z,\ldots,\sigma^{-(n-1)}Z$ are pairwise disjoint up to $\mu$-measure zero this implies that the measures $\sigma^i_*\nu$ are pairwise mutually singular for distinct $i \in \{0,\ldots,n-1\}$. Since $\iota$ is a homeomorphism this implies that the measures $(\iota \circ \sigma^i)_*\nu$ are pairwise mutually singular for distinct $i \in \{0,\ldots,n-1\}$. By Theorem \ref{th:bomo2} there can be at most $ \prod_{j=1}^k \dim V_j$ distinct ergodic equilibrium states for $\Psi$, which are necessarily pairwise mutually singular since they are distinct ergodic measures. By standard ergodic decomposition arguments every equilibrium state of $\Psi$ arises as a convex combination of these ergodic equilibrium states. It follows from this that the cardinality of a set of pairwise mutually singular equilibrium states of $\Psi$ cannot be larger than the cardinality of the set of ergodic equilibrium states of $\Psi$, which is bounded by $\prod_{j=1}^k \dim V_j$. Thus if we can show that every  $(\iota \circ \sigma^i)_*\nu$ is an equilibrium state of $\Psi$ then the bound $n \leq \prod_{j=1}^k \dim V_j$ follows and we will have proved the proposition. But it follows easily from the definition of equilibrium state and the fact that $h(\cdot)$ and $\Lambda(\Psi,\cdot)$ are affine functions that if a finite convex combination of invariant measures is an equilibrium state of $\Psi$, then so must be the measures which are the summands in the convex combination. So to complete the proof we need only show that the invariant measure $\iota_*\mu = \frac{1}{n}\sum_{i=0}^{n-1} (\iota \circ \sigma^i)_* \nu \in \mathcal{M}_\sigma(\Sigma_{N^n})$ is an equilibrium state of $\Psi$. But this is a straightforward calculation: the equation $\iota \circ \sigma^n=\sigma \circ \iota$ and the fact that $\iota$ is a homeomorphism together imply that $h(\iota_*\mu)=nh(\mu)$ by basic ergodic theory, and we have
\[P(\Psi)= \lim_{m \to \infty} \frac{1}{m} \log \sum_{\jjj \in \Sigma_{N^n}^* \colon |\jjj|=m} \Psi(\iii) = \lim_{m \to \infty} \frac{1}{m}\log\sum_{\jjj \in \Sigma_N^* \colon |\jjj|=nm} \Phi(\iii)= nP(\Phi)\] 
and $\Lambda(\Psi,\iota_*\mu)=n\Lambda(\Phi,\mu)$ by an almost identical calculation. The result follows.
 \end{proof}

We may now prove Theorem \ref{th:recoding}.
Let $k$, $N$, $V_j$, $(A_1^{(j)},\ldots,A_N^{(j)})$, $\beta_j$ and $\mu$ be as in the statement of the theorem. Since $\mu$ is not totally ergodic, it follows from Lemma \ref{le:elric-the-hedgehog} that there exists an integer $n>1$ with the property that there exists a measurable set $Z \subset X$ such that $Z, T^{-1}Z, \ldots, T^{-(n-1)}Z$ partitions $X$ up to $\mu$-measure zero and such that $T^{-n}Z=Z$ up to measure zero. Proposition \ref{pr:od} implies that every integer $n$ with this property is less than or equal to $\prod_{j=1}^k \dim V_j$, so we may choose a \emph{largest} such integer. Let $n$ be the largest integer with the aforementioned property, which clearly satisfies $1 < n \leq \prod_{j=1}^k \dim V_j$. Let $Z \subset \Sigma_N$ be a Borel set with the property that $Z, T^{-1}Z, \ldots, T^{-(n-1)}Z$ partitions $\Sigma_N$ up to $\mu$-measure zero and define a Borel probability measure $\nu$ on $\Sigma_N$ by $\nu(A):=\mu(A \cap Z)/\mu(Z)$ for all Borel sets $A \subseteq \Sigma_N$. By Proposition \ref{pr:od} there exist tuples $(B_1^{(j)},\ldots,B_{N^n}^{(j)}) \in \GL(V_j)^{N^n}$ and a potential $\Psi \colon \Sigma_{N^n}^* \to (0,+\infty)$ as in the statement of Theorem \ref{th:recoding} such that each of the measures $(\iota \circ \sigma^i)_*\nu \in \mathcal{M}_\sigma(\Sigma_{N^n})$ is a distinct equilibrium state of $\Psi$, and we have $\mu = \frac{1}{n}\sum_{i=0}^{n-1} \sigma^i_* \nu$.

To complete the proof of Theorem \ref{th:recoding} we must show that for every $i \in \{0,\ldots,n-1\}$ the measure $(\iota \circ \sigma^i)_*\nu$ is totally ergodic. Fix such an $i$. We will first show that $(\iota \circ \sigma^i)_*\nu$ is ergodic. If this is not the case then there exists a Borel set $A \subset \Sigma_{N^n}$ such that $\sigma^{-1}A=A$ up to $(\iota \circ \sigma^i)_*\nu$-measure zero and such that $0<((\iota \circ \sigma^i)_*\nu)(A)<1$. Define $B:=\iota^{-1}A \subset \Sigma_N$ so that $B=\iota^{-1}A=\iota^{-1}\sigma^{-1}A=\sigma^{-n}\iota^{-1}A=\sigma^{-n}B$ up to $\sigma^i_*\nu$-measure zero, and $0<(\sigma^i_*\nu)(B)<1$. We have $0<\mu(\sigma^{-i}B\cap Z)<\mu(Z)=\frac{1}{n}$ and $\sigma^{-n}(\sigma^{-i}B \cap Z)=\sigma^{-i}B \cap Z$ up to $\mu$-measure zero. But then $\bigcup_{j=0}^{n-1} \sigma^{-j}(\sigma^{-i}B \cap Z)$ is $\sigma$-invariant up to $\mu$-measure zero but has measure strictly between $0$ and $1$, contradicting the ergodicity of $\mu$. 

We may now show that  $(\iota \circ \sigma^i)_*\nu$ is totally ergodic. If it is not then by Lemma \ref{le:elric-the-hedgehog} there exist an integer $m>1$ and a Borel subset $A$ of $\Sigma_{N^n}$ such that $\sigma^{-m} A=A$ up to $(\iota \circ \sigma^i)_*\nu$-measure zero and such that $A,\sigma^{-1}A,\ldots,\sigma^{-(m-1)}A$ forms a partition of $\Sigma_{N^n}$ up to $(\iota \circ \sigma^i)_*\nu$-measure zero. Define $B:=\iota^{-1}A \subset \Sigma_N$; then $B,\sigma^{-n}B,\sigma^{-2n}B\ldots,\sigma^{-(m-1)n}B$ forms a partition of $\Sigma_N$ up to $\sigma^i_*\nu$-measure zero, and $B=\sigma^{-mn}B$ up to $\sigma^i_*\nu$-measure zero. This implies that the sets $\sigma^{-i}B,\sigma^{-n-i}B,\sigma^{-2n-i}B\ldots,\sigma^{-(m-1)n-i}B$ form a partition of $Z$ up to $\mu$-measure zero and are all $\sigma^{-mn}$-invariant up to $\mu$-measure zero. But then
\[\bigcup_{\ell=0}^{mn-1} \sigma^{-\ell}B = \bigcup_{j=0}^{n-1} \sigma^{-j} \left(\bigcup_{r=0}^{m-1} \sigma^{-i-rn}B\right)=\bigcup_{j=0}^{n-1}\sigma^{-j}Z=\Sigma_N\]
up to $\mu$-measure zero, and all of these unions are disjoint up to $\mu$-measure zero, which contradicts the maximality of $n$.  This completes the proof that each $(\iota \circ \sigma^i)_*\nu$ is totally ergodic and completes the proof of the theorem.

\subsection{Proof of Corollary \ref{co:vid19}}

Let $\mu$ be an ergodic generalised matrix equilibrium state on $\Sigma_N$. If $\mu$ is totally ergodic then the conclusion follows by Theorem \ref{th:main}, so suppose that $\mu$ is not totally ergodic. By Theorem \ref{th:recoding} there exist $n>1$ and a $\sigma^n$-invariant measure $\nu$ on $\Sigma_N$, which is totally ergodic with respect to $\sigma^n$ and is measurably isomorphic via a homeomorphism $\iota \colon \Sigma_N \to \Sigma_{N^n}$ satisfying $\iota \circ \sigma^n = \sigma \circ \iota$ to a generalised matrix equilibrium state on $\Sigma_{N^n}$, such that $\mu =\frac{1}{n}\sum_{i=0}^{n-1}\sigma^i_* \nu$. In particular we may write $\hat\mu=\frac{1}{n}\sum_{i=0}^{n-1}\hat\sigma^i_*\hat\nu$ where each $\hat\sigma^i_*\hat\nu$ is a distinct ergodic measure with respect to the transformation $\hat\sigma^n$ and where $\hat\nu$ has the Bernoulli property with respect to the transformation $\hat\sigma^n$ as a consequence of Theorem \ref{th:main}. Since the measures $\hat\sigma^i_*\hat\nu$ are distinct ergodic measures they are pairwise mutually singular, so there exists $Z \subset \hat\Sigma_N$ such that $\hat\nu(Z)=1$ and $\hat\nu(\hat\sigma^i Z)=0$ for all $i \in \{1,\ldots,n-1\}$, and this set satisfies $\hat\sigma^nZ=Z$ up to $\hat\mu$-measure zero by the $\hat\sigma^n$-invariance of the measure $\hat\nu$. It follows that $\hat\sigma_N = Z \cup \hat\sigma^i Z \cup \cdots \cup \hat\sigma^{n-1}Z$ up to $\hat\mu$-measure zero and that these sets are pairwise disjoint up to $\hat\mu$-measure zero. By virtue of the equation $\hat\mu=\frac{1}{n}\sum_{i=0}^{n-1}\hat\sigma^i_*\hat\nu$, the measure $\hat\nu$ must be precisely the measure $\hat\mu_Z$ on $Z$ defined by $\hat\mu_Z(A):=\hat\mu(A \cap Z)/\hat\mu(Z)$ for all Borel sets $A \subseteq \hat\Sigma_N$.

Let $\mathbb{P}$ be a Bernoulli measure on $\hat\Sigma_N$ which has the same entropy as $\hat\mu$. Let $\overline{\hat{\mathcal{B}}_{\hat\nu}}$ and $\overline{\hat{\mathcal{B}}_{\mathbb{P}}}$ denote the completion of the Borel $\sigma$-algebra on $\hat\Sigma_N$ with respect to the measures $\hat\nu$ and $\mathbb{P}$ respectively. Since $\hat\nu$ and $\mathbb{P}$ both have the Bernoulli property with respect to $\hat\sigma^n$, and both have the same entropy as $\hat\mu$ with respect to $\hat\sigma^n$, they are measurably isomorphic, so there exists a measure space isomorphism $\phi$ from $(\hat\Sigma_N, \overline{\hat{\mathcal{B}}_{\hat\nu}}, \hat\nu)$ to $(\hat\Sigma_N,\overline{\hat{\mathcal{B}}_{\mathbb{P}}}, \mathbb{P})$ such that $\phi \circ \hat\sigma^n = \hat\sigma^n \circ \phi$ and $\phi_*\hat\nu = \mathbb{P}$. 

Let $\mathbb{Z}_n$ denote the set $\{0,\ldots,n-1\}$ equipped with addition modulo $n$ and define a transformation $T \colon \hat\Sigma_N \times \mathbb{Z}_n\to \hat\Sigma_N \times \mathbb{Z}_n$ by $T(x,i):=(\hat\sigma x,i+1 \mod n)$. To prove the corollary we must construct a measure space isomorphism $\psi$ from $\hat\Sigma_N$ to $\hat\Sigma_N \times \mathbb{Z}_n$ which satisfies $\psi \circ \hat\sigma = T\circ \psi$ and $\psi_* \hat\mu = \mathbb{P} \times (\frac{1}{n}\sum_{i=0}^{n-1} \delta_i)$. To this end define $\psi \colon \hat\Sigma_N \to \hat\Sigma_N \times \mathbb{Z}_n$ by $\psi(x)=(\hat\sigma^i \phi(\hat\sigma^{-i}x),i)$ whenever $x \in \hat\sigma^iZ$ for some $i \in \mathbb{Z}_n$, and define $\psi(x)$ to be an arbitrary constant value in $\hat\Sigma_N \times \mathbb{Z}_n$ for all $x \in \hat\Sigma_N \setminus \bigcup_{i=0}^{n-1}\hat\sigma^iZ$. When $x \in \hat\sigma^iZ$ for $i \in \{0,\ldots,n-2\}$ we have
\[T(\psi(x))=( \hat\sigma^{i+1}\phi(\hat\sigma^{-i}x),i+1) = \psi(\hat\sigma x)\]
and when $x \in \hat\sigma^{n-1}$ we have
\[T(\psi(x))=( \hat\sigma^{n}\phi(\hat\sigma^{-(n-1)}x),0) =(\phi(\hat\sigma x),0)= \psi(\hat\sigma x)\]
so that $T\circ \psi = \psi \circ \hat\sigma$ almost everywhere with respect to $\hat\mu$. By construction we have $\phi_* \hat\mu_Z=\mathbb{P}$ and consequently $\psi_*\hat\nu= \mathbb{P} \times \delta_0$. It follows directly that 
\[\psi_*\mu = \psi_* \left(\frac{1}{n}\sum_{i=0}^{n-1} \hat\sigma_*^i \mu_Z\right) = \frac{1}{n}\sum_{i=0}^{n-1} T^i_* (\mathbb{P}\times \delta_0) = \mathbb{P}\times \left(\frac{1}{n}\sum_{i=0}^{n-1}\delta_i\right)\]
as required. This completes the construction of the isomorphism $\psi$ and proves the corollary.

\subsection{Proof of Proposition \ref{pr:not-tot}}

Fix $\alpha,\beta>0$ throughout the proof. It is obvious that $(A_1,A_2)$ and $(B_1,B_2)$ are irreducible since every nonzero proper subspace of $\mathbb{R}^2$ is one-dimensional but neither $A_2$ nor $B_1$ preserves any one-dimensional subspace of $\mathbb{R}^2$. On the other hand it is obvious that the horizontal and vertical axes in $\mathbb{R}^2$ both have finite orbit under the action of $(A_1,A_2)$ and similarly for $(B_1,B_2)$. If $\mu \in \mathcal{M}_\sigma(\Sigma_2)$ is an ergodic equilibrium state of the potential $\Phi \colon \Sigma_2^* \to (0,+\infty)$ defined by
\[\Phi(\iii):=\left\|A_\iii\right\|^\alpha \left\|B_\iii\right\|^\beta\]
then it follows by Theorem \ref{th:bomo} and the preceding observations that there exists a transitive subspace class $\mathcal{W}\subset \Gr_1(\mathbb{R}^2) \times \Gr_1(\mathbb{R}^2)$ such that $\mu$ is the unique equilibrium state of the potential
\[\Phi_{\mathcal{W}}(\iii):=\max_{(W_1,W_2) \in \mathcal{W}}\left\|A_\iii|_{W_1}\right\|^\alpha \left\|B_\iii|_{W_2}\right\|^\beta.\]
We claim that there exists a \emph{unique} transitive subspace class preserved by these pairs of matrices, which is the set
\[\mathcal{W}_0:=\{(\overline{e_1},\overline{e_1}), (\overline{e_1},\overline{e_2}), (\overline{e_2},\overline{e_1}), (\overline{e_2},\overline{e_2})\}.\]
(Here $e_1,e_2$ denotes the standard basis for $\mathbb{R}^2$ and $\overline{u}$ the one-dimensional subspace spanned by the nonzero vector $u$.) Indeed, if $\overline{u}\subset \mathbb{R}^2$ is a one-dimensional space with finite orbit under $(A_1,A_2)$ then the set $\{\overline{A^n_1 u} \colon n \geq 1\}$ must be finite, but this is the case only when $\overline{u} \in \{\overline{e_1},\overline{e_2}\}$; similarly $\{\overline{B^n_2 u} \colon n \geq 1\}$ is finite only when $\overline{u} \in \{\overline{e_1},\overline{e_2}\}$; we conclude that every equivariant subspace class must be a subset of $\mathcal{W}_0$. On the other hand it is easy to see that $\mathcal{W}_0$ is transitive: we may apply the symbol $1$ to pass from $(\overline{e_1},\overline{e_1})$ to $(\overline{e_1},\overline{e_2})$ and vice versa, or from $(\overline{e_2},\overline{e_1})$ to $(\overline{e_2},\overline{e_2})$ and vice versa, and we may apply the symbol $2$ to pass from $(\overline{e_1},\overline{e_1})$ to $(\overline{e_2},\overline{e_1})$ and vice versa, or from $(\overline{e_1},\overline{e_2})$ to $(\overline{e_2},\overline{e_2})$ and vice versa. Thus every element of $\mathcal{W}_0$ can be reached from any other element via a word of length $1$ or $2$. We conclude that it contains a unique transitive subspace class, which is $\mathcal{W}_0$ itself. It follows that the potential $\Phi$ above has a unique equilibrium state, namely the unique equilibrium state of the potential $\Phi_{\mathcal{W}_0}$ as defined above. Let us denote this unique equilibrium state by $\mu$. We wish to show that $\mu$ is not totally ergodic.

Define tuples
\[(\hat{A}_1,\hat{A}_2,\hat{A}_3,\hat{A}_4):=(A_1A_1,A_1A_2,A_2A_1,A_2A_2),\]
\[(\hat{B}_1,\hat{B}_2,\hat{B}_3,\hat{B}_4):=(B_1B_1,B_1B_2,B_2B_1,B_2B_2)\]
so that
\[\hat{A}_1=\begin{pmatrix}4&0\\0&1\end{pmatrix},\quad \hat{A}_2=\begin{pmatrix}0&2\\1&0\end{pmatrix},\quad  \hat{A}_3=\begin{pmatrix}0&1\\2&0\end{pmatrix},\quad  \hat{A}_4=\begin{pmatrix}1&0\\0&1\end{pmatrix},\]
\[\hat{B}_1=\begin{pmatrix}1&0\\0&1\end{pmatrix},\quad \hat{B}_2=\begin{pmatrix}0&2\\1&0\end{pmatrix},\quad  \hat{B}_3=\begin{pmatrix}0&1\\2&0\end{pmatrix},\quad  \hat{B}_4=\begin{pmatrix}1&0\\0&4\end{pmatrix} \]
and define a potential $\hat\Phi \colon \Sigma_4^* \to (0,+\infty)$ by
\[\hat\Phi(\iii):=\left\|\hat{A}_\iii\right\|^\alpha \left\|\hat{B}_\iii\right\|^\beta.\]
By arguments similar to that used in Theorem \ref{th:recoding} we may define a recoding homeomorphism $\iota \colon \Sigma_2 \to \Sigma_4$ such that $\iota \circ \sigma^2 = \sigma \circ \iota$ and such that $\hat\Phi(\iota(x)|_n)= \Phi(x|_{2n})$ for every $x \in \Sigma_2$. Easy calculations similar to those occurring in the proof of Proposition \ref{pr:od} show directly that $P(\hat\Phi)=2P(\Phi)$ and that $\iota_*\mu$ is an equilibrium state of $\hat\Phi$. To prove that $\mu$ is not totally ergodic we will show that $\iota_*\mu$ is not ergodic with respect to $\sigma \colon \Sigma_4 \to \Sigma_4$, which combined with the identity $\iota \circ \sigma^2 = \sigma \circ \iota$ implies immediately that $\mu$ is not ergodic with respect to $\sigma^2 \colon \Sigma_2 \to \Sigma_2$. Since $\iota_*\mu$ is an equilibrium state of $\hat\Phi$ it will suffice for us to identify the ergodic equilibrium states of $\hat\Phi$ and show that $\iota_*\mu$ cannot be equal to any of them.


Define
\[\mathcal{W}_1:=\{(\overline{e_1},\overline{e_1}), (\overline{e_2},\overline{e_2})\},\]
\[\mathcal{W}_2:=\{(\overline{e_1},\overline{e_2}), (\overline{e_2},\overline{e_1})\}.\]
We observe that both $\mathcal{W}_1$ and $\mathcal{W}_2$ are transitive subspace classes for $(\hat{A}_1,\hat{A}_2,\hat{A}_3,\hat{A}_4)$ and $(\hat{B}_1,\hat{B}_2,\hat{B}_3,\hat{B}_4)$: the symbols $1$ and $4$ fix every pair $(\overline{e_i},\overline{e_j})$ and the symbols $2$ and $3$ swap $(\overline{e_1},\overline{e_1})$ with $(\overline{e_2},\overline{e_2})$ and swap $(\overline{e_1},\overline{e_2})$ with $(\overline{e_2},\overline{e_1})$. Similarly to our analysis of $\Phi$, since the co-ordinate axes are the only one-dimensional subspaces which have finite orbit under $\hat{A}_1$ and $\hat{B}_4$, every transitive subspace class for $(\hat{A}_1,\hat{A}_2,\hat{A}_3,\hat{A}_4)$ and $(\hat{B}_1,\hat{B}_2,\hat{B}_3,\hat{B}_4)$ must be a subset of $\mathcal{W}_0$. It follows that there exist exactly \emph{two} transitive subspace classes for these tuples, $\mathcal{W}_1$ and $\mathcal{W}_2$. Define potentials $\hat\Phi_1,\hat\Phi_2 \colon \Sigma_4^* \to (0,+\infty)$ by
\[\hat\Phi_{i}(\iii):=\max_{(W_1,W_2) \in \mathcal{W}_i}\left\|\hat{A}_\iii|_{W_1}\right\|^\alpha \left\|\hat{B}_\iii|_{W_2}\right\|^\beta.\]
for $i=1,2$. By Theorem \ref{th:bomo} each of $\hat\Phi_1$ and $\hat\Phi_2$ has a unique equilibrium state which we denote respectively $\mu_1$ and $\mu_2$, and furthermore every ergodic equilibrium state of $\hat\Phi$ must be equal to one of these two measures. In particular at least one of the two measures is an equilibrium state for $\hat\Phi$. Straightforward checking of definitions demonstrates that for $i=1,2$ the measure $\mu_i$ is an equilibrium state of $\hat\Phi$ if and only if $P(\hat\Phi_i)=P(\hat\Phi)$.

Suppose for a contradiction that $\mu$ is totally ergodic. In particular $\mu$ is ergodic with respect to $\sigma^2 \colon \Sigma_2 \to \Sigma_2$ and therefore $\iota_*\mu$ is ergodic with respect to $\sigma \colon \Sigma_4 \to \Sigma_4$. Hence $\iota_*\mu$ is an ergodic equilibrium state of $P(\hat\Phi)$ and there exists $i \in \{1,2\}$ such that $\iota_*\mu=\mu_i$ and $P(\hat\Phi_i)=P(\hat\Phi)$. It follows from Theorem \ref{th:bomo} that there exists $C_1>0$ such that
\[C_1^{-1} e^{-nP(\Phi)}\Phi(x|_n) \leq \mu([x|_n]) \leq C_1e^{-n P(\Phi)}\Phi(x|_n)\]
for every $n \geq 1$ and $x \in \Sigma_2$, and also that there exists $C_2>0$ such that 
\[C^{-1}_2 e^{-nP(\hat\Phi_i)}\hat\Phi_i(z|_n) \leq \mu_i([z|_n]) \leq C_2e^{-n P(\hat\Phi_i)}\hat\Phi_i(z|_n)\]
for every $z \in \Sigma_4$. Recoding via $\iota$, the former inequalities imply 
\[C^{-1}_1 e^{-nP(\hat\Phi)}\hat\Phi(z|_n) \leq \left(\iota_*\mu\right)([z|_n])\leq C_1e^{-n P(\hat\Phi)}\hat\Phi(z|_n)\]
for every $z \in \Sigma_4$ and $n \geq 1$, and since $\iota_*\mu=\mu_i$ and $P(\hat\Phi_i)=P(\hat\Phi)$ by hypothesis we conclude that necessarily 
\begin{equation}\label{eq:bunflow} C_1^{-1}C_2^{-1} \leq \frac{\hat\Phi(\iii)}{\hat\Phi_i(\iii)}\leq C_1C_2\end{equation}
for every $\iii \in \Sigma_4^*$. But if  $\iii = 1^n4^n$ then 
\[\hat{A}_\iii = \hat{A}_1^n \hat{A}_4^n  = A_1^{2n} A_2^{2n} = \begin{pmatrix}4^n&0\\ 0&1\end{pmatrix},\]
\[\hat{B}_\iii = \hat{B}_1^n \hat{B}_4^n = B_1^{2n} B_2^{2n} = \begin{pmatrix}1&0\\ 0&4^n\end{pmatrix}\]
and $\hat\Phi_1(\iii)= \max\{4^{n\alpha},4^{n\beta}\}$ whereas $\hat\Phi(\iii)=4^{n(\alpha+\beta)}$, so \eqref{eq:bunflow} cannot hold for $i=1$; and if $\jjj = 31^{n-1} 24^{n-1}$ for some $n \geq 1$ then
\begin{align*}\hat{A}_\jjj = \hat{A}_3 \hat{A}_1^{n-1} \hat{A}_2 \hat{A}_4^{n-1} &= (A_2A_1)(A_1A_1)^{n-1} (A_1A_2) (A_2A_2)^{n-1}\\
&= A_2 A_1^{2n} A_2^{2n-1} = \begin{pmatrix}1&0\\ 0&4^{n}\end{pmatrix}\end{align*}
and
\begin{align*}\hat{B}_\jjj = \hat{B}_3 \hat{B}_1^{n-1} \hat{B}_2 \hat{B}_4^{n-1} &= (B_2B_1)(B_1B_1)^{n-1} (B_1B_2) (B_2B_2)^{n-1}\\
&= B_2 B_1^{2n} B_2^{2n-1} = \begin{pmatrix}1&0\\ 0&4^{n}\end{pmatrix}\end{align*}
so that $\hat\Phi_2(\jjj)= \max\{4^{n\alpha},4^{n\beta}\}$, but clearly we have $\hat\Phi(\jjj)=4^{n(\alpha+\beta)}$. It follows that \eqref{eq:bunflow} also cannot hold for $i=2$. We conclude that neither $\mu_1$ nor $\mu_2$ can be equal to $\iota_*\mu$, and since this exhausts the ergodic equilibrium states of $\hat\Phi$ the equilibrium state $\iota_*\mu$ cannot be ergodic, so $\mu$ cannot be ergodic with respect to $\sigma^2$ and hence is not totally ergodic as required. This completes the proof.

\emph{Remark.} If $\hat\Phi$ had a unique ergodic equilibrium state then it would have to be equal to either $\mu_1$ or $\mu_2$ and also to $\iota_*\mu$, which has been shown to be impossible, so $\hat \Phi$ cannot have a unique ergodic equilibrium state. By elimination the only possibility is that \emph{both} of  $\mu_1$ and $\mu_2$ are equilibrium states and that $\iota_*\mu$ is equal to a strict linear combination of these two mutually singular measures. On the other hand Theorem \ref{th:recoding} implies that $\mu$ must be equal to a \emph{balanced} linear combination of two mutually singular $\sigma^2$-invariant measures, and we conclude that necessarily $\iota_*\mu = \frac{1}{2}\mu_1+\frac{1}{2}\mu_2$. 

\section{Acknowledgments}

This research was partially supported by the Leverhulme Trust (Research Project Grant RPG-2016-194). This research grew from extensive discussions with Jairo Bochi on possible extensions of the article \cite{Pi20} and the author is indebted to him for numerous helpful conversations around this topic.

\bibliographystyle{acm}
\bibliography{bern-biblio}
\end{document}